\documentclass[10pt,a4paper,leqno]{amsart}

\usepackage[german,french,english]{babel}
\usepackage{amsmath,amsthm,amssymb,mathrsfs,nicefrac}
\usepackage{graphicx,subfig,color,ifthen,calc,comment}
\usepackage{mathptmx,courier}
\usepackage[scaled=0.95]{helvet}

\graphicspath{{./figures/}}

\specialcomment{Annotation}{\begin{annotation}}{\end{annotation}}
\specialcomment{Appendix}{}{}
\excludecomment{Annotation}

\usepackage[
            pdfauthor={Michael Eisermann}
            pdftitle={The Fundamental Theorem of Algebra made effective: an elementary real-algebraic proof via Sturm chains}
            pdfkeywords={Constructive and algorithmic aspects of the fundamental theorem of algebra, 
              real closed field, algebraic winding number, Cauchy index, Sturm--Cauchy root-finding algorithm.},
            pdfcreator={LaTeX with hyperref},
            pdfproducer={LaTeX with hyperref},
            pdfstartview=FitH,
            bookmarksopen=true,
            bookmarksnumbered,
            breaklinks=true, 
            colorlinks=true, 
            anchorcolor=black, 
            linkcolor=black, 
            urlcolor=black,
            citecolor=black, 
            pagecolor=black, 
            filecolor=black, 
            menucolor=black, 
            ]{hyperref}

\newcommand{\link}[1]{\href{http://#1}{\nolinkurl{#1}}}

\theoremstyle{plain}
\newtheorem{theorem}{Theorem}[section]
\newtheorem{lemma}[theorem]{Lemma}
\newtheorem{proposition}[theorem]{Proposition}
\newtheorem{corollary}[theorem]{Corollary}

\newtheorem*{Theorem}{Theorem}

\theoremstyle{definition}
\newtheorem{definition}[theorem]{Definition}

\theoremstyle{remark}
\newtheorem{remark}[theorem]{Remark}
\newtheorem{example}[theorem]{Example}

\newtheorem*{Question}{Question}
\newtheorem*{Example}{Example}

\numberwithin{equation}{section}
\newcommand{\sref}[1]{\textsection\ref{#1}}

\newtheoremstyle{anno}%
{0.5\baselineskip}
{0.5\baselineskip}
{\upshape\footnotesize}
{0pt}
{\bfseries\footnotesize}
{.}{0.5em}{}
\theoremstyle{anno}
\newtheorem{annotation}{Annotation}[section]

\newenvironment{mylist}[1][70pt]{%
  \begin{list}{}{
      \setlength{\labelwidth}{#1}
      \setlength{\labelsep}{5pt}
      \setlength{\leftmargin}{17pt+#1}
      }}%
  {\end{list}}

\newcommand{\NN}{\mathbb{N}}
\newcommand{\ZZ}{\mathbb{Z}}
\newcommand{\QQ}{\mathbb{Q}}
\newcommand{\RR}{\mathbb{R}}
\newcommand{\CC}{\mathbb{C}}

\newcommand{\R}{\mathbf{R}}
\newcommand{\C}{\mathbf{C}}
\newcommand{\K}{\mathbf{K}}
\newcommand{\F}{\mathbf{F}}
\newcommand{\X}{\mathscr{X}}
\newcommand{\Y}{\mathscr{Y}}
\newcommand{\Z}{\mathscr{Z}}
\renewcommand{\S}{\mathbf{S}}
\renewcommand{\P}{\mathbf{P}}

\newcommand{\abs}[1]{\lvert#1\rvert}
\newcommand{\sign}{\operatorname{sign}}
\newcommand{\half}{{\textstyle\frac{1}{2}}}
\newcommand{\Int}{\operatorname{Int}}
\newcommand{\dd}{\,\mathrm{d}}

\newcommand{\disc}{\operatorname{disc}}

\newcommand{\id}{\operatorname{id}}
\newcommand{\minus}{\smallsetminus}

\newcommand{\onto}{\mathrel{\makebox[1pt][l]{$\to$}{\to}}}
\newcommand{\isoto}{\mathrel{\xrightarrow{_\sim}}}
\newcommand{\card}{\operatorname{\#}}
\newcommand{\mcard}{\operatorname{\underset{mult}{\#}}}

\newcommand{\lc}{\operatorname{lc}}

\newcommand{\re}{\operatorname{re}}
\newcommand{\im}{\operatorname{im}}

\newcommand{\Ind}{\operatorname{Ind}}
\newcommand{\Sturm}{\operatorname{Sturm}}
\newcommand{\Routh}{\operatorname{Routh}}
\newcommand{\Wind}{w}
\newcommand{\cind}[2]{\Wind({#2}|{#1})}
\newcommand{\mind}{\overline{\Wind}}

\newcommand{\ie}[2]{\mathopen[ #1, #2 \mathclose[}
\newcommand{\ei}[2]{\mathopen] #1, #2 \mathclose]}
\newcommand{\ee}[2]{\mathopen] #1, #2 \mathclose[}

\begin{document} 


\title[The Fundamental Theorem of Algebra: a real-algebraic proof]
      {The Fundamental Theorem of Algebra made effective: 
        \\ an elementary real-algebraic proof via Sturm chains}
      
\author{Michael Eisermann}
\address{Institut f\"ur Geometrie und Topologie, Universit\"at Stuttgart, Germany}
\email{Michael.Eisermann@mathematik.uni-stuttgart.de}
\urladdr{www.igt.uni-stuttgart.de/eiserm}

\date{first version March 2008; this version compiled \today}







\maketitle

\begin{abstract}
  Sturm's theorem (1829/35) provides an elegant algorithm 
  to count and locate the real roots of any real polynomial.
  In his residue calculus (1831/37) 
  Cauchy extended Sturm's method to count and locate 
  the complex roots of any complex polynomial.
  For holomorphic functions Cauchy's index is based on contour integration,
  but in the special case of polynomials it can effectively be calculated 
  via Sturm chains using euclidean division as in the real case.
  In this way we provide an algebraic proof of Cauchy's theorem 
  for polynomials over any real closed field.
  As our main tool, we formalize Gauss' geometric notion 
  of winding number (1799) in the real-algebraic setting,
  from which we derive a real-algebraic proof of the Fundamental Theorem of Algebra.
  The proof is elementary inasmuch as it uses only 
  the intermediate value theorem and arithmetic of real polynomials.
  It can thus be formulated in the first-order language of real closed fields.
  Moreover, the proof is constructive and immediately 
  translates to an algebraic root-finding algorithm.
\end{abstract}

\begin{center}
  \footnotesize \it \selectlanguage{french} 
  L'alg\`ebre est g\'en\'ereuse, elle donne souvent plus qu'on lui
  demande. \rm (Jean le Rond d'Alembert)%
  \footnote{ \selectlanguage{english} 
    This quotation is folklore \cite[p.\,285]{Kasner:1905}.
    It corresponds, though not verbatim, to d'Alembert's article \textit{\'Equation} 
    in the \textit{Encyclop\'edie} (1751--1765, tome 5, p.\,850):
    \selectlanguage{french} 
    ``[L'alg\`ebre] r\'epond non seulement \`a ce qu'on lui demande, 
    mais encore \`a ce qu'on ne lui demandoit pas, et qu'on ne songeoit pas \`a lui demander.''
    The portraits of Gauss and Cauchy are taken from Wikimedia Commons,
    the portrait of Sturm is from Loria's biography \cite{Loria:1938}.}
\end{center}

\begin{figure}[h]
  \begin{minipage}{0.32\linewidth}\centering
    \includegraphics[width=0.9\linewidth]{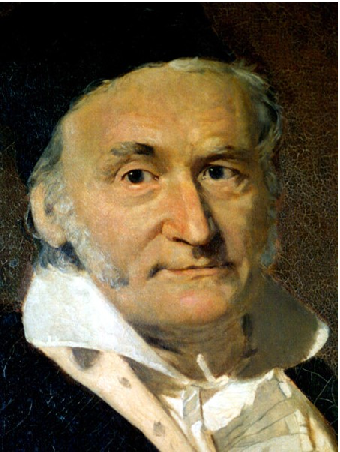}
    \\ \footnotesize Carl Friedrich Gau{\ss} \\ (1777--1855)
  \end{minipage}
  \hfill
  \begin{minipage}{0.32\linewidth}\centering
    \includegraphics[width=0.9\linewidth]{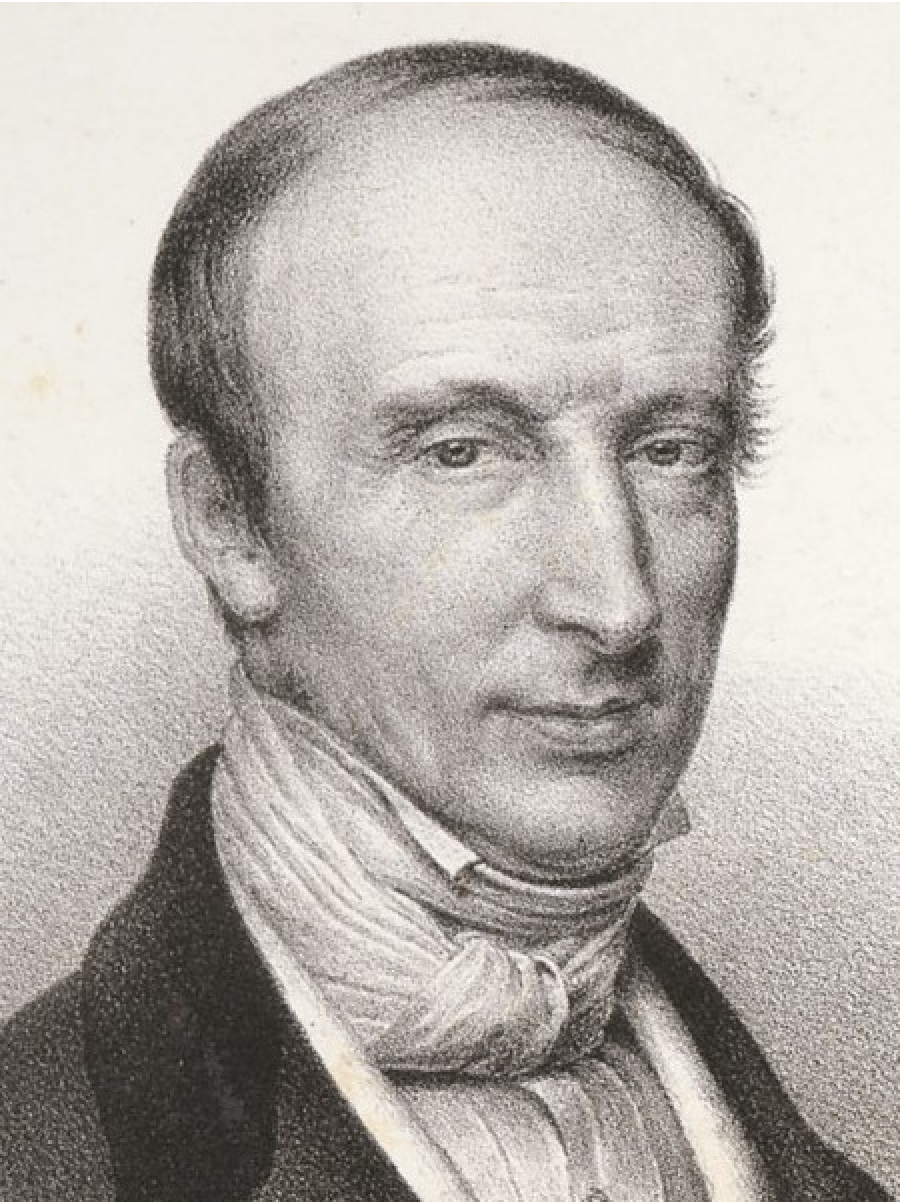}
    \\ \footnotesize Augustin Louis Cauchy \\ (1789--1857)
  \end{minipage}
  \hfill
  \begin{minipage}{0.32\linewidth}\centering
    \includegraphics[width=0.9\linewidth]{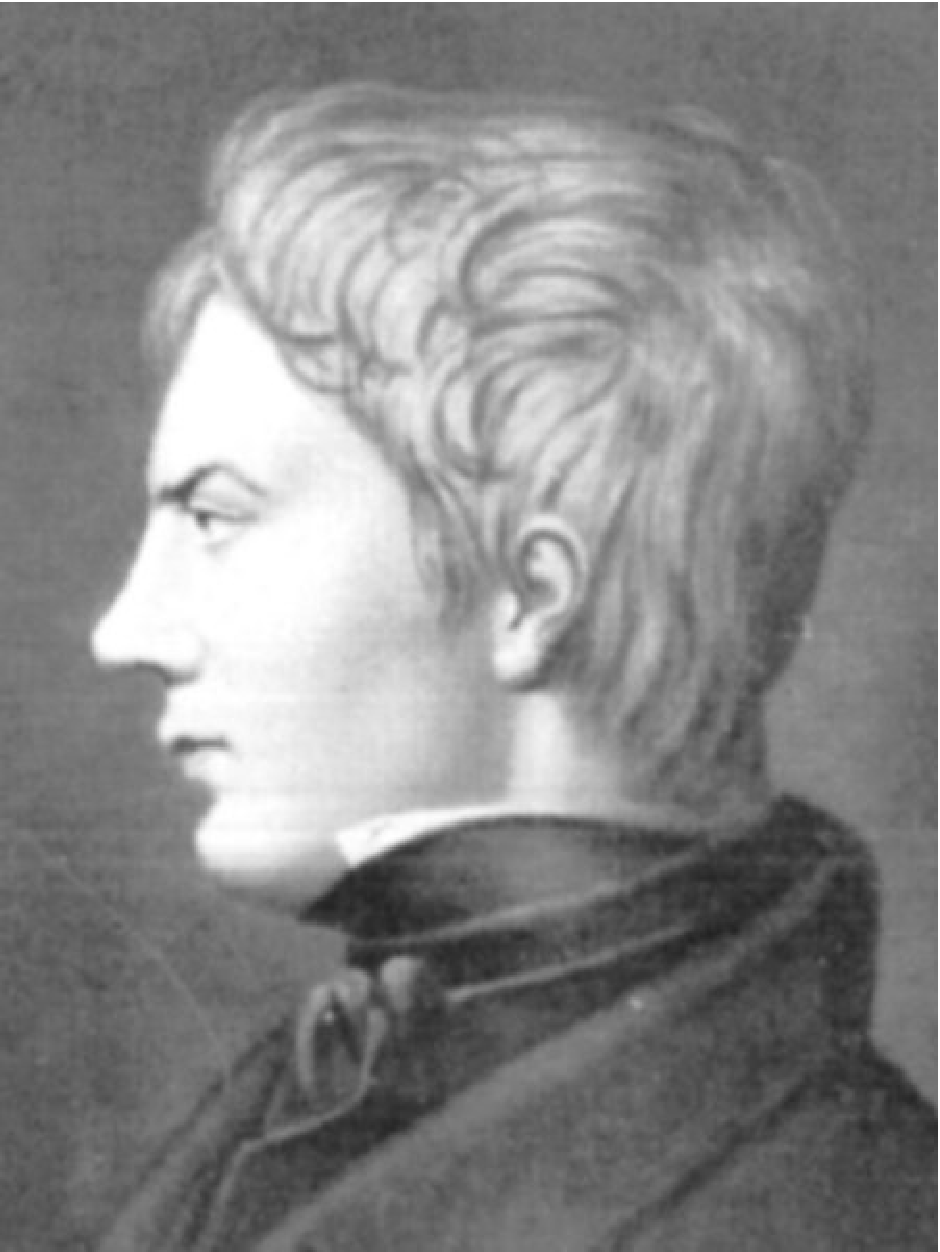}
    \\ \footnotesize Charles-Fran\c{c}ois Sturm \\ (1803--1855)
  \end{minipage}
\end{figure}











\section{Introduction and statement of results} \label{sec:Introduction}

\subsection{Historical origins} \label{sub:HistoricalOrigins}

Sturm's theorem \cite{Sturm:1829,Sturm:1835}, 
announced in 1829 and published in 1835,
provides an elegant and ingeniously simple algorithm 
to determine for each real polynomial $P \in \RR[X]$ the number 
of its real roots in any given interval $[x_0,x_1] \subset \RR$. 
Sturm's breakthrough solved an outstanding problem 
of his time and earned him instant fame.

In his residue calculus, 
outlined in 1831 and fully developed in 1837, 
Cauchy \cite{Cauchy:1831,Cauchy:1837} extended Sturm's method 
to determine for each complex polynomial $F \in \CC[Z]$ 
the number of its complex roots in a given domain, 
say in any rectangle of the form $[x_0,x_1] \times [y_0,y_1] \subset \CC$, 
where we identify $\CC$ with $\RR^2$ in the usual way.
For holomorphic functions Cauchy's index is based on contour integration,
but in the special case of polynomials it can effectively be calculated 
via Sturm chains using euclidean division as in the real case.

Combining Sturm's real algorithm and Cauchy's complex approach, we provide 
an algebraic proof of Cauchy's theorem for polynomials over any real closed field.
As our main tool, we formalize Gauss' geometric notion 
of winding number in real-algebraic language.
This leads to a real-algebraic proof of the Fundamental Theorem of Algebra,
assuring that every nonconstant complex polynomial has at least one complex zero.
Since zeros split off as linear factors, this is equivalent 
to the following extensive formulation.

\begin{theorem}[Fundamental Theorem of Algebra, existence only]
  For every polynomial \[ F = Z^n + c_1 Z^{n-1} + \dots + c_{n-1} Z + c_n \]
  with complex coefficients $c_1,\dots,c_{n-1},c_n \in \CC$
  there exist $z_1,z_2,\dots,z_n \in \CC$ 
  such that \[ F = (Z-z_1)(Z-z_2)\cdots(Z-z_n). \]
\end{theorem}

Numerous proofs of this important theorem have been published over 
the last two centuries.  According to the tools used, they can 
be grouped into three families (\sref{sec:HistoricalRemarks}):
\begin{enumerate}
\item \label{strategy:analysis}
  Analysis, using compactness, integration, transcendental functions, etc.;
\item\label{strategy:algebra}
  Algebra, using polynomials and the intermediate value theorem;
\item \label{strategy:topology}
  Algebraic topology, using some form of the winding number.
\end{enumerate}

There are proofs for every taste and each has its merits.
From a more ambitious, constructive viewpoint, however,
a mere existence proof only ``announces the presence of a treasure, 
without divulging its location'', as Hermann Weyl put it.
``It is not the existence theorem that is valuable,
but the construction carried out in its proof.''%
\footnote{ \selectlanguage{german}
  ``Bezeichne ich Erkenntnis als einen wertvollen Schatz,
  so ist das Urteilsabstrakt ein Papier, welches das Vorhandensein 
  eines Schatzes anzeigt, ohne jedoch zu verraten, an welchem Ort.''
  \cite[p.\,54]{Weyl:1921}
  ``Nicht das Existenztheorem ist das Wertvolle, 
  sondern die im Beweise gef\"uhrte Konstruktion.''
  \cite[p.\,55]{Weyl:1921} }

The real-algebraic approach presented here is situated 
between \eqref{strategy:algebra} and \eqref{strategy:topology}.
It combines algebraic computation (Cauchy's index and Sturm's algorithm)
with geometric reasoning (Gauss' notion of winding number) 
and therefore enjoys some remarkable features.
\begin{itemize}
\item It uses only the intermediate value theorem and arithmetic of real polynomials.
\item It is elementary, in the colloquial as well as the formal sense of first-order logic.
\item All arguments and constructions hold verbatim over every real closed field.
\item The proof is constructive and immediately translates to a root-finding algorithm.
\item The algorithm is easy to implement, and reasonably efficient in moderate degree.
\item It can be formalized to a computer-verifiable proof (of theorem \emph{and} algorithm).
\end{itemize}

The logical structure of such a proof was already outlined 
by Sturm \cite{Sturm:1836} in 1836, but his article lacks 
the elegance and perfection of his famous 1835 m\'emoire.
This may explain why his sketch found little resonance, was not further 
worked out, and became forgotten by the end of the 19th century. 
The aim of the present article is to save the real-algebraic proof 
from oblivion and to develop Sturm's idea in due rigour.
The presentation is intended for non-experts and thus 
contains much introductory and expository material.


\subsection{The algebraic winding number} 

Our arguments work over every ordered field $\R$ 
that satisfies the intermediate value property for polynomials,
i.e., a \emph{real closed field} (\sref{sec:RealClosedFields}).
We choose this starting point as the axiomatic 
foundation of Sturm's theorem (\sref{sec:Sturm}).
We then deduce that the field $\C = \R[i]$ with $i^2 = -1$ is algebraically closed, 
which was first proven by Artin and Schreier \cite{ArtinSchreier:1926,ArtinSchreier:1927}.
Moreover, we construct the algebraic winding number and establish 
an algorithm to locate the zeros of any given polynomial $F \in \C[Z]^*$.
(Here for every ring $A$, we denote by $A^* = A \minus \{0\}$ the set of its nonzero elements.)

The geometric idea is very intuitive: the winding number $\Wind(\gamma)$ 
counts the number of turns that a loop $\gamma \colon [0,1] \to \C^*$ 
performs around $0$. 
Theorem \ref{thm:AlgebraicWindingNumber} turns the geometric idea into a rigorous 
algebraic construction and provides an effective computation. 

\begin{figure}[ht]
  \centering
  \includegraphics[height=25ex]{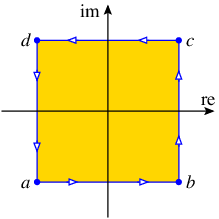} 
  \quad
  \includegraphics[height=25ex]{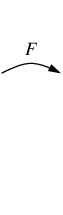} 
  \quad
  \includegraphics[height=25ex]{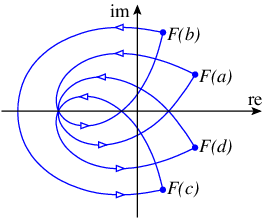} 
  \caption{The winding number $\cind{\partial\Gamma}{F}$ of a polynomial $F \in \C[Z]$ along the boundary
    of a rectangle $\Gamma \subset \C$.  In this example $\cind{\partial\Gamma}{F} = 2$.}
  \label{fig:WindingNumber}
\end{figure}

In order to work algebraically, a \emph{loop} $\gamma$ 
will be understood to be a piecewise polynomial map from the interval 
$[0,1] = \{\, x \in \R \mid 0 \le x \le 1 \,\}$ to $\C^*$ 
such that $\gamma(0) = \gamma(1)$; see \sref{sub:AlgebraicWindingNumber}.  
Likewise, a \emph{homotopy} between loops will be required to be 
piecewise polynomial, as explained in \sref{sub:HomotopyInvariance}.  
We can now formulate our main result.

\begin{theorem}[algebraic winding number] \label{thm:AlgebraicWindingNumber}
  Consider an ordered field $\R$ and its extension $\C = \R[i]$ where $i^2 = -1$.
  Let $\Omega$ be the set of piecewise polynomial loops $\gamma \colon [0,1] \to \C^*$.
  We define the \emph{algebraic winding number} $\Wind \colon \Omega \to \ZZ$ 
  by the following algebraic property:
  \begin{enumerate}
    \setcounter{enumi}{-1}
    \renewcommand{\theenumi}{W\arabic{enumi}}
  \item \label{property:AlgebraicComputation} Computation:
    $\Wind(\gamma)$ equals half the Cauchy index of $\frac{\re\gamma}{\im\gamma}$,
    recalled in \sref{sec:Sturm}, and can thus be calculated by Sturm's algorithm
    via iterated euclidean division. 
  \end{enumerate}
  If $\R$ is real closed, then $\Wind$ enjoys the following geometric properties:
  \begin{enumerate}
    \renewcommand{\theenumi}{W\arabic{enumi}}
  \item \label{property:Normalization} Normalization:
    Let $\Gamma \subset \C$ be a rectangle of the form $\Gamma = [x_0,x_1] \times [y_0,y_1]$.
    If $\gamma$ parametrizes the boundary $\partial\Gamma \subset \C^*$,
    positively oriented as in Figure \ref{fig:WindingNumber} (left), then
    \[
    \Wind( \gamma ) = \begin{cases} 
      1 & \text{if $0 \in \Int\Gamma$,} \\
      0 & \text{if $0 \in \C\minus\Gamma$.} 
    \end{cases}
    \]
  \item \label{property:Multiplicativity} Multiplicativity:
    For all $\gamma_1,\gamma_2 \in \Omega$ we have 
    \[
    \Wind( \gamma_1 \cdot \gamma_2 ) = \Wind( \gamma_1) + \Wind( \gamma_2 ) .
    \]
  \item \label{property:HomotopyInvariance} Homotopy invariance: 
    For all $\gamma_0,\gamma_1 \in \Omega$ we have
    \[
    \Wind( \gamma_0 ) = \Wind( \gamma_1 )
    \quad \text{whenever $\gamma_0$ and $\gamma_1$ are homotopic in $\C^*$.}
    \]
  \end{enumerate}
  Conversely, if over some ordered field $\R$ there exists a map 
  $\Wind \colon \Omega \to \ZZ$ satisfying properties \eqref{property:Normalization},
  \eqref{property:Multiplicativity}, \eqref{property:HomotopyInvariance},
  then $\R$ is real closed and $\Wind$ can be calculated as in \eqref{property:AlgebraicComputation}
\end{theorem}

\begin{remark} \label{rem:WindingNumberConstruction}
  Since polynomials form the simplest function algebra and can immediately be used for computations,
  Theorem \ref{thm:AlgebraicWindingNumber} has both practical and theoretical relevance.
  Over the real numbers $\RR$, the Stone-Weierstrass theorem can be used to extend
  the winding number 
  to continuous loops and homotopies, such that the geometric properties 
  \eqref{property:Normalization}, \eqref{property:Multiplicativity}, 
  \eqref{property:HomotopyInvariance} continue to hold.
  Several alternative constructions over $\RR$ lead to this result:
  \begin{enumerate}
  \item \label{alt:FundamentalGroup}
    Fundamental group, $\Wind \colon \pi_1(\CC^*,1) \isoto \ZZ$ via the Seifert--van\,Kampen theorem,
  \item \label{alt:CoveringTheory}
    Covering theory, $\exp \colon \CC \onto \CC^*$ with monodromy $\Wind \colon \pi_1(\CC^*,1) \isoto \ZZ$,
  \item \label{alt:Homology}
    Homology, $\Wind \colon H_1(\CC^*) \isoto \ZZ$ via the Eilenberg--Steenrod axioms,
  \item \label{alt:ComplexIntegration}
    Complex analysis, analytic winding number $\Wind(\gamma) = \frac{1}{2 \pi i}\int_\gamma \frac{\dd z}{z}$ via integration.
  \end{enumerate}

  Each of these approaches relies on some characteristic property 
  of the field $\RR$ of real numbers, such as metric completeness or some equivalent,
  and therefore does not extend to any other real closed field.
  In this article we develop an independent algebraic proof using only polynomial arithmetic,
  avoiding compactness, integrals, covering spaces, etc.

  We remark that constructions \eqref{alt:FundamentalGroup} 
  and \eqref{alt:CoveringTheory} are dual via Galois correspondence,
  while their abelian counterparts \eqref{alt:Homology} and 
  \eqref{alt:ComplexIntegration} are dual via the homology-cohomology 
  pairing. 
  The real-algebraic approach 
  appears to be self-dual, as expressed in Theorem \ref{thm:AlgebraicWindingNumber} 
  by the equivalence of the algebraic computation \eqref{property:AlgebraicComputation} 
  with the geometric properties \eqref{property:Normalization},
  \eqref{property:Multiplicativity}, and \eqref{property:HomotopyInvariance}.
  This dual nature conjugates 
  real-algebraic geometry and effective algebraic topology.
\end{remark}  

\begin{remark} \label{rem:WindingNumberSubtleties}
  The algebraic winding number turns out to be 
  slightly more general than stated in the theorem.
  The algebraic definition \eqref{property:AlgebraicComputation} 
  of $\Wind(\gamma)$ also applies to loops $\gamma$ that pass through $0$.  
  Normalization \eqref{property:Normalization} extends to 
  $\Wind(\gamma) = \nicefrac{1}{2}$ if $0$ lies in an edge of $\Gamma$,
  and $\Wind(\gamma) = \nicefrac{1}{4}$ if $0$ is one of the vertices of $\Gamma$.
  Multiplicativity \eqref{property:Multiplicativity} continues to hold 
  provided that $0$ is not a vertex of $\gamma_1$ or $\gamma_2$.
  Homotopy invariance \eqref{property:HomotopyInvariance} applies only to loops in $\C^*$. 
\end{remark}

\subsection{Counting complex roots} \label{sub:FTAbyWindingNumber}

For the rest of this introduction, $\R$ denotes a real closed field 
and $\C = \R[i]$ its complex extension. 
From Theorem \ref{thm:AlgebraicWindingNumber} we can deduce 
the Fundamental Theorem of Algebra using the geometric properties \eqref{property:Normalization}, 
\eqref{property:Multiplicativity}, \eqref{property:HomotopyInvariance} as follows.

As the first step (\sref{sec:Cauchy}) we obtain the following algebraic version of Cauchy's theorem.
We write $\cind{\partial\Gamma}{F}$ as a short-hand for $\Wind(F \circ \gamma)$
where $\gamma$ parametrizes $\partial\Gamma$ as in Figure \ref{fig:WindingNumber}.

\begin{theorem}[local winding number] \label{thm:RootCounting}
  If $F \in \C[Z]$ does not vanish at any of the four vertices 
  of the rectangle $\Gamma \subset \C$, then the algebraic winding number 
  $\cind{\partial\Gamma}{F}$ equals the number of roots of $F$ in $\Gamma$.
  Here each root in the interior of $\Gamma$ is counted with its multiplicity,
  whereas each root in an edge of $\Gamma$ is counted with half its multiplicity.
\end{theorem}

To prove this, consider $F = (Z-z_1) \cdots (Z-z_m) G$ with 
$z_1,\dots,z_m \in \Gamma$ such that $G$ has no zeros in $\Gamma$.  
For $a \in \Gamma$ the homotopy $G_t = G( a + t(Z-a) )$ deforms $G_1 = G$
to $G_0 = G(a)$, whence homotopy invariance \eqref{property:HomotopyInvariance} 
implies that $\cind{\partial\Gamma}{G_1} = \cind{\partial\Gamma}{G_0} = 0$.
The theorem then follows from multiplicativity \eqref{property:Multiplicativity}
and normalization \eqref{property:Normalization} as in Remark \ref{rem:WindingNumberSubtleties}.

\begin{example} \label{exm:SmallPolynomial}
  Figure \ref{fig:WindingNumber} displays the situation for 
  $F = Z^5 - 5 Z^4 - 2 Z^3 - 2 Z^2 - 3 Z - 12$ and $\Gamma = [-1,+1]^2$.
  Here the winding number is $\cind{\partial\Gamma}{F} = 2$.  
  This is in accordance with the approximate location of zeros:
  $\Gamma$ contains $z_{1,2} \approx -0.9 \pm 0.76i$ whereas 
  $z_{3,4} \approx 0.67 \pm 1.06i$ and $z_5 \approx 5.46$ lie outside of $\Gamma$.
\end{example}

The hypothesis that $F$ does not vanish at any of the vertices of $\Gamma$ 
is very mild and easy to check in every concrete application.
Unlike Cauchy's integral formula $\Wind(\gamma) = \frac{1}{2 \pi i}\int_\gamma \frac{\dd z}{z}$, 
the algebraic winding number behaves well if zeros lie on (or close to) the boundary,
and the uniform treatment of all configurations of roots simplifies 
theoretical arguments and practical implementations alike.
This is yet another manifestation of the oft-quoted wisdom of d'Alembert 
that \emph{Algebra is generous, she often gives more than we ask of her}.

As the second step (\sref{sec:FTA}) we formalize Gauss' geometric argument (1799)
saying that $F \approx Z^n$ outside of a sufficiently big rectangle $\Gamma \subset \C$,
whence $F | \partial\Gamma$ has winding number $n$.

\begin{theorem}[global winding number] \label{thm:GlobalWindingNumber}
  For each polynomial $F = Z^n + c_1 Z^{n-1} + \dots + c_n$ in $\C[Z]$,
  we define its Cauchy radius to be $\rho_F := 1 + \max\{\abs{c_1},\dots,\abs{c_n}\}$.
  Then $F$ satisfies $\cind{\partial\Gamma}{F} = n$ on every rectangle $\Gamma$ 
  containing the Cauchy disk $B(\rho_F) = \{\, z \in \C \mid \abs{z} < \rho_F \,\}$.
\end{theorem}

The proof uses the homotopy $F_t = Z^n + t(c_1 Z^{n-1} + \dots + c_n)$ 
to deform $F_1 = F$ to $F_0 = Z^n$.  All zeros of $F_t$ lie in $B(\rho_F)$.
The hypothesis $\Gamma \supset B(\rho_F)$ ensures that $F_t$ has no zeros on $\partial\Gamma$, 
so homotopy invariance \eqref{property:HomotopyInvariance} allows us to conclude 
that $\cind{\partial\Gamma}{F_1} = \cind{\partial\Gamma}{F_0} = n$.

Theorems \ref{thm:RootCounting} and \ref{thm:GlobalWindingNumber} 
imply that $\C$ is algebraically closed.  Each polynomial $F \in \C[Z]$ of degree $n$ 
has $n$ roots in $\C$, more precisely in the square $\Gamma = [-\rho_F,\rho_F]^2 \subset \C$. 
(The latter is only a coarse estimate and can be improved for practical purposes; see Remark \ref{rem:CauchyPolynomial}.)


\subsection{The Fundamental Theorem of Algebra made effective}

The winding number proves more than mere existence of roots:
it also establishes a root-finding algorithm (\sref{sub:GlobalRootFinding}).
Here we have to assume that the ordered field $\R$ is archimedean, which amounts to $\R \subset \RR$.


\begin{theorem}[Fundamental Theorem of Algebra, effective version] \label{thm:RootLocation}
  For every complex polynomial $F = Z^n + c_1 Z^{n-1} + \dots + c_n$ in $\CC[Z]$
  there exist complex roots $z_1,\dots,z_n \in \CC$ such that $F = (Z-z_1)\cdots(Z-z_n)$
  and the algebraic winding number provides an algorithm to locate them.
  Starting from some rectangle containing all $n$ roots, as in Theorem \ref{thm:GlobalWindingNumber}, 
  we can subdivide and keep only those rectangles that 
  actually contain roots, using Theorem \ref{thm:RootCounting}.
  All computations can be carried out using Sturm chains according to 
  Theorem \ref{thm:AlgebraicWindingNumber}. 
  By iterated bisection we can thus approximate all roots to any desired precision.
\end{theorem}

\begin{remark}[computability]
  In the real-algebraic setting of this article 
  we consider the field operations $(a,b) \mapsto a+b$, 
  $a \mapsto -a$, $(a,b) \mapsto a \cdot b$, $a \mapsto a^{-1}$ 
  and the comparisons $a = b$, $a < b$ as primitive operations.
  Over the real numbers $\RR$, this point of view was advanced 
  by Blum--Cucker--Shub--Smale \cite{SmaleEtAl:1998} by 
  postulating a hypothetical \emph{real number machine}.

  In order to implement the required real-algebraic operations on a Turing machine, 
  however, a more careful analysis is necessary (\sref{sub:TuringComputability}).
  Given $F = c_0 Z^n +  c_1 Z^{n-1} + \dots + c_n$  we have to assume that 
  the operations of the ordered field $\QQ(\re(c_0),\im(c_0),\dots,\re(c_n),\im(c_n))$ 
  are computable in the Turing sense (\sref{sub:GlobalRootFinding}).
  This is the case for the field $\QQ$ of rational numbers, for example, 
  or every real-algebraic number field $\QQ(\alpha) \subset \RR$.
\end{remark}

\begin{remark}[complexity]
  On a Turing machine we can compare time requirements by measuring bit-complexity.
  The above Sturm--Cauchy method requires $\tilde{O}( n^4 b^2 )$ bit-operations
  to approximate all $n$ roots to a precision of $b$ bits (\sref{sub:SubresultantAlgorithm}).
  Further improvement is necessary to reach the nearly optimal bit-complexity $\tilde{O}( n^3 b )$
  of Sch\"onhage \cite{Schoenhage:1986} (\sref{sub:AlgorithmicImprovements}).

  Nevertheless, the Sturm--Cauchy method can be useful in hybrid algorithms, 
  in order to verify numerical approximations
  and to improve them as necessary \cite{Rump:2003}.
  Once sufficient approximations of the roots have been obtained,
  one can switch to Newton's method, which converges much faster 
  but vitally depends on good starting values (\sref{sub:NewtonMethod}).
\end{remark}

\begin{Annotation}
  \textbf{(Sturm's forgotten proof)}
  Attracted by the aforementioned features, I worked out 
  the real-algebraic approach for a computer algebra course 
  at the University of Grenoble in 2008.  The idea seems natural, 
  and so I was surprised not to find any such proof in the literature.
  Retracing its history (\sref{sec:HistoricalRemarks}), 
  I was even more surprised when I finally unearthed the key ideas 
  in the works of Cauchy and Sturm 
  (\sref{sub:CauchySturmLiouville}).  Why have they been lost?
  The real-algebraic proof is, of course, based on classical ideas.
  The geometric idea goes back to Gauss in 1799; Sturm's algebraic method 
  and Cauchy's analytic techniques have been developed in the 1830s.
  Since then they have evolved in very different directions:
  
  Sturm's theorem has become a cornerstone of real algebra.
  Cauchy's integral is the starting point of complex analysis.
  Their algebraic method for counting complex roots, however,
  has transited from algebra to applications, where its 
  conceptual and algorithmic simplicity are much appreciated.
  Algebra textbooks published since the end of the 19th century do not present it;
  by now it is known almost exclusively to specialists as as a tool for \emph{computation},
  for example in the Routh--Hurwitz theorem on the stability of motion. 
  After Sturm's outline of 1836, this algebraic tool seems to have 
  never been employed to \emph{prove} the existence of roots.
  
  In retrospect, the proof presented here is a fortunate rediscovery 
  of Sturm's algebraic vision (\sref{sub:SturmsAlgebraicVision}).
  The present article gives a modern, rigorous, and complete presentation, 
  which means to set up the right definitions and to provide elementary, real-algebraic proofs.
  
\end{Annotation}

\subsection{How this article is organized} \label{sub:ArticleOrganization}

Section \ref{sec:RealClosedFields} briefly recalls the notion 
of real closed fields, on which we build Sturm's theorem 
and the theory of Cauchy's index.

Section \ref{sec:Sturm} presents Sturm's theorem \cite{Sturm:1835}
counting real roots of real polynomials.
The only novelty is the extension to boundary points,
which is needed in Section \ref{sec:Cauchy}.

Section \ref{sec:Cauchy} proves Cauchy's theorem \cite{Cauchy:1837}
counting complex roots of complex polynomials,
by establishing multiplicativity \eqref{property:Multiplicativity} 
of the algebraic winding number.

Section \ref{sec:FTA} establishes homotopy invariance \eqref{property:HomotopyInvariance},
and proves the Fundamental Theorem of Algebra by Gauss' winding number argument.

Section \ref{sec:AlgorithmicAspects} discusses algorithmic aspects,
such as Turing computability, the efficient computation of Cauchy indices,
and the crossover to Newton's local method.

Section \ref{sec:HistoricalRemarks}, finally, provides historical comments 
in order to put the real-algebraic approach into a wider perspective. 



I have tried to keep the exposition elementary yet detailed.
I hope that the interest of the subject justifies the resulting length of this article.

\begin{Annotation}
  \textbf{(Why should we care for yet another proof?)}
  There are several lines of proof leading to the Fundamental Theorem of Algebra,
  and literally hundreds of variants have been published over the last 200 years
  (see \sref{sec:HistoricalRemarks}).  The motivations for the present work are threefold:
  \begin{itemize}

  \item
    First, on a philosophical level, it is satisfying 
    to minimize the hypotheses and the tools used in the proof, 
    and simultaneously maximize the conclusion.
    
    
  \item
    Second, when teaching mathematics, it is advantageous to have 
    different proofs to choose from, adapted to the course's level and context.
    
  
  \item
    Third, from a practical point of view, it is desirable 
    to have a constructive proof, even more so if it 
    directly translates to a practical algorithm.
    
    
  \end{itemize}
  
  In this annotated student version, some complementary remarks are included that will not appear in the published version.  
  They are set in small font, as this one, and numbered separately in order to ensure consistent references.

  \bigskip
  \makeatletter
  \begin{quote}
    \renewcommand{\contentsline}[4]{#2}
    \renewcommand{\tocsection}[3]{%
      \ifthenelse{\equal{#2}{}}{\unskip\ignorespaces}{\item[#2.] \bf#3\unskip\@addpunct.\rm }}
    \renewcommand{\tocsubsection}[3]{%
      \ifthenelse{\equal{#2}{}}{\unskip\ignorespaces}{#2.~#3\unskip\@addpunct. }}
    \renewcommand{\tocsubsubsection}[3]{\hspace{-3pt}\unskip\ignorespaces}
    \begin{center}\contentsnamefont\contentsname\end{center}
    \medskip
    \setTrue{toc}
    \@input{\jobname.toc}%
    \if@filesw
    \@xp\newwrite\csname tf@toc\endcsname
    \immediate\@xp\openout\csname tf@toc\endcsname \jobname.toc\relax
    \fi
    \global\@nobreakfalse \par
    \addvspace{32\p@\@plus14\p@}%
    \let\tableofcontents\relax
  \end{quote}
  \makeatother
  
\end{Annotation}


\section{Real closed fields} \label{sec:RealClosedFields}

This section sets the scene by recalling the notion of a real closed field, 
on which we build Sturm's theorem in \sref{sec:Sturm}, 
and also sketches its mathematical context.



\begin{Annotation} 
  \textbf{(Fields)} \label{anno:Field}
  We assume that the reader is familiar with the algebraic notion of a \emph{field}.
  In order to highlight the field axioms formulated in first-order logic, we recall that 
  a field $(\R,+,\cdot)$ is a set $\R$ equipped with two binary operations 
  $+ \colon \R \times \R \to \R$ and $\cdot \colon \R \times \R \to \R$
  satisfying the following three sets of axioms: 

  First, addition enjoys the following four properties,
  saying that $(\R,+)$ is an abelian group:
  \begin{mylist}[80pt]
  \item[(A1) associativity]    For all $a,b,c \in \R$ we have $(a+b)+c = a+(b+c)$.
  \item[(A2) commutativity]    For all $a,b \in \R$ we have $a+b = b+a$. 
  \item[(A3) neutral element]  There exists $0 \in \R$ such that for all $a \in \R$ we have $a+0 = a$.
  \item[(A4) opposite elements] For each $a \in \R$ there exists $b \in \R$ such that $a+b = 0$.
  \end{mylist}
  The neutral element $0 \in \R$ whose existence is required by axiom (A3) is unique by (A2).
  This ensures that axiom (A4) is unambiguous.  The opposite element of $a \in \R$ 
  required by axiom (A4) is unique and denoted by $-a$.

  Second, multiplication enjoys the following four properties,
  saying that $(\R^*,\cdot)$ is an abelian group:
  \begin{mylist}[80pt]
  \item[(M1) associativity]    For all $a,b,c \in \R$ we have $(a \cdot b) \cdot c = a \cdot (b \cdot c)$.
  \item[(M2) commutativity]    For all $a,b \in \R$ we have $a \cdot b = b \cdot a$.
  \item[(M3) neutral element]  There exists $1 \in \R$, $1 \ne 0$, such that for all $a \in \R$ we have $a \cdot 1 = a$.
  \item[(M4) inverse elements] For each $a \in \R$, $a \ne 0$, there exists $b \in \R$ such that $a \cdot b  = 1$.
  \end{mylist}
  The neutral element $1 \in \R$ whose existence is required by axiom (M3) is unique by (M2).
  This ensures that axiom (M4) is unambiguous.  The inverse element of $a \in \R$ 
  required by axiom (M4) is unique and denoted by $a^{-1}$.

  Third, multiplication is distributive over addition:
  \begin{mylist}[80pt]
  \item[(D) distributivity]    For all $a,b,c \in \R$ we have $a \cdot (b + c) = (a \cdot b) + (a \cdot c)$. 
  \end{mylist}
\end{Annotation}
    
\begin{Annotation} 
  \textbf{(Ordered fields)} \label{anno:OrderedField}
  An \emph{ordered field} is a field $\R$ with a distinguished subset 
  $\R_{>0} \subset \R$ of positive elements, denoted $x > 0$, that is
  compatible with the field operations in the following sense:
  \begin{mylist}[80pt]
  \item[(O1) trichotomy]       For each $x \in \R$ we have either $x > 0$ or $x = 0$ or $-x > 0$.
  \item[(O2) compatibility]    For all $x,y \in \R$ the conditions $x > 0$ and $y > 0$ imply $x+y > 0$ and $x y > 0$.
  \end{mylist}
  
  We define the ordering $x > y$ by $x-y > 0$.
  The weak ordering $x \ge y$ means $x > y$ or $x = y$.
  The inverse ordering $x < y$ is defined by $y > x$,
  and likewise $x \le y$ is defined by $y \ge x$.
  From the above axioms follow the usual properties; see 
  Jacobson \cite[\textsection 5.1]{Jacobson:1985:1989}, 
  Cohn \cite[\textsection 8.6]{Cohn:2003}, 
  or Lang \cite[\textsection XI.1]{Lang:2002}.
  Intervals in $\R$ will be denoted, as usual, by 
  \begin{align*}
    [a,b] & = \{\, x \in \R \mid a \le x \le b \,\} ,
    & \ei{a}{b} = \{\, x \in \R \mid a < x \le b \,\} , \\
    \ee{a}{b} & = \{\, x \in \R \mid a < x < b \,\} ,
    & \ie{a}{b} = \{\, x \in \R \mid a \le x < b \,\} .
  \end{align*}
  
  Every ordered field $\R$ inherits a natural topology generated by open intervals:
  a subset $U \subset \R$ is open if for each $x \in U$ there exists $\delta > 0$
  such that $\ee{x-\delta}{x+\delta} \subset U$.  We can thus apply the usual 
  notions of topological spaces and continuous functions.  Addition and 
  multiplication are continuous, and so are polynomial functions.

  For $x \in \R$ we define the absolute value 
  to be $\abs{x} := x$ if $x \ge 0$ and $\abs{x} := -x$ if $x \le 0$.
  We record the following properties, which hold for all $x,y \in \R$:
  \begin{enumerate}
  \item $\abs{x} \ge 0$, and $\abs{x} = 0$ if and only if $x=0$.
  \item $\abs{x + y} \le \abs{x} + \abs{y}$ for all $x,y \in \R$.
  \item $\abs{x \cdot y} = \abs{x} \cdot \abs{y}$ for all $x,y \in \R$.
  \end{enumerate}
  %

  For every $x \in \R$ we have $x^2 \ge 0$ with equality if and only if $x = 0$.
  The polynomial $X^2-a$ can thus have a root $x \in \R$ only for $a \ge 0$;
  if it has a root, then $X^2 - a = (X-x)(X+x)$ and among the two roots $\pm x$ 
  we can choose signs such that $x \ge 0$, denoted $\sqrt{a} := x$. This implies that $\sqrt{x^2} = \abs{x}$.  
  In \sref{sub:ComplexFields} we will extend the absolute value $\abs{\;} \colon \R \to \R_{\ge0}$ 
  to a norm $\abs{\;} \colon \C \to \R_{\ge0}$ on the complex field $\C = \R[i]$.
\end{Annotation}

\begin{Annotation} 
  \textbf{(Rings)} \label{anno:Ring}
  A \emph{ring} $(\R,+,\cdot)$ is only required 
  to satisfy axioms (A1-A4), (M1-M3), and (D) but not necessarily (M4).
  This is sometimes called a \emph{commutative ring with unit},
  for emphasis, but we will have no need for this distinction.
  A ring $\R$ is called \emph{integral} if for all $a,b \in \R^*$ we have $a b \in \R^*$.
  Every integral ring $\R$ can be embedded into a field; the smallest 
  such field is unique and thus called the \emph{field of fractions} of $\R$.
  If a ring ist ordered (in the sense of Annotation \ref{anno:OrderedField}),
  then it is integral and the ordering uniquely extends to its field of fractions.
  For example, the ring $\ZZ$ of integers thus yields the field $\QQ$ of rational numbers.
  In this article we will study the ring $\R[X]$ polynomials 
  over some ordered field $\R$, as explained below, which has 
  as field of fractions the field of rational functions $\R(X)$.
\end{Annotation}

\subsection{Real numbers} \label{sub:RealNumbers}

As usual we denote by $\RR$ the field of real numbers, that is, 
an ordered field $(\RR,+,\cdot,<)$ such that every nonempty 
bounded subset $A \subset \RR$ has a least upper bound in $\RR$. 
This is a very strong property, and in fact it characterizes $\RR$.

\begin{theorem} \label{thm:CharacterizingTheReals}
  Let $\R$ be an ordered field, with the order-topology generated by the open intervals. 
  Then the following conditions are equivalent:
  \begin{enumerate}
  \item The ordered set $(\R,<)$ satisfies the least upper bound property.
  \item Each interval $[a,b] \subset \R$ is compact as a topological space.
  \item Each interval $[a,b] \subset \R$ is connected as a topological space.
  \item The intermediate value property holds for all continuous functions $f \colon \R \to \R$.
  \end{enumerate}
  
  Any two ordered fields satisfying these properties are isomorphic 
  by a unique field isomorphism, and this isomorphism preserves order.
  Any construction of the real numbers shows that one such field exists.
  \qed
\end{theorem}

\begin{Annotation}
  \textbf{(Sketch of proof)}
  Existence and uniqueness of the field $\RR$ of real numbers 
  form the foundation of any analysis course.
  Many analysis books prove $(1) \Rightarrow (2) \Rightarrow (4)$, while 
  $(3) \Leftrightarrow (4)$ is essentially the definition of connectedness.
  Here we only show $(4) \Rightarrow (1)$, in the form $\neg(1) \Rightarrow \neg(4)$.
  
  Let $A \subset \R$ be nonempty and bounded above.
  Define $f \colon \R \to \{\pm1\}$ by
  $f(x) = 1$ if $a \le x$ for all $a \in A$,
  and $f(x) = -1$ if $x < a$ for some $a \in A$.
  In other words, we have $f(x) = 1$ 
  if and only if $x$ is an upper bound of $A$.
  If $f$ is discontinuous at $x$, then $f(x) = +1$
  but $f(y) = -1$ for all $y < x$, whence $x = \sup A$.
  If $A$ does not have a least upper bound in $\R$, 
  then $f$ is continuous but does not satisfy 
  the intermediate value property.
\end{Annotation}

\subsection{Real closed fields} \label{sub:RealClosedFields}

The field $\RR$ of real numbers provides the foundation of analysis.
In the present article it appears as the most prominent 
example of the much wider class of real closed fields.
The reader who wishes to concentrate on the classical case 
may skip the rest of this section and assume $\R = \RR$ throughout.

\begin{Annotation}
  \textbf{(Polynomials)}
  In the sequel we shall assume that the reader is familiar 
  with the polynomial ring $\K[X]$ over some ground ring $\K$,
  see Jacobson \cite[\textsection 2.9--2.12]{Jacobson:1985:1989} 
  or Lang \cite[\textsection II.2, \textsection IV.1]{Lang:2002}.
  We briefly recall some notation.
  Let $\K$ be a ring, that is, satisfying axioms (A1-A4), (M1-M3), and (D) 
  of Annotation \ref{anno:OrderedField}, but not necessarily (M4).
  There exists a ring $\K[X]$ characterized by the following two properties: 
  First, $\K[X]$ contains $\K$ as a subring and $X$ as an element.
  Second, every nonzero element $P \in \K[X]$ can be uniquely written as
  \[
  P = c_0 + c_1 X + \dots + c_n X^n  \quad\text{where}\quad 
  n \in \NN \text{ and } c_0,c_1,\dots,c_n \in \K, c_n \ne 0 .
  \]
  
  In this situation $\K[X]$ is called the \emph{ring of polynomials} over $\K$ in the variable $X$, 
  and each element $P \in \K[X]$ is called a \emph{polynomial} over $\K$ in $X$.
  In the above notation we call $\deg P := n$ the \emph{degree}
  and $\lc P := c_n$ the \emph{leading coefficient} of $P$.
  The zero polynomial is special: we set $\deg 0 := -\infty$ and $\lc 0 := 0$.
\end{Annotation}

\begin{Annotation}
  \textbf{(Polynomial functions)}
  The ring $\K[X]$ has the following universal property:
  for every ring $\K'$ containing $\K$ as a subring
  and every element $x \in \K'$ there exists a unique
  ring homomorphism $\Phi \colon \K[X] \to \K'$
  such that $\Phi|_\K = \id_\K$ and $\Phi(X) = x$.
  Explicitly, $\Phi$ sends 
  $P = c_0 + c_1 X + \dots + c_n X^n$ to 
  $P(x) = c_0 + c_1 x + \dots + c_n x^n$.
  In particular, each polynomial $P \in \K[X]$ defines 
  a polynomial function $f_P \colon \K \to \K$, $x \mapsto P(x)$.
  If $\K$ is an infinite integral ring, for example an ordered ring or field,
  then the map $P \mapsto f_P$ is injective, and we can thus identify
  each polynomial $P \in \K[X]$ with the associated function $f_P \colon \K \to \K$.
  Traditionally equations have \emph{roots} and functions have \emph{zeros}.
  In this article we use both words ``roots'' and ``zeros'' synonymously. 
\end{Annotation}

\begin{definition} \label{def:RealClosedField}
  An ordered field $(\R,+,\cdot,<)$ is \emph{real closed}
  if it satisfies the intermediate value property for polynomials: 
  whenever $P \in \R[X]$ satisfies $P(a) P(b) < 0$ for some $a < b$ in $\R$, 
  then there exists $x \in \R$ with $a < x < b$ such that $P(x) = 0$.
\end{definition}

\begin{example}
  The field $\RR$ of real numbers is real closed 
  by Theorem \ref{thm:CharacterizingTheReals} above.  
  The field $\QQ$ of rational numbers is not real closed, 
  as shown by the example $P = X^2-2$ on $[1,2]$.
  The algebraic closure $\QQ^c$ of $\QQ$ in $\RR$ is a real closed field.
  In fact, $\QQ^c$ is the smallest real closed field, in the sense
  that $\QQ^c$ is contained in any real closed field.
  Notice that $\QQ^c$ is much smaller than $\RR$, 
  in fact $\QQ^c$ is countable whereas $\RR$ is uncountable.
\end{example}

The theory of real closed fields originated in the work 
of Artin and Schreier \cite{ArtinSchreier:1926,ArtinSchreier:1927} 
in the 1920s, culminating in Artin's solution \cite{Artin:1927} of Hilbert's 17th problem.
Excellent textbook references include Jacobson \cite[chap.\,I.5 and II.11]{Jacobson:1985:1989} 
and Bochnak--Coste--Roy \cite[chap.\,1 and 6]{BochnakCosteRoy:1998}.
For the present article, Definition \ref{def:RealClosedField} above is 
the natural starting point because it captures the essential geometric feature.
It deviates from the algebraic definition of Artin--Schreier \cite{ArtinSchreier:1926},
saying that an ordered field is real closed if no proper algebraic extension can be ordered.
For a proof of their equivalence see \cite[Prop.\,8.8.9]{Cohn:2003} 
or \cite[\textsection 1.2]{BochnakCosteRoy:1998}.

\begin{remark}
  In a real closed field $\R$ every positive element has a square root,
  and so the ordering on $\R$ can be characterized in algebraic terms:
  For every $a \in \R$ we have $a \ge 0$ if and only if there exists $b \in \R$ such that $b^2 = a$.
  In particular, if a field is real closed, then it admits precisely 
  one ordering that is compatible with the field structure.
\end{remark}

Every archimedean ordered field can be embedded 
into $\RR$; see \cite[\textsection 8.7]{Cohn:2003}. 
The field $\RR(X)$ of rational functions can be ordered 
(in many different ways; see \cite[\textsection 1.1]{BochnakCosteRoy:1998})
but does not embed into $\RR$.  Nevertheless 
it can be embedded into its real closure.

\begin{theorem}[Artin--Schreier {\cite[Satz 8]{ArtinSchreier:1926}}] \label{thm:RealClosure}
  Every ordered field $\K$ admits a real closure, i.e., a real closed field 
  that is algebraic over $\K$ and whose unique ordering extends that of $\K$.
  Any two real closures of $\K$ are isomorphic via a unique isomorphism fixing $\K$.
  \qed
\end{theorem}

The real closure is thus completely rigid, in contrast to the algebraic closure.


\begin{remark} \label{rem:ArtinSchreierTheorem}
  Artin and Schreier \cite[Satz 3]{ArtinSchreier:1926} proved that 
  if a field $\R$ is real closed, then $\C = \R[i]$ is algebraically closed,
  recasting the classical algebraic proof of the Fundamental Theorem of Algebra (\sref{ssub:Algebra}).
  Conversely \cite{ArtinSchreier:1927}, if a field $\C$ is algebraically closed 
  and contains a subfield $\R$ such that $1 < \dim_\R(\C) < \infty$, 
  then $\R$ is real closed and $\C = \R[i]$.  
\end{remark}

\begin{Annotation}
  \textbf{(Finiteness conditions)}
  In the sequel we will not appeal to the least upper bound property, nor compactness nor connectedness.
  The intermediate value property for polynomials is a sufficiently strong hypothesis. 
  In order to avoid compactness, a sufficient finiteness condition will be the fact 
  that a polynomial of degree $n$ over a field $\K$ can have at most $n$ roots in $\K$.

  In general $P$ can have \emph{less} than $n$ roots, of course, 
  as illustrated by the polynomial $X^2+1$ over $\RR$.
  The fact that $P$ cannot have \emph{more} than $n$ roots 
  relies on commutativity (M2) and invertibility (M4). 
  For example $X^2-1$ has four roots in the nonintegral ring 
  $\ZZ/8\ZZ$ of integers modulo $8$, namely $\pm1$ and $\pm3$.
  On the other hand, $X^2+1$ has infinitely many roots 
  in the skew field $\mathbb{H} = \RR + \RR i + \RR j + \RR k$ 
  of Hamilton's quaternions \cite[chap.\,7]{EbbinghausEtAl:1991}, 
  namely every combination $a i + b j + c k$ with $a,b,c \in \RR$ such that $a^2+b^2+c^2 = 1$.  
  The limitation on the number of roots makes the theory of fields very special.  
  We will repeatedly use it as a crucial finiteness condition.
\end{Annotation}
  

\subsection{Elementary theory of ordered fields} \label{sub:ElementaryTheory}

The axioms of an ordered field $(\R,+,\cdot,<)$ are formulated in first-order logic,
which means that we quantify over elements of $\R$, but not over subsets, functions, etc.
By way of contrast, the characterization of the field $\RR$ of real numbers 
(Theorem \ref{thm:CharacterizingTheReals}) is of a different nature:
here we have to quantify over subsets of $\RR$, or functions $\RR \to \RR$,
and such a formulation uses second-order logic.

The algebraic condition for an ordered field $\R$ to be real closed is of first order.
It is given by an axiom scheme where for each degree $n \in \NN$ we have the axiom 
\begin{multline}
  \label{eq:IntermediatValueAxiom}
  \forall a,b,c_0,c_1,\dots,c_n \in \R \, \bigl[
  (c_0 + c_1 a + \dots + c_n a^n)(c_0 + c_1 b + \dots + c_n b^n) < 0 
  \\ \Rightarrow \exists x \in \R \, \bigl( (x-a)(x-b) < 0 \;\land\; c_0 + c_1 x + \dots + c_n x^n = 0 \bigr) \bigr].
\end{multline}

First-order formulae are customarily called \emph{elementary}.
The collection of all first-order formulae that are true over 
a given ordered field $\R$ is called its \emph{elementary theory}.

Tarski's theorem \cite{Jacobson:1985:1989,BochnakCosteRoy:1998}
says that all real closed fields share the same elementary theory:
if an assertion in the first-order language of ordered fields
is true over one real closed field, for example the real numbers, 
then it is true over every real closed field.
(This no longer holds for second-order assertions, where $\RR$ 
is singled out as in Theorem \ref{thm:CharacterizingTheReals}.)

Tarski's theorem implies that euclidean geometry, 
seen as cartesian geometry modeled on the vector space $\RR^n$,
remains unchanged if the field $\RR$ of real numbers 
is replaced by any other real closed field $\R$.
This is true as far as its first-order properties are
concerned, and these comprise the core of classical geometry.
In this vein we encode the geometric notion of winding number 
in the first-order theory of real closed fields.

\begin{remark} 
  Tarski's theorem is a vast generalization of Sturm's technique, 
  and so is its effective formulation, called \emph{quantifier elimination},
  which provides explicit decision procedures.
  In principle such procedures could be used to generate a proof 
  of the Fundamental Theorem of Algebra in every fixed degree.
  We will not use Tarski's theorem, however, and we only mention it
  in order to situate our approach in its logical context.
\end{remark}

\begin{Annotation}
  \textbf{(Categoricity)}
  An axiom system is called \emph{categorical} if any two of its models are isomorphic,
  or stated differently, if it has only one model up to isomorphism.
  For example, this is the case for the axioms characterizing the natural numbers
  (where the axiom of induction is of second order) or the axioms characterizing 
  the real numbers (where the least upper bound axiom is of second order).
  The first-order axioms of real closed fields are not categorical,
  because there are many nonisomorphic models, besides $\RR$ for example $\QQ^c$.
\end{Annotation}

\begin{Annotation}
  \textbf{(Decidability)}
  The elementary theory of real closed fields can be recursively axiomatized,
  as seen above.  By Tarski's theorem it is \emph{complete} in the sense that 
  any two models of it share the same elementary theory.  
  This implies decidability, that is, the true formulae can be determined effectively.
  Effective decision procedures are provided by \emph{quantifier elimination},
  and the quest for \emph{efficient} quantifier elimination algorithms,
  at least for certain families of formulae, is an active area of research.
\end{Annotation}


\section{Sturm's theorem for real polynomials} \label{sec:Sturm}

This section recalls Sturm's theorem for polynomials over a real closed field
-- a gem of 19th century algebra and one of the greatest discoveries in the theory of polynomials.

It seems impossible to surpass the elegance of the original 
m\'emoires by Sturm \cite{Sturm:1835} and Cauchy \cite{Cauchy:1837}.
One technical improvement of our presentation, however, seems noteworthy: 
The inclusion of boundary points streamlines the arguments so that 
they will apply seamlessly to the complex setting in \sref{sec:Cauchy}.
The necessary amendments render the development hardly any longer or more complicated.
They pervade, however, all statements and proofs, so that it seems 
worthwhile to review the classical arguments in full detail.

\subsection{Counting sign changes}

For every ordered field $\R$, we define $\sign \colon \R \to \{-1,0,+1\}$ 
by $\sign(x) = +1$ if $x > 0$, $\sign(x) = -1$ if $x < 0$, and $\sign(0)=0$.
Given a finite sequence $s = (s_0,\dots,s_n)$ in $\R$, we say that 
the pair $(s_{k-1},s_k)$ presents a \emph{sign change} if $s_{k-1} s_k < 0$.
The pair presents \emph{half a sign change} if one element is zero while the other 
is nonzero.  In the remaining cases there is no sign change.
All cases can be subsumed by the formula 
\begin{equation} \label{eq:SingleSignChange}
  V(s_{k-1},s_k) := \half \bigl| \sign(s_{k-1}) - \sign(s_{k}) \bigr| .
\end{equation}

\begin{definition} \label{def:SignCounting}
  For a finite sequence $s = (s_0,\dots,s_n)$ in $\R$ 
  the \emph{number of sign changes} is 
  \begin{equation} \label{eq:SignChanges}
    V(s) := \sum_{k=1}^n V(s_{k-1},s_k) 
    = \sum_{k=1}^n \half \bigl| \sign(s_{k-1}) - \sign(s_{k}) \bigr| .
  \end{equation}

  For a finite sequence $(S_0,\dots,S_n)$ of polynomials in $\R[X]$
  and $a \in \R$ we set 
  \begin{equation} \label{eq:PolynomialSignChanges}
    V_a\bigl( S_0,\dots,S_n \bigr) := 
    V\bigl( S_0(a),\dots,S_n(a) \bigr) .
  \end{equation}

  For the difference at two points $a,b \in \R$ 
  we use the notation $V_a^b := V_a - V_b$.
\end{definition}

\begin{Annotation}
  The number $V(s_0,\dots,s_n)$ does not change 
  if we multiply all $s_0,\dots,s_n$ by some constant $q \in \R^*$.
  Likewise, $V_a^b(S_0,\dots,S_n)$ remains unchanged if we multiply all $S_0,\dots,S_n$ 
  by some polynomial $Q \in \R[X]^*$ that does not vanish in $\{a,b\}$.
  Such operations will be used repeatedly later on.
\end{Annotation}

There is no universal agreement how to count sign changes 
because each application requires its specific conventions. 
While there is no ambiguity for $s_{k-1} s_k < 0$ and $s_{k-1} s_k > 0$, 
some arbitration is needed to take care of possible zeros.
Our definition \eqref{eq:SingleSignChange} has been chosen 
to account for boundary points in Sturm's theorem, as explained below. 

The traditional way of counting sign changes, following Descartes, 
is to extract the subsequence $\hat{s}$ by discarding all zeros of $s$
and to define $\hat{V}(s) := V(\hat{s})$.  (This counting rule is 
nonlocal whereas in \eqref{eq:SignChanges} only neighbours interact.)
As an illustration we recall Descartes' rule of signs 
and its generalization due to Budan and Fourier 
\cite[chap.\,10]{RahmanSchmeisser:2002}.

\begin{theorem} \label{thm:DescartesBudanFourier}
  For every nonzero polynomial $P = c_0 + c_1 X + \dots + c_n X^n$ 
  over an ordered field $\R$,  the number of positive roots
  counted with multiplicity satisfies the inequality
  \begin{equation} \label{eq:DescartesRule}
    \mcard \bigl\{\, x \in \R_{>0} \bigm| P(x)=0 \,\bigr\} 
    \quad\le\quad \hat{V}(c_0,c_1,\dots,c_n) .
  \end{equation}
  More generally, the number of roots in any interval $\ei{a}{b} \subset \R$ satisfies the inequality
  \begin{equation} \label{eq:BudanFourier}
    \mcard \bigl\{\, x \in \ei{a}{b} \bigm| P(x)=0 \,\bigr\}
        \quad\le\quad \hat{V}_a^b(P,P',\dots,P^{(n)}) .
  \end{equation}
  Equality holds for every interval $\ei{a}{b} \subset \R$ if and only if $P$ has $n$ roots in $\R$.
  \\
  The excess $(\text{r.h.s.} - \text{l.h.s.})$ is even for all $P,a,b$ if and only if $\R$ is real closed.
  \qed
\end{theorem}

\begin{Annotation}
  \textbf{(Sketch of proof)}
  Descartes' rule \eqref{eq:DescartesRule} is a special case of 
  the Budan--Fourier bound \eqref{eq:BudanFourier} for $a=0$ 
  and $b \to +\infty$, so we concentrate on the latter.
  The number $\hat{V}_a$ counts the sign variations at $a$,
  and the difference $\hat{V}_a^b$ tells us how many sign variations 
  we lose when going from $a$ to $b$.  
  We assume that $P$ is of degree $n$, which means $c_n \ne 0$, so that
  the polynomials $P,P',\dots,P^{(n)}$ have only finitely many zeros.

  \begin{enumerate}
  \item[(i)]
    Passing a simple zero of $P$, where $P(x) = 0$ but $P'(x) \ne 0$, we lose one sign variation: 
    either $(-,+,\dots) \to (0,+,\dots) \to (+,+,\dots)$ or $(+,-,\dots) \to (0,-,\dots) \to (-,-,\dots)$.
    At a zero of $P$ of multiplicity $2$, where $P(x) = P'(x) = 0$ but $P''(x) \ne 0$, we lose two sign variations:
    either $(+,-,+,\dots) \to (0,0,+,\dots) \to (+,+,+,\dots)$ or $(-,+,-,\dots) \to (0,0,-,\dots) \to (-,-,-,\dots)$.
    In general, at every zero of $P$ of multiplicity $m$, where $P(x) = P'(x) = \dots P^{(m-1)}(x) = 0$ 
    but $P^{(m)}(x) \ne 0$, we lose $m$ sign variations.  
  \item[(ii)]
    At a zero $x$ of some derivative, where $P^{(k)}(x) = 0$ for $0 < k < n$ 
    but $P^{(k-1)}(x) \ne 0$, we always lose an even number of sign variations,
    for example $(\dots,+,+,-,\dots) \to (\dots,+,0,-,\dots) \to (\dots,+,-,-,\dots)$
    or $(\dots,+,-,+,\dots) \to (\dots,+,0,+,\dots) \to (\dots,+,+,+,\dots)$. 
    (The details are left to the reader.)
  \end{enumerate}
  
  If we assume the field $\R$ to be real closed, then (i) and (ii) cover all sign changes.
  Counting all zeros of $P$ on one hand and adding up all sign variations of $P,P',\dots,P^{(n)}$ 
  on the other hand thus yields the Budan--Fourier inequality \eqref{eq:BudanFourier} by (i), 
  and the excess $(\text{r.h.s.} - \text{l.h.s.})$ is always an even integer by (ii).
  If $\R$ is not real closed, we may lose additional sign variations without intermediate zeros, 
  so the inequality \eqref{eq:BudanFourier} still holds.  The excess, however, can now be odd:
  this happens whenever $P(a) P(b) < 0$ without any intermediate zero of $P$.
  
  For $a \to -\infty$ and $b \to +\infty$ we find $\hat{V}_a^b(P,P',\dots,P^{(n)}) = n$,
  because only the highest term counts and yields $\hat{V}_a = n$ and $\hat{V}_b = 0$.
  If equality in \eqref{eq:BudanFourier} holds, then $P$ has $n$ zeros in $\R$.
  Conversely, if $P$ has $n$ zeros in $\R$, then case (i) shows that we lose $n$ sign variations 
  according to the zeros of $P$.  Since $\hat{V}_a^b = n$, we cannot lose any further 
  sign variations in case (ii), so equality in \eqref{eq:BudanFourier} holds for all intervals.
\end{Annotation}

\begin{example}[signature] \label{example:Signature}
  For a self-adjoint matrix $A \in \CC^{n \times n}$, where $A^{\smash{\mathrm{T}}} = \overline{A}$,
  all eigenvalues are real.  Its \emph{signature} is defined as the difference $p-q$ 
  where $p$ resp.\ $q$ is the number of positive resp.\ negative eigenvalues.
  These can be read from the characteristic polynomial $P = c_0 + c_1 X + \dots + c_n X^n$
  as $p = \hat{V}(c_0,c_1,\dots,c_n)$ and $q = \hat{V}(c_0,-c_1,\dots,(-1)^n c_n)$.
\end{example}

\begin{remark} 
  \label{rem:DifferentiableFunctions}
  The Budan--Fourier bound is not restricted to polynomials.  
  Over the real numbers $\RR$ the inequality \eqref{eq:BudanFourier} 
  holds for every $n$-times differentiable function $P \ne 0$
  such that $P^{(n)}$ is of constant sign on $[a,b]$. 
  This extends to every ordered field $\R$, provided that differentiability 
  of $f \colon [a,b] \to \R$ means that there exists $f' \colon [a,b] \to \R$ and $C > 0$ 
  such that $\abs{ f(x)-f(x_0) - f'(x_0)(x-x_0) } \le C \abs{ x-x_0 }^2$ for all $x,x_0 \in [a,b]$.
\end{remark}


\begin{Annotation}
  \textbf{(Sketch of proof)}
  First one shows that $f' \ge 0$ (resp.\ $f' > 0$) implies that $f$ is (strictly) increasing, 
  and symmetrically,  $f' \le 0$ (resp.\ $f' < 0)$ implies that $f$ is (strictly) decreasing.
  Over the reals this is a familiar result for differentiable functions;
  over an arbitrary ordered field, Lipschitz differentiability is a convenient substitute.
  (As a counterexample where the usual differentiability fails over an ordered field, 
  consider $f \colon \QQ \to \QQ$ defined by $f(x) = 0$ for $x^2 > 2$ and $f(x) = 1$ for $x^2 < 2$.
  This function is differentiable, with $f' = 0$, but $f$ is not constant.
  Likewise $g(x) = f(x) + x$ is differentiable, with $g' = 1$, but $g$ is not increasing.)

  We can assume $f^{(n)} \ne 0$. 
  By hypothesis, $f^{(n)}$ is of constant sign, so $f^{(n-1)}$ is strictly monotonic:
  it thus has at most one zero and is of constant sign on each of two complementary subintervals.
  Likewise, $f^{(n-2)}$ has at most two zeros and is of constant sign on each of three complementary subintervals.
  Iterating this argument, $f$ has at most $n$ zeros and is of constant sign on each of $n+1$ complementary subintervals.
  
  Cases (i) and (ii) work as before.  If $f,f',\dots,f^{(n-1)}$ satisfy 
  the intermediate value property (over $\RR$, for example, because they are continuous),
  then (i) and (ii) cover all sign changes; otherwise we may lose 
  the same sign variations without intermediate zeros.
  Counting all zeros of $f$ on one hand and adding up all sign variations of $f,f',\dots,f^{(n)}$ 
  on the other hand thus yields the Budan--Fourier inequality \eqref{eq:BudanFourier}.
  If $f$ satisfies the intermediate value property, then the excess 
  $(\text{r.h.s.} - \text{l.h.s.})$ is always an even integer;
  otherwise $f$ may change signs without intermediate zero, 
  and the excess need no longer be even.

  We also notice that $0 \le \hat{V}_a^b \le n$.  If $f$ happens to have $n$ zeros 
  in $\ei{a}{b}$, then equality holds in \eqref{eq:BudanFourier}.
  Conversely, if equality holds, then case (i) shows that we lose sign variations only according 
  to the zeros of $f$ and cannot lose any further sign variations in case (ii).
  In this situation equality in \eqref{eq:BudanFourier} holds for all subintervals.
\end{Annotation}

The upper bounds \eqref{eq:DescartesRule} and \eqref{eq:BudanFourier} 
are easy to compute but often overestimate the number of roots.
This was the state of knowledge before Sturm's ground-breaking discovery in 1829.
Sturm's theorem (Corollary \ref{cor:CountingRealRoots} below) gives the precise number of roots.

\subsection{The Cauchy index} \label{sub:RealCauchyIndex}

The Cauchy index judiciously counts roots with a sign $\pm1$ 
encoding the passage from negative to positive or from positive to negative.
Instead of zeros of $P$, it is customary to count poles of $f = \frac{1}{P}$, which is of course equivalent.

Informally, as illustrated in Figure \ref{fig:CauchyIndex}, 
we set $\Ind_a(f) = +1$ if $f$ jumps from $-\infty$ to $+\infty$,
and $\Ind_a(f) = -1$ if $f$ jumps from $+\infty$ to $-\infty$, 
and $\Ind_a(f) = 0$ in all other cases.  

\begin{figure}[ht]
  \centering
  \includegraphics[width=\linewidth]{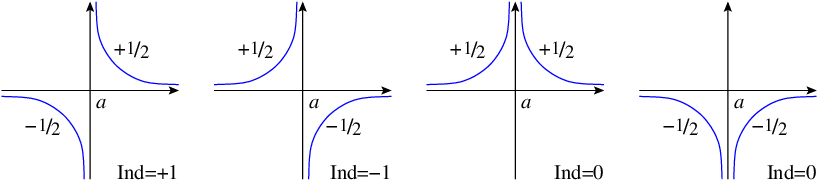}
  \caption{A pole $a$ and its Cauchy index $\Ind_a(f) = \Ind_a^+(f) - \Ind_a^-(f)$}
  \label{fig:CauchyIndex}
\end{figure}

Formally, we define the right limit $\lim_a^+ f$ and the left limit $\lim_a^- f$ 
of $f \in \R(X)^*$ at $a \in \R$ by factoring $f = (X-a)^m g$ 
with $m \in \ZZ$ and $g \in \R(X)^*$ such that $g(a) \in \R^*$.
If $m \ge 0$, then $\lim_a^\varepsilon f = f(a) \in \R$ for both $\varepsilon \in \{\pm\}$; 
if $m < 0$, then $\lim_a^\varepsilon f = \varepsilon^m \cdot \sign g(a) \cdot (+\infty) \in \{\pm\infty\}$. 


\begin{definition} \label{def:CauchyIndex}
  The \emph{Cauchy index} of a rational function $f \in \R(X)^*$ at a point $a \in \R$ is 
  \begin{equation}
    \Ind_a(f) := \Ind_a^+(f) - \Ind_a^-(f)
    \quad\text{where}\quad
    \Ind_a^\varepsilon(f) := \begin{cases} 
      +\half & \text{if $\lim_a^\varepsilon f = +\infty$,} \\
      -\half & \text{if $\lim_a^\varepsilon f = -\infty$,} \\
      0      & \text{otherwise.}
    \end{cases}
  \end{equation}

  For $a < b$ in $\R$ we define the Cauchy index of $f \in \R(X)^*$ on the interval $[a,b]$ by
  \begin{equation} \label{eq:DefCauchyIndex}
    \Ind_a^b(f) := \Ind_a^+(f) + \sum_{x \in \ee{a}{b}} \Ind_x(f) - \Ind_b^-(f).
  \end{equation}

  The sum is well-defined because only finitely many points $x \in \ee{a}{b}$ contribute.
  
  For $b < a$ we define $\Ind_a^b(f) := - \Ind_b^a(f)$,
  and for $a=b$ we set $\Ind_a^a(f) := 0$.

  Finally, we set $\Ind_a^b(\frac{R}{S}) := 0$
  in the degenerate case where $R=0$ or $S=0$.
\end{definition}

Here we opt for a more comprehensive definition \eqref{eq:DefCauchyIndex}
than usual, in order to take care of boundary points.
We will frequently subdivide intervals, 
and this technique works best 
with a uniform definition that avoids case distinctions.
Moreover, we will have reason to consider piecewise
rational functions in \sref{sec:Cauchy}.

\begin{proposition} \label{prop:RealIndexProperties}
  The Cauchy index enjoys the following properties.
  \begin{enumerate}
  \item[(a)] Subdivision: \;
    $\Ind_a^b(f) + \Ind_b^c(f) = \Ind_a^c(f)$ \; 
    for all $a,b,c \in \R$.  
  \item[(b)] Invariance: \;
    $\Ind_a^b(f \circ \tau) = \Ind_{\smash{\tau(a)}}^{\smash{\tau(b)}}(f)$ \; 
    for every linear fractional transformation $\tau \colon [a,b] \to \R$, 
    $\tau(x)=\frac{p x + q}{r x + s}$ where $p,q,r,s \in \R$, without poles on $[a,b]$.
  \item[(c)] Scaling: \;
    $\Ind_a^b(g f) = \sign(g) \Ind_a^b(f)$ \; 
    if $g$ is of constant sign on $[a,b]$.
  \item[(d)] Addition: \;
    $\Ind_a^b(f+g) = \Ind_a^b(f) + \Ind_a^b(g)$ \; 
    if $f,g$ have no common poles. 
    \qed
  \end{enumerate}
\end{proposition}

\begin{Annotation}
  \textbf{(Rational functions as maps)}
  In view of Definition \ref{def:CauchyIndex}
  we wish to interpret rational functions $f \in \R(X)$ as maps.
  The right way to do this is to extend the affine line $\R$ 
  to the projective line $\P\R = \R\cup\{\infty\}$.

  We construct $\P\R = (\R^2\minus\{0\})/_\sim$ as the quotient of $\R^2\minus\{0\}$ 
  by the quivalence $(p,q) \sim (s,t)$ defined by the condition 
  that there exists $u \in \R^*$ such that $(p,q) = (u r, u s)$.
  The equivalence class of $(p,q)$ is denoted by $[p:q]$ and represents 
  the line passing through the origin $(0,0)$ and $(p,q)$ in $\R^2$.
  The affine line $\R$ can be identified with $\{\, [p:1] \mid p \in \R \,\}$;
  this covers all points of $\P\R$ except one: the point at infinity, $\infty = [1:0]$. 

  Likewise we construct $\P\R(X) = (\R(X)^2\minus\{0\})/_\sim$ as the quotient of $\R(X)^2\minus\{0\}$ 
  by the quivalence $(P,Q) \sim (R,S)$ defined by the condition 
  that there exists $U \in \R(X)^*$ such that $(P,Q) = (U R, U S)$.
  The equivalence class of $(P,Q)$ is denoted by $[P:Q]$.
  Here $\R(X)$ can be identified with $\{\, [P:Q] \mid P,Q \in \R[X], Q \ne 0 \,\}$
  using only polynomials.  Again this covers all points of $\P\R(X)$ 
  except one: the point at infinity, $\infty = [1:0]$. 

  Consider $f = [P:Q] \in \P\R(X)$ with $P,Q \in \R[X]$.
  We can assume $\gcd(P,Q)=1$ and set $m =: \max\{\, \deg P, \deg Q \,\}$.
  We then construct homogenous polynomials $\hat{P},\hat{Q} \in \R[X,Y]$ 
  by $X^k \mapsto X^k Y^{m-k}$.  We have $(\hat{P}(x,y),\hat{Q}(x,y)) \ne (0,0)$ 
  for all $(x,y) \ne (0,0)$ in $\R^2$, and the map $\hat{f} \colon \P\R \to \P\R$
  given by $\hat{f}([x:y]) = [\hat{P}(x,y),\hat{Q}(x,y)]$ is well-defined.

  This construction allows us to interpret every $f \in \P\R(X)$ and in particular 
  every rational fraction $f \in \R(X)$ as a map $\hat{f} \colon \P\R \to \P\R$.  
  In the sequel most constructions for $P/Q$ resp.\ $[P:Q]$ are slightly 
  easier in the generic case where $P,Q \in \R[X]^*$, but can easily 
  be extended to the exceptional cases where $P=0$ or $Q=0$.
\end{Annotation}

\begin{Annotation} 
  \textbf{(Winding number)} \label{anno:WindingNumber}
  We can present the ordered field $\R$ as an oriented line,
  the two ends being denoted by $-\infty$ and $+\infty$.
  It is sometimes convenient to formally adjoin two further elements 
  $\pm\infty$ and to extend the order of $\R$ to $\bar\R := \R\cup\{\pm\infty\}$
  such that $-\infty < x < +\infty$ for all $x \in \R$.
  This turns $\bar\R$ into a closed interval.

  \begin{figure}[ht]
    \centering
    \includegraphics[height=16ex]{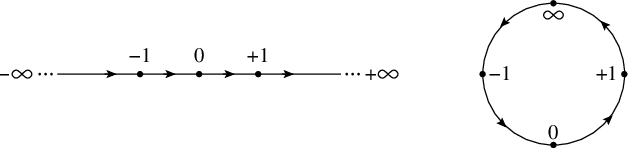}
  \end{figure}

  We can think of the projective line $\P\R = \R\cup\{\infty\}$ as an oriented circle.
  In the above picture this is obtained by identifying $+\infty$ and $-\infty$ in $\bar\R$.
  Even though we cannot extend the ordering of $\R$ to $\P\R$,
  we can nevertheless define a sign function $\P\R \to \{-1,0,+1\}$
  by $\sign([p:q]) = \sign(pq)$, which simply means that $\sign(\infty) = 0$.
  
  The intermediate value property now takes the following form:
  if $f \in \P\R(X)$ satisfies $f(a) f(b) < 0$ for some $a < b$ in $\R$,
  then there exists $x \in \ee{a}{b}$ such that $\sign f(x) = 0$,
  that is $f(x) = 0$ or $f(x) = \infty$.  In other words, a rational 
  function does not change sign without passing through $0$ or $\infty$.
  
  The Cauchy index $\Ind_a^b(f)$ counts the number of times 
  that $f$ crosses $\infty$ from $-$ to $+$ (clockwise in the figure above) 
  minus the number of times that $f$ crosses $\infty$ from $+$ to $-$ 
  (counter-clockwise in the above figure).  
  This geometric interpretation anticipates the winding number 
  of loops in the plane constructed in \sref{sec:Cauchy}.
\end{Annotation}

\subsection{Counting real roots} \label{sub:CountingRealRoots}

The ring $\R[X]$ is equipped with a derivation $P \mapsto P'$
sending each polynomial $P = \sum_{k=0}^n p_k X^k$ 
to its formal derivative $P' = \sum_{k=1}^n k p_k X^{k-1}$. 
This extends in a unique way to a derivation on the field $\R(X)$ 
sending $f = \frac{R}{S}$ to $f' = \frac{R' S - R S'}{S^2}$.
This is an $\R$-linear map satisfying Leibniz' rule $(f g)' = f' g + f g'$.
For $f \in \R(X)^*$ the quotient $f'/f$ is called
the \emph{logarithmic derivative} of $f$; it enjoys the following property.

\begin{proposition} \label{prop:LogarithmicDerivative}
  For every $f \in \R(X)^*$ we have 
  $\Ind_a(f'/f) = +1$ if $a$ is a zero of $f$,
  and $\Ind_a(f'/f) = -1$ if $a$ is a pole of $f$,
  whereas $\Ind_a(f'/f) = 0$ in all other cases.
\end{proposition}

\begin{proof}
  We have $f = (X-a)^m g$ with $m \in \ZZ$
  and $g \in \R(X)^*$ such that $g(a) \in \R^*$.
  By Leibniz' rule we obtain 
  $\frac{f'}{f} = \frac{m}{X-a} + \frac{g'}{g}$.
  The fraction $\frac{g'}{g}$ does not contribute 
  to the index because it does not have a pole at $a$.
  We conclude that $\Ind_a(f'/f) = \sign(m)$.
\end{proof}

\begin{corollary}
  For every $f \in \R(X)^*$ and $a < b$ in $\R$ the index $\Ind_a^b(f'/f)$ 
  equals the number of roots minus the number of poles of $f$ in $[a,b]$,
  counted without multiplicity.  Roots and poles on the boundary count for one half.
  \qed
\end{corollary}

The corollary remains true for $f = \frac{R}{S}$ when $R=0$ or $S=0$,
with the convention that we count only \emph{isolated} roots and poles.
Polynomials $P \in \R[X]^*$ have no poles, whence $\Ind_a^b(P'/P)$ 
simply counts the number of roots of $P$ in $[a,b]$.

\subsection{The inversion formula} \label{sub:InversionFormula}

While the Cauchy index can be defined over any ordered field $\R$,
the following results require $\R$ to be real closed.
They will allow us to calculate the Cauchy index by Sturm chains (\sref{sub:SturmChains})
via iterated Euclidean division (\sref{sub:EuclideanChains}).

The starting point is the observation that 
the intermediate value property of polynomials $P \in \R[X]$ 
can then be reformulated quantitatively as $\Ind_a^b(\frac{1}{P}) = V_a^b(1,P)$.  
More generally, we have the following inversion formula 
of Cauchy \cite[\textsection I, Thm.\,I]{Cauchy:1837}.

\begin{theorem} \label{thm:Inversion}
  Let $\R$ be a real closed field. 
  For all $P,Q \in \R[X]$ and $a,b \in \R$ we have
  \begin{equation} \label{eq:InversionGeneral}
    \Ind_a^b\Bigl(\frac{Q}{P}\Bigr) + \Ind_a^b\Bigl(\frac{P}{Q}\Bigr) 
    = V_a^b\Bigl( 1, \frac{P}{Q} \Bigr) = V_a^b\Bigl( 1, \frac{Q}{P} \Bigr) .
  \end{equation}
  If $P$ and $Q$ do not have common zeros at $a$ or $b$, 
  then this simplifies to 
  \begin{equation} \label{eq:Inversion}
    \Ind_a^b\Bigl(\frac{Q}{P}\Bigr) + \Ind_a^b\Bigl(\frac{P}{Q}\Bigr) 
    = V_a^b\bigl( P, Q \bigr) .
  \end{equation}
\end{theorem}


If $a$ or $b$ is a pole of $\frac{P}{Q}$ or $\frac{Q}{P}$, then the signs 
in \eqref{eq:InversionGeneral} are evaluated using the convention $\sign(\infty) = 0$.
The inversion formula will follow as a special case 
from the product formula \eqref{eq:RealProductFormula},
but its proof is short enough to be given separately here.

\begin{proof}
  We can assume $a < b$ and $P,Q \in \R[X]^*$ and $\gcd(P,Q) = 1$,
  so each pole is a zero of either $P$ or $Q$, and 
  Equations \eqref{eq:InversionGeneral} and \eqref{eq:Inversion} become equivalent.  
  They are additive with respect to subdivision of $[a,b]$, by Proposition \ref{prop:RealIndexProperties}(a),
  so it suffices to treat the case where $[a,b]$ contains at most one pole.

  \textit{Global analysis away from poles:}
  Suppose that $[a,b]$ does not contain zeros of $P$ or $Q$.
  Then both indices $\Ind_a^b\bigl(\frac{Q}{P}\bigr)$ and 
  $\Ind_a^b\bigl(\frac{P}{Q}\bigr)$ vanish in the absence of poles,
  and the intermediate value property ensures that $P$ and $Q$ 
  are of constant sign on $[a,b]$, whence $V_a^b(P,Q) = 0$.

  \textit{Local analysis at a pole:}
  Suppose that $[a,b]$ contains a pole.
  Subdividing, if necessary, we can assume that this pole is either $a$ or $b$.  
  Applying the symmetry $X \mapsto a+b-X$, if necessary,
  we can assume that the pole is $a$.
  Since Equation \eqref{eq:Inversion} is symmetric 
  in $P$ and $Q$, we can assume that $P(a) = 0$.  
  We then have $Q(a) \ne 0$, whence $Q$ has constant 
  sign on $[a,b]$ and $\Ind_a^b\bigl(\frac{P}{Q}\bigr) = 0$.
  Likewise, $P$ has constant sign on $\ei{a}{b}$ and
  $\Ind_a^b\bigl(\frac{Q}{P}\bigr) = \Ind_a^+\bigl(\frac{Q}{P}\bigr)$.
  On the right hand side we find $V_a(P,Q) = \nicefrac{1}{2}$,
  and for $V_b(P,Q)$ two cases occur:
  \begin{itemize}
  \item
    If $V_b(P,Q) = 0$, then $\frac{Q}{P} >0$ on $\ei{a}{b}$,
    whence $\lim_a^+\bigl(\frac{Q}{P}\bigr) = +\infty$.
  \item
    If $V_b(P,Q) = 1$, then $\frac{Q}{P} < 0$ on $\ei{a}{b}$,
    whence $\lim_a^+\bigl(\frac{Q}{P}\bigr) = -\infty$.
  \end{itemize}
  In both cases we find $\Ind_a^+\bigl(\frac{Q}{P}\bigr) = V_a^b(P,Q)$,
  whence Equation \eqref{eq:Inversion} holds.
\end{proof}

\begin{Annotation}
  \textbf{(Local and global arguments)}
  The previous proof relies on a local argument around a pole $a$, in the neighbourhoods 
  $[a,a+\varepsilon]$ and $[a-\varepsilon,a]$ for some chosen $\varepsilon > 0$,
  and a global argument, for a given interval $[a,b]$ without poles.
  The local argument only uses continuity and is valid for polynomials over any ordered field.
  It is in the global argument that we need the intermediate value property.
  This interplay of local and global arguments is a recurrent theme 
  in the proofs of \sref{sub:ProductFormula} and \sref{sub:CountingComplexRoots}.
\end{Annotation}

\subsection{Sturm chains} \label{sub:SturmChains}

In the rest of this section we exploit the inversion formula \eqref{eq:Inversion}, 
and we will therefore continue to assume $\R$ to be real closed.
We can then calculate the Cauchy index $\Ind_a^b(\frac{R}{S})$ 
by iterated euclidean division (\sref{sub:EuclideanChains}).
The crucial condition is the following.

\begin{definition} \label{def:SturmChain}
  A sequence of polynomials $(S_0,\dots,S_n)$ in $\R[X]$ is a \emph{Sturm chain} 
  with respect to an interval $I \subset \R$ 
  if it satisfies Sturm's condition:
  \begin{equation} \label{eq:SturmCondition}
    \text{If $S_k(x)=0$ for some $x \in I$ and $0 < k < n$, then $S_{k-1}(x) S_{k+1}(x) < 0$.}
  \end{equation}
\end{definition}

We will usually not explicitly mention the interval 
if it is understood from the context, or if 
$(S_0,\dots,S_n)$ is a Sturm chain on all of $\R$.
For $n=1$ Condition \eqref{eq:SturmCondition}
is void and should be replaced by the requirement
that $S_0$ and $S_1$ have no common zeros.

\begin{theorem} \label{thm:SymmetricSturm}
  If $(S_0,S_1,\dots,S_{n-1},S_n)$ is a Sturm chain in $\R[X]$ with respect to $[a,b]$, then
  \begin{equation} \label{eq:SymetricSturm}
    \Ind_a^b \Bigl( \frac{S_1}{S_0} \Bigr) + \Ind_a^b \Bigl( \frac{S_{n-1}}{S_n} \Bigr)
    = V_a^b \bigl( S_0,S_1,\dots,S_{n-1},S_n \bigr) . 
  \end{equation}
\end{theorem}

\begin{proof}
  For $n=1$ this is the inversion formula \eqref{eq:Inversion}.
  For $n = 2$ the inversion formula implies 
  \[
    \Ind_a^b \Bigl( \frac{S_1}{S_0} \Bigr) 
    + \Ind_a^b \Bigl( \frac{S_0}{S_1} \Bigr) 
    + \Ind_a^b \Bigl( \frac{S_2}{S_1} \Bigr)
    + \Ind_a^b \Bigl( \frac{S_1}{S_2} \Bigr)
    = V_a^b \bigl( S_0,S_1,S_2 \bigr) . 
  \]
  This is a telescopic sum:  
  contributions to the middle indices arise at zeros of $S_1$, 
  but at each zero of $S_1$ its neighbours $S_0$ and $S_2$
  have opposite signs, which means that these terms cancel each other.
  Iterating this argument, we obtain \eqref{eq:SymetricSturm} by induction on $n$.
\end{proof}


The following algebraic criterion for Sturm chains will be useful 
in \sref{sub:EuclideanChains} and \sref{sub:CountingComplexRoots}:

\begin{proposition} \label{prop:AlgebraicSturmCriterion}
  Consider a sequence $(S_0,\dots,S_n)$ in $\R[X]$ such that 
  \begin{equation} \label{eq:AlgebraicSturmCondition}
    A_k S_{k+1} + B_k S_{k} + C_k S_{k-1} = 0 
    \qquad\text{for}\qquad 0 < k < n,
  \end{equation}
  with $A_k,B_k,C_k \in \R[X]$ satisfying $A_k > 0$ 
  and $C_k \ge 0$ on some interval $I \subset \R$.
  Then $(S_0,\dots,S_n)$ is a Sturm chain on $I$ if and only if 
  the terminal pair $(S_{n-1},S_n)$ has no common zeros in $I$.
\end{proposition}

\begin{proof}
  We assume that $n \ge 2$. If $(S_{n-1},S_n)$ has a common zero, 
  then the Sturm condition \eqref{eq:SturmCondition} is obviously violated.
  Suppose that $(S_{n-1},S_n)$ has no common zeros in $I$.
  If $S_k(x) = 0$ for $x \in I$ and $0 < k < n$, then $S_{k+1}(x) \ne 0$.
  Otherwise Equation \eqref{eq:AlgebraicSturmCondition} would imply 
  that $S_k,\dots,S_n$ vanish at $x$, which is excluded by our hypothesis.
  Now, the equation $A_k(x) S_{k+1}(x) + C_k(x) S_{k-1}(x) = 0$ 
  with $A_k(x) S_{k+1}(x) \ne 0$ implies $C_k(x) S_{k-1}(x) \ne 0$.
  Using $A_k(x) > 0$ and $C_k(x) > 0$ we conclude that $S_{k-1}(x) S_{k+1}(x) < 0$.
\end{proof}

For many calulcations $A_k = C_k = 1$ suffices, as in \sref{sub:EuclideanChains},
but the general setting is more flexible because $A_k$ and $C_k$ can absorb 
positive factors and thus purge $S_{k+1}$ and $S_{k-1}$ of irrelevancy.
Sturm chains as in \eqref{eq:AlgebraicSturmCondition} also occur naturally for orthogonal polynomials.

\begin{Annotation}
  \textbf{(A historical example)}
  The following example is taken from Kronecker (1872) citing Gauss (1849)
  in his course \textit{Theorie der algebraischen Gleichungen}. 
  [Notes written by Kurt Hensel, archived at the University of Strasbourg,
  available at \link{num-scd-ulp.u-strasbg.fr/429}, page 165.]
 
  \begin{Example}
    We consider $P_0 = X^7 - 28 X^4 + 480$ and its derivative $P_1 = P_0' = 7X^2 ( X^4 - 16 X )$.
    We set $S_0 = P_0$ and $S_1 = X^4 - 16 X$, neglecting the positive factor $7X^2$.
    We wish to calculate $\Ind_a^b(\frac{P_1}{P_0}) = \Ind_a^b(\frac{S_1}{S_0})$ by constructing 
    a suitable Sturm chain.  Euclidean division yields $P_2 = (X^3-12) S_1  - S_0 = 192 X - 480$, 
    which we reduce to $S_2 = 2X-5$. 
    Likewise $P_3 = \frac{1}{16}( 8X^3 + 20 X^2 + 50 X - 3)S_2 - S_1 = \frac{15}{16}$ is reduced to $S_3 = 1$.
    We thus obtain a judiciously reduced Sturm chain $(S_0,S_1,S_2,S_3)$ 
    of the form $A_k S_{k+1} + B_k S_{k} + C_k S_{k-1} = 0$ with $A_k, C_k > 0$.
  \end{Example}
\end{Annotation}

\begin{Annotation}
  \textbf{(Orthogonal polynomials)}
  The Legendre polynomials $P_0,P_1,P_2,\dots$ in $\RR[X]$ are recursively
  defined by $P_0 = 1$, $P_1 = X$, and $(k+1) P_{k+1} + k P_{k-1} = (2k+1) X P_k$.
  They satisfy the orthogonality relation $\int_{-1}^{+1} P_m(x) P_n(x) \dd x = \frac{2}{2n + 1} \delta_{mn}$
  with respect to the inner product $\langle f, g \rangle = \int_{-1}^{+1} f(x) g(x) \dd x$.

  More generally, one can fix a measure $\mu$ on the real line $\RR$, 
  say with compact support, and consider the inner product 
  $\langle f, g \rangle = \int f(x) g(x) \dd\mu$.
  Orthogonality of $P_0,P_1,P_2,\dots$ means that 
  $\langle P_k , P_\ell \rangle = 0$ if $k \ne \ell$, and $>0$ if $k = \ell$.  
  If we further assume that $P_0,P_1,P_2,\dots$ are polynomials in $\RR[X]$ with $\deg P_k = k$,
  then orthogonality entails a three-term recurrence relation $A_k P_{k+1} + B_k P_k + C_k P_{k-1} = 0$
  with constants $A_k,C_k > 0$ and some polynomial $B_k$ of degree $1$, depending on $k$ and $\mu$.
  Orthogonal polynomials thus form a Sturm chain.
\end{Annotation}


\subsection{Euclidean chains} \label{sub:EuclideanChains}

The definition of Sturm chains is fairly general and 
could be used for more general functions than polynomials.
The crucial observation for polynomials is that the euclidean algorithm 
can be used to construct Sturm chains as follows.

Consider a rational function $f = \frac{R}{S} \in \R(X)^*$ 
represented by polynomials $R,S \in \R[X]^*$.
Iterated euclidean division produces a sequence of polynomials 
starting with $P_0=S$ and $P_1=R$, such that $P_{k-1} = Q_k P_k - P_{k+1}$ 
and $\deg P_{k+1} < \deg P_{k}$ for all $k=1,2,3,\dots$.
This process eventually stops when we reach $P_{n+1} = 0$,
in which case $P_n \sim \gcd(P_0,P_1)$.

Stated differently, this construction
is the expansion of $f$ into the continued fraction
\begin{equation*}
  f 
  = \frac{P_1}{P_0}
  = \frac{P_1}{Q_1 P_1 - P_2} 
  = \cfrac{1}{Q_1 - \cfrac{P_2}{P_1}} 
  = \cfrac{1}{Q_1 - \cfrac{1}{Q_2 - \cfrac{P_3}{P_2}}} = \dots 
  = \cfrac{1}{Q_1 - \cfrac{1}{Q_2 - \cfrac{\dots}{Q_{n-1} - \cfrac{1}{Q_n}}}} .
\end{equation*}

\begin{definition} \label{def:EuclideanChain}
  In this euclidean remainder sequence, the last polynomial $P_n \ne 0$ 
  divides all preceding polynomials $P_0,P_1,\dots,P_{n-1}$.  
  The \emph{euclidean chain} $(S_0,S_1,\dots,S_n)$ 
  associated to the fraction $\frac{R}{S} \in \R(X)^*$
  is defined by $S_k := P_k/P_n$ for $k=0,\dots,n$.
\end{definition}

We thus obtain $\frac{R}{S} = \frac{S_1}{S_0}$ with $\gcd(S_0,S_1) = S_n = 1$,
and by construction $(S_0,S_1,\dots,S_n)$ depends only on the fraction $\frac{R}{S}$ 
and not on the polynomials $R,S$ representing it.
By Proposition \ref{prop:AlgebraicSturmCriterion} the equations 
$S_{k-1} + S_{k+1} = Q_k S_k$ ensure that $(S_0,S_1,\dots,S_n)$ is a Sturm chain.

\begin{Annotation}
  \textbf{(The euclidean cochain)}
  The polynomials $(Q_1,\dots,Q_n)$ suffice to reconstruct the fraction $f$.
  This presentation is quite economic because they usually have low degree; 
  generically we expect $\deg(Q_k)=1$.

  We recover $(S_0,S_1,\dots,S_n)$ working backwards 
  from $S_{n+1} = 0$ and $S_n = 1$ by calculating 
  $S_{k-1} = Q_k S_k - S_{k+1}$ for all $k=n-1,\dots,0$.
  This procedure also provides an economic way to 
  evaluate $(S_0,S_1,\dots,S_n)$ at $a \in \R$.

  This indicates that, from an algorithmic point of view,
  the cochain $(Q_1,\dots,Q_n)$ is of primary interest.
  From a mathematical point of view it is often more convenient 
  to use the chain $(S_0,S_1,\dots,S_n)$.
\end{Annotation}

\subsection{Sturm's theorem} \label{sub:SturmTheorem}

We can now fix the following convenient notation.

\begin{definition} \label{def:SturmIndex}
  For $\frac{R}{S} \in \R(X)$ and $a,b \in \R$ we define the \emph{Sturm index} to be 
  \[ 
  \Sturm_a^b\Bigl( \frac{R}{S} \Bigr) := V_a^b\bigl( S_0,S_1,\dots,S_n \bigr) ,
  \]
  where $(S_0,S_1,\dots,S_n)$ is the euclidean chain associated to $\frac{R}{S}$.  
  We include two exceptional cases.
  If $S = 0$ and $R \ne 0$, the euclidean chain is $(0,1)$ of length $n=1$.
  If $R = 0$, we take the chain $(1)$ of length $n=0$.
  In both cases we obtain $\Sturm_a^b\bigl( \frac{R}{S} \bigr) = 0$. 
\end{definition}

This definition is effective in the sense that 
$\smash{\Sturm_a^b\bigl( \frac{R}{S} \bigr)}$ can immediately be calculated.
Definition \ref{def:CauchyIndex} of the Cauchy index 
$\smash{\Ind_a^b\bigl(\frac{R}{S}\bigr)}$, however, assumes knowledge of all roots of $S$ in $[a,b]$.
This difficulty is overcome by Sturm's celebrated theorem, generalized by Cauchy,
equating the Cauchy index with the Sturm index over a real closed field.

\begin{theorem}[Sturm 1829/35, Cauchy 1831/37] \label{thm:Sturm}
  For every pair $R,S \in \R[X]$ of polynomials 
  over a real closed field $\R$ we have 
  \begin{equation} \label{eq:CauchyEqualsSturm}
    \Ind_a^b\Bigl( \frac{R}{S} \Bigr) = \Sturm_a^b\Bigl( \frac{R}{S} \Bigr) .
  \end{equation}
\end{theorem}

\begin{proof}
  Equation \eqref{eq:CauchyEqualsSturm} is trivially true
  if $R = 0$ or $S = 0$, according to our definitions.
  We can thus assume $R,S \in \R[X]^*$.
  Let $(S_0,S_1,\dots,S_n)$ be the euclidean chain
  associated to the fraction $\frac{R}{S}$.
  Since $\frac{R}{S} = \frac{S_1}{S_0}$ and $S_n=1$,
  Theorem \ref{thm:SymmetricSturm} implies that
  \[
  \Ind_a^b\Bigl( \frac{R}{S} \Bigr) =
  \Ind_a^b\Bigl( \frac{S_1}{S_0} \Bigr) + \Ind_a^b\Bigl( \frac{S_{n-1}}{S_n} \Bigr)
  = V_a^b\bigl( S_0,S_1,\dots,S_n \bigr) 
  = \Sturm_a^b\Bigl( \frac{R}{S} \Bigr) .
  \qedhere
  \]
\end{proof}

This theorem is usually stated under the additional 
hypotheses that $\gcd(R,S) = 1$ and $S(a) S(b) \ne 0$.
Our formulation of Theorem \ref{thm:Sturm} does not require either of these conditions, 
because they are absorbed into our slightly refined definitions: 
$\gcd(R,S) = 1$ becomes superfluous by formulating Definitions \ref{def:CauchyIndex} 
and \ref{def:SturmIndex} such that both indices become well-defined on $\R(X)$.
The exception $S(a) S(b) = 0$ is anticipated in Definitions \ref{def:SignCounting} 
and \ref{def:CauchyIndex} by counting boundary points correctly.
Arranging these details is not only an aesthetic preoccupation:
it clears the way for a uniform treatment of the complex case 
in \sref{sec:Cauchy}. 

As an immediate consequence of \sref {sub:CountingRealRoots} we obtain 
Sturm's classical theorem \cite[\textsection 2]{Sturm:1835}.

\begin{corollary}[Sturm 1829/35] \label{cor:CountingRealRoots} 
  For every polynomial $P \in \R[X]^*$ we have
  \begin{equation} \label{eq:RootCount}
    \card\bigl\{\, x \in [a,b] \bigm| P(x)=0 \,\bigr\}
    = \Sturm_a^b\Bigl( \frac{P'}{P} \Bigr) ,
  \end{equation}
  where roots on the boundary count for one half.
  \qed
\end{corollary}

By the usual bisection method, Formula \eqref{eq:RootCount} provides 
an algorithm to locate all real roots of any given real polynomial.
Once the roots are well separated, one can switch to Newton's method
(\sref{sub:NewtonMethod}), which is simpler to apply and converges much faster.

\begin{remark}
  Formula \eqref{eq:RootCount} counts real roots of $P$ 
  without multiplicity. Multiplicities can be counted 
  by observing that $x$ is a root of $P$ of multiplicity $m \ge 2$ 
  if and only if $x$ is a root of $\gcd(P,P')$ of multiplicity $m-1$. 
  See Rahman--Schmeisser \cite[Thm.\,10.5.6]{RahmanSchmeisser:2002}.
\end{remark}

\begin{remark}
  The intermediate value property is essential 
  for \eqref{eq:CauchyEqualsSturm} and \eqref{eq:RootCount}.
  Over $\QQ$, for example, the function $f(x) = 2x/(x^2-2)$ has no poles, whence $\Ind_1^2(f) = 0$.  
  A Sturm chain is given by $S_0 = X^2-2$ and $S_1 = 2X$ and $S_2 = 2$, 
  whence $V_1^2(S_0,S_1,S_2) = 1$.  Here the Sturm index does not count 
  zeros and poles in $\QQ$ but in the real closure $\QQ^c$.
\end{remark}

\begin{remark}
  Sturm's theorem can be seen as an algebraic analogue 
  of the fundamental theorem of calculus. 
  It reduces a $1$--dimensional counting problem on the interval $[a,b]$ 
  to a $0$--dimensional counting problem on the boundary $\{a,b\}$.
  In \sref{sec:Cauchy} we will generalize this to the complex realm, 
  reducing a $2$--dimensional counting problem on a rectangle $\Gamma$ 
  to a $1$--dimensional counting problem on the boundary $\partial\Gamma$.
\end{remark}

\begin{Annotation}
  \textbf{(Invariance)}
  Over a real closed field $\R$ we can strengthen Proposition \ref{prop:RealIndexProperties}(b):
  the Cauchy index satisfies $\Ind_a^b\bigl( f \circ g \bigr) = \Ind_{g(a)}^{g(b)}\bigl( f \bigr)$ 
  for all $f,g \in \R(X)$ where $g$ has no poles in $[a,b]$.

  Assume $f = R/S$ and $g = P/Q$ with $P,Q,R,S \in \R[X]$ such that $\gcd(P,Q)=1$ and $\gcd(R,S)=1$.
  Since $g$ has no poles, $Q$ has no roots in $[a,b]$.  If $(S_0,S_1,\dots,S_n)$ in $\R[X]$ 
  is a Sturm chain for $f = R/S$ on $[a,b]$ with $S_n$ constant, then $(P_0, P_1, \dots, P_n)$ 
  defined by $P_k = Q^m S_k ( P / Q )$ with $m = \max\{\deg S_0,\dots,\deg S_n\}$ 
  is a Sturm chain for $f \circ g = R(P/Q) / S(P/Q)$ with $P_n$ constant.
  As in Theorem \ref{thm:Sturm} we thus find
  \begin{multline*}
    \Ind_a^b\bigl( f \circ g \bigr) 
    = V_a^b\bigl( P_0, P_1, \dots, P_n \bigr) 
    = V_a^b\bigl( S_0 \circ g , S_1 \circ g , \dots, S_n \circ g \bigr)
    = V_{g(a)}^{g(b)}\bigl( S_0, S_1, \dots, S_n \bigr) 
    = \Ind_{g(a)}^{g(b)}\bigl( f \bigr) .
  \end{multline*}
  Contrary to the fractional linear transformations 
  of Proposition \ref{prop:RealIndexProperties}(b),
  the intermediate value property is essential.
  Consider for example $f(x) = \frac{1}{x-2}$ and $g(x) = x^2$ over $\QQ$.
  Then $\Ind_1^2(f \circ g) = 0$ differs from $\Ind_{g(1)}^{g(2)}(f) = 1$.
\end{Annotation}

\subsection{Pseudo-euclidean division} \label{sub:PseudoEuclideanDivision}

Euclidean division works for polynomials over a field.
In \sref{sub:CountingComplexRoots} we consider polynomials $S,P \in \R[Y,X] = \K[X]$ over $\K = \R[Y]$.
To this end we introduce pseudo-euclidean division over an integral ring $\K$:
for all $S,P \in \K[X]$ with $P \ne 0$ there exists a unique pair $Q^*,R^* \in \K[X]$ 
such that $c^d S = P Q^* - R^*$ and $\deg R^* < \deg P$, 
where $c \in \K$ is the leading coefficient of $P$
and $d = \max\{\, 0, 1 + \deg S - \deg P \,\}$.

When working over a field $\F \supset \K$, the leading coefficient $c \ne 0$ is invertible in $\F$,
and we can divide $c^d S = P Q^* - R^*$ by $c^d$ to recover $S = P Q - R$, where $Q = Q^*/c^d$ and $R = R^*/c^d$.
Pseudo-euclidean division may nevertheless be more convenient.
For polynomials in $\QQ[X]$, for example, it is often more efficient 
to clear denominators and to work in $\ZZ[X]$ in order to avoid coefficient swell;
see \cite[\textsection 6.12]{GathenGerhard:2003}.

For Sturm chains it is advantageous to have $c^d S = P Q^* - R^*$ with $d$ even.
In a typical Sturm chain we would expect $\deg S = \deg P + 1$ and thus $d = 2$.
If $d$ happens to be odd, we can multiply $Q^*$ and $R^*$ by $c$ and augment $d$ by $1$.
Starting from $S_0,S_1 \in \K[X]$ we can thus construct a chain 
$S_0,S_1,\dots,S_n \in \K[X]$ with $S_{k+1} = B_k S_{k} - c_k^2 S_{k-1}$
as in Proposition \ref{prop:AlgebraicSturmCriterion}.

\begin{Annotation}
  \textbf{(Pseudo-euclidean division)}
  For every ring $\K$, the degree $\deg \colon \K[X] \to \NN \cup \{-\infty\}$ satisfies:
  \begin{enumerate}
  \item
    $\deg(P+Q) \le \sup\{\, \deg P, \deg Q \,\}$, with equality 
    iff $\deg P \ne \deg Q$ or $\lc(P) + \lc(Q) \ne 0$.
  \item
    $\deg(PQ) \le \deg P + \deg Q$, with equality 
    iff $P = 0$ or $Q = 0$ or $\lc(P) \cdot \lc(Q) \ne 0$.
  \end{enumerate}

  If $\K$ is integral, then $\deg(PQ) = \deg P + \deg Q$ 
  and $\lc(PQ) = \lc(P) \cdot \lc(Q)$ for all $P,Q \in \K[X]^*$,
  and the polynomial ring $\K[X]$ is again integral.
  Moreover, for every $S \in \K[X]$ and $P \in \K[X]^*$ 
  there exists a unique pair $Q,R \in \K[X]$ such that
  $c^d S = P Q - R$ and $\deg R < \deg P$, where 
  $c = \lc(P)$ and $d = \max\{\, 0, 1 + \deg S - \deg P \,\}$.
  
  \textit{Existence:}
  We proceed by induction on $d$.
  If $d=0$, then $\deg S < \deg P$ and $Q=0$ and $R=S$ suffice.
  If $d\ge1$, then we set $M := \lc(S) \cdot X^{\deg S - \deg P}$
  and $\tilde{S} := c S - P M$.
  We see that $\deg(S) = \deg(cS) = \deg(PM)$ and $\lc(cS) = \lc(PM)$,
  whence $\deg\tilde{S} < \deg S$.  By hypothesis, there exists
  $\tilde{Q}, R \in A[X]$ such that $c^{d-1} \tilde{S} = P \tilde{Q} + R$.  
  We conclude that $c^d S = c^{d-1} \tilde{S} + c^{d-1} P M = P Q + R$ 
  with $Q = \tilde{Q} + c^{d-1} M$.

  \textit{Uniqueness:}
  For $PQ+R = PQ'+R'$ with $\deg R < \deg P$ and $\deg R'< \deg P$, 
  we find $P(Q-Q') = R'-R$, whence $\deg P + \deg(Q-Q') = \deg[ P(Q-Q') ] = \deg(R-R') < \deg P$.
  This is only possible for $\deg(Q-Q') < 0$, which means $Q-Q'=0$.
  We conclude that $Q = Q'$ and $R = R'$.
\end{Annotation}


\begin{Annotation}
  \textbf{(Cauchy functions)}
  Sturm's theorem \ref{cor:CountingRealRoots} works miraculously well for polynomials, 
  but like the Budan--Fourier Theorem \ref{thm:DescartesBudanFourier} it is not restricted to polynomials.
  Continuing Remark \ref{rem:DifferentiableFunctions}, 
  if $f \colon [a,b] \to \R$ is $n$-times differentiable, 
  $f^{(n)}$ is of constant sign and each $f^{(k)}$ has the maximal number 
  of $n-k$ distinct zeros, then $f,f',\dots,f^{(n)}$ is a Sturm chain
  in the sense of \eqref{eq:SturmCondition}, whence the zeros of $f$ 
  can then be counted on subintervals and located using \eqref{eq:RootCount}.
  
  More generally, we call $f \colon [a,b] \to \R$ a \emph{Cauchy function}
  if $f$ does not change sign without passing through zero
  and on every interval $I_k = [t_{k-1},t_k]$ of some subdivision 
  $a = t_0 < t_1 < \dots < t_m =b$ the restriction $f_k = f|I_k$ 
  is $n_k$-times differentiable with $f_k^{\smash{(n_k)}}$ of constant sign.
  For example, over $\RR$ every real-analytic function $f \colon [a,b] \to \RR$ is a Cauchy function;
  in fact, quasi-analytic functions suffice, i.e., $C^\infty$ functions such that 
  $f^{(n)}(x_0) = 0$ for some $x_0$ and all $n \ge 0$ implies $f = 0$. 
  Over a real closed field $\R$ one can consider piecewise polynomial functions $f \colon [a,b] \to \R$, 
  or more generally Nash functions \cite[chap.\,8]{BochnakCosteRoy:1998}, 
  i.e., $C^\infty$ functions that are semi-algebraic, for example $f(x) = \sqrt{1+x^2}$.
  
  Sturm's theory extends smoothly: 
  Given Cauchy functions $f,g \colon [a,b] \to \R$, we assume their quotient 
  $f/g$ to be reduced, which means that $f$ and $g$ have no common zeros.
  The Cauchy index $\Ind_a^b(f/g)$ can then be defined as above,
  the inversion formula holds verbatim, and there exists a Sturm chain $S_0=g,S_1=f,S_2,\dots,S_n$ 
  of Cauchy functions with $S_{k-1} + S_{k+1} = Q_k S_k$ for $0 < k < n$ such that $S_n$ is of constant sign on $[a,b]$. 
  One can explicitly construct such a chain if one knows the zeros of $f$ and $g$ in $[a,b]$.
  This underlines that the remarkable fact about polynomials is not so much the 
  existence of Sturm chains but their construction from the euclidean algorithm,
  without prior knowledge of any zeros.
\end{Annotation}

\begin{Annotation}
  \textbf{(Sketch of proof)}
  Every Cauchy function has only a finite number of zeros, 
  see Remark \ref{rem:DifferentiableFunctions}, or perhaps intervals of zeros.
  We can assume that only $S_0$ has zeros in $U_1 = [a,a_1] \cup [b_1,b]$, with $S_0(a_1) \ne 0$ and $S_0(b_1) \ne 0$.
  We choose $u_1,v_1 \in \R$ and define $Q_1 \colon [a,b] \to \R$ by  $Q_1(x) = 0$ for $x \in [a_1,b_1]$, 
  and $Q_1(x) = u_1 (a_1-x)$ for $x \in [a,a_1]$, and $Q_1(x) = v_1 (x-b_1)$ for $x \in [b_1,b]$, 
  such that $S_2 = Q_1 S_1 - S_0$ has no zeros in $U_1$.  To this end it suffices 
  to choose $u_1,v_1$ large enough and to arrange their signs such that 
  $u_1 S_1(a_1) S_0(a_1) < 0$ and $v_1 S_1(b_1) S_0(b_1) < 0$.
  On $[a_1,b_1]$ we have $Q_1 = 0$ and thus $S_2 = -S_0$, 
  in particular $S_1(t) = 0$ implies $S_0(t) S_2(t) < 0$.
  By construction, $S_2$ is again a Cauchy function.
  Continuing like this we can eliminate zeros of $S_1$ that are not bounded by zeros of $S_2$
  to obtain $S_3 = Q_2 S_2 - S_1$.  By induction on the number of (intervals of) zeros we finally reach $S_n$ without zeros.
\end{Annotation}

\begin{Annotation}
  \textbf{(Is there a natural construction?)}
  In the proof sketched above we choose $Q_k$ piecewise linear, 
  which is allowed in our definition of Cauchy functions but remains somewhat artificial.  
  It would be nice to have a more natural construction, for example 
  for real-analytic functions where $Q_k$ should again be real-analytic.
  Moreover, our construction assumes knowledge of the zeros of $S_0,S_1,\dots$, 
  which is contrary to the envisaged application of finding the zeros!
  The ideal model is, of course, euclidean division of polynomials.
\end{Annotation}


\section{Cauchy's theorem for complex polynomials} \label{sec:Cauchy}

We continue to work over a real closed field $\R$ and now
consider its complex extension $\C = \R[i]$ where $i^2=-1$.
In this section we define the algebraic winding number 
and use it to prove Cauchy's theorem (Corollary \ref{cor:CountingComplexRoots}).
To this end we establish the product formula \eqref{eq:RealProductFormula}, 
which seems to be new.  It ensures, for example, that the algebraic winding number
can cope with roots on the boundary, as already emphasized in Theorem \ref{thm:RootCounting}.



\subsection{Real and complex fields} \label{sub:ComplexFields}

Let $\R$ be an ordered field.
For every $x \in \R$ we have $x^2 \ge 0$, whence $x^2+1 > 0$.
The polynomial $X^2 + 1$ is thus irreducible in $\R[X]$,
and the quotient $\C = \R[X]/(X^2+1)$ is a field.
It is denoted by $\C = \R[i]$ with $i^2 = -1$.
Each element $z \in \C$ can be uniquely written as $z = x + y i$ with $x,y \in \R$.  
We can thus identify $\C$ with $\R^2$ via the map $\R^2 \to \C$, 
$(x,y) \mapsto z = x + y i$, and define $\re(z) := x$ and $\im(z) := y$.

Using this notation, addition and multiplication in $\C$ are given by
\begin{align*}
  (x + y i) + (x' + y' i) & = (x+x') + (y+y')i , \\
  (x + y i) \cdot (x' + y' i) & = (x x' - y y') + (x y' + x' y) i .
\end{align*}

The ring automorphism $\R[X] \to \R[X]$, $X \mapsto -X$, fixes $X^2+1$ 
and thus descends to a field automorphism $\C \to \C$
that maps each $z = x + y i$ to its conjugate $\bar{z} = x - y i$.
We have $\re(z) = \frac{1}{2}(z+\bar{z})$ and $\im(z) = \frac{1}{2i}(z-\bar{z})$.
The product $z\bar{z} = x^2 + y^2 \ge 0$ 
vanishes if and only if $z = 0$.  For $z \ne 0$ we thus find 
\[ z^{-1} = \frac{\bar{z}}{z\bar{z}} = \frac{x}{x^2+y^2} - \frac{y}{x^2+y^2} i . \]

If $\R$ is real closed, then every $x \in \R_{\ge0}$ has a square root $\sqrt{x} \in \R_{\ge 0}$. 
For $z \in \C$ we can thus define $\abs{z} := \sqrt{z\bar{z}}$,
which extends the absolute value of $\R$.
For all $u,v \in \C$ we have:
\begin{enumerate}
  \setcounter{enumi}{-1}
\item $\abs{\re(u)} \le \abs{u}$ and $\abs{\im(u)} \le \abs{u}$,
\item $\abs{u} \ge 0$, and $\abs{u} = 0$ if and only if $u=0$,
\item $\abs{u \cdot v} = \abs{u} \cdot \abs{v}$ and $\abs{\bar{u}} = \abs{u}$,
\item $\abs{u + v} \le \abs{u} + \abs{v}$.
\end{enumerate}

All verifications are straightforward.  The triangle inequality (3)
can be derived from the preceding properties as follows.
If $u+v=0$, then (3) follows from (1).  If $u+v \ne 0$, then 
$1 = \frac{u}{u+v} + \frac{v}{u+v}$, and applying (0) and (2) we find
\begin{equation*}
  1 = \re\Bigl(\frac{u}{u+v}\Bigr) + \re\Bigl(\frac{v}{u+v}\Bigr) 
  \le \Bigl|\frac{u}{u+v}\Bigr| + \Bigl|\frac{v}{u+v}\Bigr|
  = \frac{\abs{u}}{\abs{u+v}} + \frac{\abs{v}}{\abs{u+v}}.
\end{equation*}


\subsection{Real and complex variables} \label{sub:ComplexVariables}

Just as we identify $(x,y) \in \R^2$ with $z = x + i y \in \C$, we consider
$\C[Z]$ 
as a subring of $\C[X,Y]$ with $Z = X + i Y$.  
The conjugation on $\C$ extends to a ring automorphism of $\C[X,Y]$ fixing $X$ and $Y$,
so that the conjugate of $Z = X + i Y$ is $\overline{Z} = X - i Y$.
In this sense, $X$ and $Y$ are real variables, whereas $Z$ is a complex variable.

Every polynomial $F \in \C[X,Y]$ can be uniquely decomposed as $F = R + i S$ with $R,S \in \R[X,Y]$,
namely $R = \re F := \frac{1}{2}(F+\overline{F})$ and $S = \im F := \frac{1}{2i}(F-\overline{F})$.
In particular, we thus recover the familiar formulae $X = \re Z$ and $Y = \im Z$.

For $F,G \in \C[X,Y]$ we set $F \circ G := F( \re G, \im G )$.
The map $F \mapsto F \circ G$ is the unique ring endomorphism 
$\C[X,Y] \to \C[X,Y]$ that maps $Z \mapsto G$
and is equivariant with respect to conjugation,
because $Z \mapsto G$ and $\overline{Z} \mapsto \overline{G}$ 
are equivalent to $X \mapsto \re G$ and $Y \mapsto \im G$.

\subsection{The algebraic winding number} \label{sub:AlgebraicWindingNumber}

Given a polynomial $P \in \C[X]$ and two parameters $t_0 < t_1$ in $\R$, 
the map $\gamma \colon [t_0,t_1] \to \C$ defined by $\gamma(t) = P(t)$ 
describes a polynomial path in $\C$.  We define its winding number $\Wind(\gamma)$
to be half the Cauchy index of $\frac{\re P}{\im P}$ on $[t_0,t_1]$:
\[
\cind{[t_0,t_1]}{P} := \textstyle 
\frac{1}{2} \Ind_{t_0}^{t_1}\bigl( \frac{\re P}{\im P} \bigr) .
\]

This definition is geometrically motivated as follows.
Assuming that $\gamma(t) \ne 0$ for all $t \in [t_0,t_1]$,
the winding number $\Wind(\gamma)$ counts the number of turns 
that $\gamma$ performs around $0$.  It changes by $+\half$ 
each time $\gamma$ crosses the real axis in counter-clockwise 
direction, and by $-\half$ if the passage is clockwise.
Our algebraic definition is slightly more comprehensive than 
the geometric one since it does not exclude zeros of $\gamma$.

\begin{definition} \label{def:PiecewisePolynomialLoop}
  Consider a subdivision $0 = t_0 < t_1 < \dots < t_n = 1$ in $\R$ 
  and polynomials $P_1,\dots,P_n \in \C[X]$ that 
  satisfy $P_k(t_k) = P_{k+1}(t_k)$ for $k=1,\dots,n-1$.
  This defines a \emph{piecewise polynomial path} $\gamma \colon [0,1] \to \C$ 
  by $\gamma(t) := P_k(t)$ for $t \in [t_{k-1},t_k]$.
  If $\gamma(a) = \gamma(b)$, then $\gamma$ is called a \emph{closed path} or \emph{loop}.
  Its \emph{winding number} is defined as
  \begin{equation}
    \Wind(\gamma) := \sum_{k=1}^n \cind{[t_{k-1},t_k]}{P_k} .
  \end{equation}
  
  This is well-defined according to Proposition \ref{prop:RealIndexProperties}(a),
  because it depends only on the path $\gamma \colon [0,1] \to \R$ 
  and not on the chosen subdivision of the interval $[0,1]$.
\end{definition}


\subsection{Normalization} \label{sub:Normalization}

The following notation will be convenient.  Given $a,b \in \C$, 
the map $\gamma \colon [0,1] \to \C$ defined by $\gamma(x) = a + x(b-a)$ 
joins $\gamma(0) = a$ and $\gamma(1) = b$ by a straight line segment.
Its image will be denoted by $[a,b] := \gamma([0,1])$. 
For $a \ne b$ we set $\ee{a}{b} := \gamma(\ee{0}{1})$.

For $F \in \C[X,Y]$, we set $\cind{[a,b]}{F} := \Wind(F \circ \gamma)$.
This is the winding number of the path traced 
by $F(z)$ as $z$ runs from $a$ straight to $b$.
According to Proposition \ref{prop:RealIndexProperties}(b),
the reverse orientation yields $\cind{[b,a]}{F} = - \cind{[a,b]}{F}$.

A \emph{rectangle} (with sides parallel to the axes) 
is a subset $\Gamma = [x_0,x_1] \times [y_0,y_1]$
in $\C = \R^2$ with $x_0 < x_1$ and $y_0 < y_1$ in $\R$.  
Its \textit{interior} is $\Int\Gamma = \ee{x_0}{x_1} \times \ee{y_0}{y_1}$.
Its \emph{boundary} $\partial\Gamma$ consists of the four vertices
$a = (x_0,y_0)$, $b = (x_1,y_0)$, $c = (x_1,y_1)$, $d = (x_0,y_1)$,
and the four edges $\ee{a}{b}$, $\ee{b}{c}$, $\ee{c}{d}$, $\ee{d}{a}$ 
between them (see Figure \ref{fig:WindingNumber}).

\begin{definition} \label{def:AlgebraicWindingNumber}
  Given a polynomial $F \in \C[X,Y]$ and a rectangle $\Gamma \subset \C$ as above, we set
  \begin{equation}
    \cind{\partial\Gamma}{F} := \cind{[a,b]}{F} + \cind{[b,c]}{F} + \cind{[c,d]}{F} + \cind{[d,a]}{F} .
  \end{equation}
  Stated differently, we have $\cind{\partial\Gamma}{F} = \Wind(F \circ \gamma)$
  where the path $\gamma \colon [0,1] \to \C$ linearly interpolates between 
  the vertices $\gamma(0) = a$, $\gamma(\nicefrac{1}{4}) = b$, $\gamma(\nicefrac{1}{2}) = c$,
  $\gamma(\nicefrac{3}{4}) = d$, and $\gamma(1) = a$. 
\end{definition}


\begin{lemma}[subdivision] \label{lem:Subdivision}
  Suppose that we subdivide $\Gamma = [x_0,x_2] \times [y_0,y_2]$ 
  \begin{itemize}
  \item horizontally into $\Gamma' = [x_0,x_1] \times [y_0,y_2]$ 
    and $\Gamma'' = [x_1,x_2] \times [y_0,y_2]$, 
  \item or vertically into $\Gamma' = [x_0,x_2] \times [y_0,y_1]$ 
    and $\Gamma'' = [x_0,x_2] \times [y_1,y_2]$,
  \end{itemize}
  where $x_0 < x_1 < x_2$ and $y_0 < y_1 < y_2$. 
  Then $\cind{\partial\Gamma}{F} = \cind{\partial\Gamma'}{F} + \cind{\partial\Gamma''}{F}$.
\end{lemma}

\begin{proof}
  This follows from Definition \ref{def:AlgebraicWindingNumber} by 
  one-dimensional subdivision (Proposition \ref{prop:RealIndexProperties})
  and cancellation of the two internal edges having opposite orientations.
\end{proof}

We will frequently use subdivision in the sequel.
As a first application we use it to establish the normalization \eqref{property:Normalization} 
of the algebraic winding number stated in Theorem \ref{thm:AlgebraicWindingNumber}.

\begin{proposition} \label{prop:Normalization}
  For a linear polynomial $F = Z - z_0$ with $z_0 \in \C$ we find
  \[
  \cind{\partial\Gamma}{F} = \begin{cases}
    1 & \text{if $z_0$ is in the interior of $\Gamma$,} \\
    \nicefrac{1}{2} & \text{if $z_0$ is in one of the edges of $\Gamma$,} \\
    \nicefrac{1}{4} & \text{if $z_0$ is in one of the vertices of $\Gamma$,} \\
    0 & \text{if $z_0$ is in the exterior of $\Gamma$.} 
  \end{cases}
  \]
\end{proposition}

\begin{proof}
  By subdivision, all configurations can be reduced 
  to the case where $z_0$ is a vertex of $\Gamma$.
  By symmetry, translation, and homothety we can assume 
  that $z_0 = a = 0$, $b=1$, $c=1+i$, $d=i$.  
  Here an easy explicit calculation shows that 
  $\cind{\partial\Gamma}{F} = \nicefrac{1}{4}$ by adding 
  \begin{align*}
    \cind{[a,b]}{F} & = \textstyle \cind{[0,1]}{X} = \frac{1}{2}\Ind_0^1(\frac{X}{0}) = 0 , \\
    \cind{[b,c]}{F} & = \textstyle \cind{[0,1]}{1+iX} = \frac{1}{2}\Ind_0^1(\frac{1}{X}) = \frac{1}{4} , \\
    \cind{[c,d]}{F} & = \textstyle \cind{[0,1]}{1+i-X} = \frac{1}{2}\Ind_0^1(\frac{1-X}{1}) = 0 , \text{and} \\
    \cind{[d,a]}{F} & = \textstyle \cind{[0,1]}{i-iX} = \frac{1}{2}\Ind_0^1(\frac{0}{1-X}) = 0 .
    \qedhere
  \end{align*}
\end{proof}

\begin{Annotation}
  \textbf{(Normalization)}
  The factor $\half$ in the definition of the winding number 
  compared to the Cauchy index is chosen so as to achieve 
  the normalization of Proposition \ref{prop:Normalization}.
  It also has a natural geometric interpretation.
  Compare the circle $\S = \{\, z \in \C \colon \abs{z} = 1 \,\}$ 
  with the projective line $\P\R$ of Annotation \ref{anno:WindingNumber}.
  The winding number $\Wind(\gamma)$ of a path $\gamma \colon [0,1] \to \C^*$
  is defined using the map $q \colon \C^* \to \P\R$, $(x,y) \mapsto [x:y]$.
  The quotient map $q$ is the composition of the deformation 
  retraction $r \colon \C^* \to \S$, $z \mapsto z/\abs{z}$,
  and the two-fold covering $p \colon \S \to \P\R$, $(x,y) \mapsto [x:y]$.
  This means that \emph{one} full circle in $\C^*$ maps to \emph{two}
  full circles in $\P\R$. 
\end{Annotation}

\begin{Annotation} 
  \textbf{(Angles)} \label{anno:AverageWindingNumber}
  Proposition \ref{prop:Normalization} generalizes from rectangles to convex polygons,
  and then to arbitrary polygons by subdivision. The only subtlety occurs 
  when $z_0$ is a vertex of the boundary $\partial\Gamma$:
  in general, we find $\cind{\partial\Gamma}{Z} \in \{0, \nicefrac{1}{4}, \nicefrac{1}{2}, \nicefrac{3}{4}, 1 \}$,
  and one can easily construct examples showing that all possibilities are realized:
  \begin{figure}[ht]
    \centering
    \includegraphics[width=\linewidth]{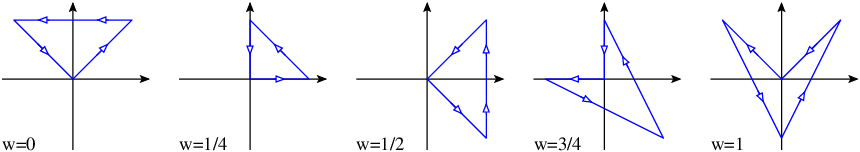}
  \end{figure}

  These examples illustrate how the result depends on 
  the angle at $0$ and its incidence with the real axis.
  The reference to the real axis breaks the rotational symmetry, and so 
  $\Wind(\gamma)$ may differ from $\Wind(c \gamma)$ for some $c \in \C^*$. 
  Over $\CC$ the average value $\mind(\gamma) = \int_0^1 \Wind( e^{2 \pi i t} \gamma ) \, \dd t \in [0,1]$ 
  measures the angle at $0$. For $\C = \R[i]$ over a real closed field $\R$ we can likewise define 
  $\mind(\gamma) := \lim_{N \to \infty} \frac{1}{N} \sum_{k=0}^{N-1} \Wind( e^{2 \pi i / N} \gamma ) \in \RR$
  for every piecewise polynomial loop $\gamma \colon [0,1] \to \C$.
  Measuring angles in this way does not follow the paradigm of effective calculation 
  emphasized here, but the definition of $\mind(\gamma)$ may still be useful in some other context.
  For the purpose of this article, however, it is only an amusing curiosity
  and will not be further developed.
\end{Annotation}



\subsection{The product formula} \label{sub:ProductFormula}

The product of two polynomials $F = P + i Q$ 
and $G = R + i S$ with $P,Q,R,S \in \R[X]$
is given by $F G = (P R - Q S) + i (P S + Q R)$.
The following result relates the Cauchy indices of $\frac{P}{Q}$ 
and $\frac{R}{S}$ to that of $\frac{P R - Q S}{P S + Q R}$.

\begin{theorem}[product formula] \label{thm:RealProductFormula}
  For all $P,Q,R,S \in \R[X]$ and $a,b \in \R$ we have 
  \begin{equation}
    \label{eq:RealProductFormula}
    \Ind_a^b\Bigl(\frac{P R - Q S}{P S + Q R}\Bigr)
    = \Ind_a^b\Bigl(\frac{P}{Q}\Bigr) + \Ind_a^b\Bigl(\frac{R}{S}\Bigr) 
    - V_a^b\Bigl( 1, \frac{P}{Q} + \frac{R}{S} \Bigr) .
  \end{equation} 
\end{theorem}

We remark that in the last term we have $\frac{P}{Q} + \frac{R}{S} = \frac{ P S + Q R }{ Q S } = \frac{\im(F G)}{\im(F)\im(G)}$, whence
\begin{equation} \label{eq:CorrectionTerms}
  \textstyle 
  V_a^b\bigl( 1, \frac{P}{Q} + \frac{R}{S} \bigr)
  = \half \bigl[ \sign\bigl( \frac{ P S + Q R }{ Q S } \mid X \mapsto b \bigr) 
  - \sign\bigl( \frac{ P S + Q R }{ Q S } \mid X \mapsto a \bigr) \bigr] .
\end{equation}
If $a$ or $b$ is a pole, this is evaluated using the convention $\sign(\infty) = 0$.

For $(P=0,Q=1)$ or $(R=0,S=1)$ or $(P=S,Q=R)$, 
the product formula \eqref{eq:RealProductFormula} reduces to 
the inversion formula \eqref{eq:InversionGeneral}.
The proof of the general case follows the same lines.

\begin{proof}
  Equation \eqref{eq:RealProductFormula} trivially holds 
  in the degenerate cases where $Q=0$, $S=0$, or $P S + Q R = 0$; 
  we can thus concentrate on the generic case where $Q, S, P S + Q R \in \R[X]^*$.
  We can further assume $\gcd(P,Q) = \gcd(R,S) = 1$.
  Since \eqref{eq:RealProductFormula} is additive 
  with respect to subdivision of the interval $[a,b]$,
  we can assume that $[a,b]$ contains at most one pole.

  \textit{Global analysis away from poles:}
  Suppose that $[a,b]$ does not contain zeros of $Q$, $S$, or $P S + Q R$.
  Then all three indices in \eqref{eq:RealProductFormula} vanish in the absence of poles,
  and the intermediate value property ensures that 
  $Q$, $S$, and $P S + Q R$ are of constant sign on $[a,b]$,
  whence $V_a^b\bigl(1,\frac{P S + Q R}{Q S}\bigr) = 0$
  and Equation \eqref{eq:RealProductFormula} holds.
  
  \textit{Local analysis at a pole:}
  Suppose that $[a,b]$ contains a pole.
  Subdividing, if necessary, we can assume that this pole is either $a$ or $b$.  
  Applying the symmetry $X \mapsto a+b-X$, if necessary, we can assume that the pole is $a$.
  We thus have $V_a^b = \half \sign( \frac{P}{Q} + \frac{R}{S} \mid X \mapsto b )$
  and $Q$, $S$, $P S + Q R$ are of constant sign on $\ei{a}{b}$.
  Applying the symmetry $(P,Q,R,S) \mapsto (P,-Q,R,-S)$, if necessary, 
  we can assume that $\frac{P}{Q} + \frac{R}{S} > 0$ on $\ei{a}{b}$, whence $V_a^b = +\half$.
  Based on these preparations we distinguish three cases.

  \textit{First case.}
  Suppose first that either $Q(a) = 0$ or $S(a) = 0$. 
  Applying the symmetry $(P,Q,R,S) \mapsto (R,S,P,Q)$, 
  if necessary, we can assume that $Q(a) = 0$ and $S(a) \ne 0$.
  Then $P S + Q R$ does not vanish at $a$, whence 
  $\Ind_a^b\bigl(\frac{P R - Q S}{P S + Q R}\bigr) 
  = \Ind_a^b\bigl(\frac{R}{S}\bigr) = 0$.  
  Since $\frac{P}{Q} + \frac{R}{S} > 0$ on $\ei{a}{b}$
  we have $\lim_a^+ \frac{P}{Q} = +\infty$,
  whence $\Ind_a^b\bigl(\frac{P}{Q}\bigr) = +\half$
  and Equation \eqref{eq:RealProductFormula} holds.
  
  \textit{Second case.}
  Suppose that $P S + Q R$ vanishes at $a$, but $Q(a) \ne 0$ and $S(a) \ne 0$.
  Then $\Ind_a^b\bigl(\frac{P}{Q}\bigr) = \Ind_a^b\bigl(\frac{R}{S}\bigr) = 0$,
  and we only have to study the pole of 
  \begin{equation} \label{eq:ProductQuotient}
    \frac{P R - Q S}{P S + Q R} =
    \frac{ \frac{P}{Q} \cdot \frac{R}{S} - 1 }{ \frac{P}{Q} + \frac{R}{S} } .
  \end{equation}

  At $a$ the denominator vanishes and the numerator is negative:
  \[ \textstyle 
  \frac{P(a)}{Q(a)} + \frac{R(a)}{S(a)} = 0, \quad\text{whence}\quad
  \frac{P(a)}{Q(a)} \cdot \frac{R(a)}{S(a)} - 1 = - \frac{P^2(a)}{Q^2(a)} - 1 < 0 .
  \]
  This implies $\lim_a^+ \frac{P R - Q S}{P S + Q R} = -\infty$,
  whence $\Ind_a^b\bigl(\frac{P R - Q S}{P S + Q R}\bigr) = -\half$
  and Equation \eqref{eq:RealProductFormula} holds.

  \textit{Third case.}
  Suppose that $a$ is a common pole of $\frac{P}{Q}$ and 
  $\frac{R}{S}$, whence also of $\frac{P R - Q S}{P S + Q R}$.
  Since $\frac{P}{Q} + \frac{R}{S} > 0$ on $\ei{a}{b}$, we have 
  $\lim_a^+ \frac{P}{Q} = +\infty$ or $\lim_a^+ \frac{R}{S} = +\infty$.
  Equation \eqref{eq:ProductQuotient} implies that 
  $\lim_a^+ \bigl(\frac{P R - Q S}{P S + Q R}\bigr)
  = \lim_a^+ \bigl(\frac{P}{Q}\bigr) \cdot \lim_a^+ \bigl(\frac{R}{S}\bigr)$,
  whence Equation \eqref{eq:RealProductFormula} holds.
\end{proof}

The product formula \eqref{eq:RealProductFormula} entails 
the multiplicativity \eqref{property:Multiplicativity} 
stated in Theorem \ref{thm:AlgebraicWindingNumber}.

\begin{corollary}[multiplicativity of winding numbers] \label{cor:Multiplicativity}
  We have $\Wind( \gamma_1 \cdot \gamma_2 ) = \Wind( \gamma_1) + \Wind( \gamma_2 )$ 
  for all piecewise polynomial loops $\gamma_1,\gamma_2 \colon [0,1] \to \C$ whose vertices are not mapped to $0$.
\end{corollary}

\begin{proof}
  On a common subdivision $0 = t_0 < t_1 < \dots < t_n = 1$, both $\gamma_1,\gamma_2$ 
  are polynomials on each interval.  There exist $F_k = P_k + i Q_k$ 
  and $G_k = R_k + i S_k$ with $P_k,Q_k,R_k,S_k \in \R[X]$ such that 
  $\gamma_1(t) = F_k(t)$ and $\gamma_2(t) = G_k(t)$ for all $t \in [t_{k-1},t_k]$.
  By excluding zeros of $\gamma_1,\gamma_2$ on the vertices $t_0,t_1,\dots,t_n$, 
  we ensure that $\frac{P_k}{Q_k}(t_k) = \frac{P_{k+1}}{Q_{k+1}}(t_k)$
  and $\frac{R_k}{S_k}(t_k) = \frac{R_{k+1}}{S_{k+1}}(t_k)$ for all $k=1,\dots,n-1$. 
  Since both paths $\gamma_1,\gamma_2$ are closed, this also holds for $k=n$
  with the understanding that $F_{n+1} = F_1$ and $G_{n+1} = G_1$.
  The desired result $\Wind( \gamma_1 \cdot \gamma_2 ) = \Wind( \gamma_1) + \Wind( \gamma_2 )$ 
  now follows from the product formula \eqref{eq:RealProductFormula},
  because at each vertex $t_k$ the incoming and the outgoing boundary term 
  from \eqref{eq:CorrectionTerms} cancel each other.
\end{proof}

\begin{corollary} \label{cor:ComplexProductFormula}
  Let $\gamma \colon [0,1] \to \R^2$ be a piecewise polynomial loop.
  If $F,G \in \C[X,Y]$ do not vanish at any of the vertices of $\gamma$, 
  then $\cind{\gamma}{ F \cdot G } = \cind{\gamma}{F} + \cind{\gamma}{G}$.
  \qed
\end{corollary}

More specifically, if $F,G$ do not vanish at any of the vertices of the rectangle $\Gamma \subset \R^2$,
then $\cind{\partial\Gamma}{ F \cdot G } = \cind{\partial\Gamma}{F} + \cind{\partial\Gamma}{G}$.

\begin{remark}
  The corollary allows zeros of $F$ or $G$ on $\gamma$ but excludes zeros on the vertices. 
  This is not an artefact of our proof, but inherent to the algebraic winding number.
  As an illustration consider $\Gamma = [0,1] \times [0,1]$ and $H_s = Z \cdot (Z - 2 - i s)$.
  The root $z_1=0$ lies on a vertex of $\Gamma$ while the other root $z_2=2+is$ is outside of $\Gamma$.
  In particular, we have $\cind{\partial\Gamma}{Z} = \nicefrac{1}{4}$ and  $\cind{\partial\Gamma}{Z - 2 - i s} = 0$.
  A little calculation shows that $\cind{\partial\Gamma}{H_{1}} = 0$ and
  $\cind{\partial\Gamma}{H_{0}} = \nicefrac{1}{4}$ and $\cind{\partial\Gamma}{H_{-1}} = \nicefrac{1}{2}$,
  whence $\cind{\partial\Gamma}{H_s}$ is not multiplicative.
\end{remark}

\begin{Annotation}
  \textbf{(Roots on vertices)} 
  Roots on vertices are special because our arbitrary reference to the real axis 
  breaks the rotational symmetry, as illustrated in Annotation \ref{anno:AverageWindingNumber}.
  The average winding number $\mind(\gamma)$ of a piecewise polynomial path
  $\gamma \colon [0,1] \to \C$ repairs this defect by restoring rotational symmetry, 
  such that $\mind(\gamma_1 \gamma_2) = \mind(\gamma_1) + \mind(\gamma_2)$ 
  even if zeros happen to lie on vertices.  For every polynomial $F \in \C[Z]^*$ 
  and every polygonal domain $\Gamma \subset \C$, the average winding number 
  $\mind(F|\partial\Gamma)$ thus counts the number of roots of $F$ in $\Gamma$.
  Here each root is counted with $\alpha$ times its multiplicity,
  where $\alpha \in [0,1]$ measures the angle of $\Gamma$ at this root.
  For example, $\alpha \in \{1,\nicefrac{1}{2},\nicefrac{1}{4},0\}$ if $\Gamma$ is a rectangle
  and the zero lies, respectively, in $\Int\Gamma$, an edge, a vertex, or outside of $\Gamma$.
\end{Annotation}

\begin{corollary} \label{cor:CountingComplexRoots}
  Consider a split polynomial $F = (Z-z_1) \cdots (Z-z_n)$ in $\C[Z]$.
  If $F$ does not vanish at any vertex of $\gamma$,
  then $\Wind( F \circ \gamma ) = \sum_{k=1}^n \Wind( \gamma - z_k )$.
  \qed
\end{corollary}

More specifically, if $F$ does not vanish at any vertex of the rectangle $\Gamma \subset \C$,
then $\cind{\partial\Gamma}{F}$ counts the number of zeros in $\Gamma$. 
Each zero in the interior of $\Gamma$ is counted with its multiplicity, whereas 
each zero in an edge of $\partial\Gamma$ is counted with half its multiplicity.

\begin{Annotation}
  \textbf{(Hypotheses)}
  In this section we have barely used the hypothesis that $\R$ be real closed.
  The intermediate value property is essential only for 
  the product formula \eqref{eq:RealProductFormula}, 
  where we use it for the denominators $Q,S,P S + Q R \in \R[X]^*$.
  One might suspect that if the polynomials $F,G \in \C[X]$ split, then the product 
  formula can be applied and Corollaries \ref{cor:ComplexProductFormula} and 
  \ref{cor:CountingComplexRoots} hold independently of any further hypothesis on $\R$.
  This independence could be forced if we changed our definition of the winding number in \sref{sub:AlgebraicWindingNumber}
  from the Cauchy index $\cind{[t_0,t_1]}{P} := \frac{1}{2} \Ind_{t_0}^{t_1}\bigl( \frac{\re P}{\im P} \bigr)$
  to the Sturm index $\cind{[t_0,t_1]}{P} := \frac{1}{2} \Sturm_{t_0}^{t_1}\bigl( \frac{\re P}{\im P} \bigr)$.
  Both coincide over a real closed field, but the latter depends only 
  on the given coefficients and is independent of the ambient field $\R$.
\end{Annotation}

\begin{remark}
  If we assume that $\C$ is algebraically closed, then \emph{every} polynomial 
  $F \in \C[Z]^*$ splits into linear factors as required in Corollary \ref{cor:CountingComplexRoots}.
  So if you prefer some other existence proof for the roots,
  then you may skip the next section and still benefit 
  from root location (Theorem \ref{thm:RootLocation}).
  This seems to be the point of view adopted by Cauchy \cite{Cauchy:1831,Cauchy:1837}
  in 1831/37, which may explain why he did not attempt to use his index 
  for a constructive proof of the Fundamental Theorem of Algebra.
  (In 1820 he had already given a non-constructive proof; see \sref{ssub:Analysis}.)
  In 1836 Sturm and Liouville \cite{SturmLiouville:1836,Sturm:1836} proposed 
  to extend Cauchy's approach so as to obtain an algebraic existence proof.
  This is our aim in the next section. 
\end{remark}


\section{The Fundamental Theorem of Algebra} \label{sec:FTA}

In the preceding section we have constructed the algebraic winding number 
and derived its multiplicativity. We will now show its homotopy invariance 
and thus complete the real-algebraic proof of the Fundamental Theorem of Algebra.
The geometric idea goes back to Gauss' doctoral dissertation (see \sref{sub:Gauss}),
but the algebraic proof seems to be new.  


\subsection{Counting complex roots} \label{sub:CountingComplexRoots}

The following algebraic method for counting complex roots is the counterpart 
of Sturm's theorem for counting real roots (\sref{sub:CountingRealRoots}).

\begin{theorem}[root counting] \label{thm:CountingComplexRoots}
  Consider a polynomial $F \in \C[Z]^*$ and a rectangle $\Gamma \subset \C$
  such that $F$ does not vanish at any of the vertices of $\Gamma$.
  Then the algebraic winding number $\cind{\partial\Gamma}{F}$ 
  counts the number of zeros of $F$ in $\Gamma$:
  each zero in the interior of $\Gamma$ is counted with its multiplicity, whereas
  each zero in an edge of $\partial\Gamma$ is counted with half its multiplicity.
\end{theorem}

\begin{proof}
  We factor $F = (Z-z_1) \cdots (Z-z_m) G$ with $z_1,\dots,z_m \in \Gamma$ 
  such that $G  \in \C[Z]^*$ has no zeros in $\Gamma$.  
  Then $\cind{\partial\Gamma}{G} = 0$ according to Lemma \ref{lem:GlobalIndexZero} below.
  The assertion now follows from normalization (Proposition \ref{prop:Normalization})
  and the product formula (Corollary \ref{cor:ComplexProductFormula}).
\end{proof}

\begin{Annotation}
  \textbf{(Hypotheses)}
  As in Sturm's theorem, Corollary \ref{cor:CountingRealRoots}, 
  the intermediate value property of $\R$ is essential. 
  As a counterexample consider $\R = \QQ$ and $\C = \QQ[i]$.
  The winding number of $F = Z^2 - i$ in $\C[Z]$ with respect to 
  $\Gamma = [0,1] \times [0,1] \subset \C$ is $\cind{\partial\Gamma}{F} = 1$.
  This corresponds to the zero $\frac{1}{2}\sqrt{2} + \frac{i}{2}\sqrt{2}$.
  Here the winding number does not count zeros in $\QQ[i]$ but in $\QQ^c[i]$. 
\end{Annotation}

The crucial point is to show that $\cind{\partial\Gamma}{F} = 0$
whenever $F$ has no zeros in $\Gamma$, or by contraposition, that
$\cind{\partial\Gamma}{F} \ne 0$ implies that $F$ vanishes at some point in $\Gamma$.


\begin{lemma}[local version] \label{lem:LocalIndexZero}
  If $F \in \C[X,Y]$ satisfies $F(x,y) \ne 0$ for some point $(x,y) \in \R^2$,
  then there exists $\delta>0$   such that $\cind{\partial\Gamma}{F} = 0$ for every 
  $\Gamma \subset [x-\delta,x+\delta] \times [y-\delta,y+\delta]$.
\end{lemma}

\begin{Annotation}
  A proof could be improvised as follows.
  Suppose first that $\im F(x,y) > 0$.  By continuity 
  there exists $\delta > 0$ such that $\im F > 0$ on the rectangle 
  $U = [x-\delta,x+\delta] \times [y-\delta,y+\delta]$.
  For every $\Gamma \subset U$ we then have $\cind{\partial\Gamma}{F} = 0$.
  The case $\im F(x,y) < 0$ is analogous.
  If $\im F(x,y) = 0$, then our hypothesis ensures that $\re F(x,y) \ne 0$. 
  Again there exists $\delta > 0$ such that $\re F \ne 0$ on the rectangle 
  $U = [x-\delta,x+\delta] \times [y-\delta,y+\delta]$.
  Now Corollary \ref{cor:ComplexProductFormula} shows that 
  $\cind{\partial\Gamma}{F} = \cind{\partial\Gamma}{iF} = 0$ as in the first case.
  The following detailed proof makes the choice 
  of $\delta$ explicit and thus avoids case distinctions.
\end{Annotation}

\begin{proof}
  Let us make the standard continuity argument explicit.
  For all $s,t \in \R$ we have $F(x+s,y+t) = a + \sum_{j+k \ge 1} a_{jk} s^j t^k$ 
  with $a = F(x,y) \ne 0$ and certain coefficients $a_{jk} \in \C$.
  We set $M := \max \sqrt[j+k]{\abs{a_{jk}/a}}$, so that 
  $\abs{a_{jk}} \le \abs{a} \cdot M^{j+k}$. 
  For $\delta := \frac{1}{4M}$ and $\abs{s},\abs{t} \le \delta$ we find
  \begin{equation} \label{eq:LocalEstimate}
    \Bigl| \sum_{j+k \ge 1} a_{jk} s^j t^k \Bigr|
    \le \sum_{n \ge 1} \sum_{j+k = n} \abs{a} \cdot M^{j+k} \cdot \abs{s}^j \cdot \abs{t}^k
    \le \abs{a} \sum_{n \ge 1} (n+1) \textstyle \bigl(\frac{1}{4}\bigr)^n 
    = \frac{7}{9} \abs{a} . 
  \end{equation}
  This shows that $F$ does not vanish in
  $U := [x-\delta,x+\delta] \times [y-\delta,y+\delta]$.
  Corollary \ref{cor:ComplexProductFormula} ensures that 
  $\cind{\partial\Gamma}{F} = \cind{\partial\Gamma}{cF}$ for 
  every rectangle $\Gamma \subset U$ and every constant $c \in \C^*$.
  Choosing $c=i/a$ we can assume that $F(x,y) = i$.  
  The estimate \eqref{eq:LocalEstimate} then shows that $\im F > 0$ on $U$, 
  whence $\cind{\partial\Gamma}{F} = 0$ for every rectangle $\Gamma \subset U$.
\end{proof}

While the preceding local lemma uses only continuity 
of polynomials and thus holds over every ordered field,
the following global version requires the field $\R$ to be real closed.

\begin{lemma}[global version] \label{lem:GlobalIndexZero}
  Let $\Gamma = [x_0,x_1] \times [y_0,y_1]$ be a rectangle in $\R^2$.
  If the polynomial $F \in \C[X,Y]$ satisfies $F(x,y) \ne 0$ for all 
  $(x,y) \in \Gamma$, then $\cind{\partial\Gamma}{F} = 0$.
\end{lemma}

We remark that over the real numbers $\RR$, a short proof can be given as follows:

\begin{proof}[Proof of Lemma \ref{lem:GlobalIndexZero} for the case $\R = \RR$, using compactness.]
  The rectangle $\Gamma$ is covered by open sets $U(x,y) = \ee{x-\delta}{x+\delta} \times \ee{y-\delta}{y+\delta}$
  as in Lemma \ref{lem:LocalIndexZero}, where $(x,y)$ ranges over $\Gamma$ and $\delta > 0$ depends on $(x,y)$.
  Compactness of $\Gamma$ ensures that there exists $\lambda>0$, 
  called a Lebesgue number of the cover, such that 
  every rectangle $\Gamma' \subset \Gamma$ of diameter $< \lambda$ 
  is contained in $U(x,y)$ for some $(x,y) \in \Gamma$.

  For all subdivisions $x_0 = s_0 < s_1 < \dots < s_m = x_1$ 
  and $y_0 = t_0 < t_1 < \dots < t_n = y_1$, Lemma \ref{lem:Subdivision} ensures that 
  $\cind{\partial\Gamma}{F} = \sum_{j=1}^m \sum_{k=1}^n \cind{\partial\Gamma_{jk}}{F}$
  where $\Gamma_{jk} = [s_{j-1},s_j] \times [t_{k-1},t_k]$.
  For $s_j = x_0 + j \frac{x_1-x_0}{m}$ and $t_k = y_0 + k \frac{y_1-y_0}{n}$
  with $m,n$ sufficiently large, each $\Gamma_{jk}$ has diameter $< \lambda$, so
  Lemma \ref{lem:LocalIndexZero} implies that $\cind{\partial\Gamma_{jk}}{F}=0$ 
  for all $j,k$, whence $\cind{\partial\Gamma}{F}=0$.
\end{proof}

The preceding compactness argument applies only to $\CC = \RR[i]$ 
over the field $\RR$ of real numbers (\sref{sub:RealNumbers})
and not to an arbitrary real closed field (\sref{sub:RealClosedFields}).
In particular, it is no longer elementary in the sense that 
it uses a second-order property (\sref{sub:ElementaryTheory}).
We therefore provide an elementary real-algebraic proof using Sturm chains:

\begin{proof}[Algebraic proof of Lemma \ref{lem:GlobalIndexZero}, using Sturm chains]
  Each $F \in \C[X,Y]$ can be written as $F = \sum_{k=0}^m f_k X^k$ 
  with $f_k \in \C[Y]$.  In this way we consider $\R[X,Y] = \R[Y][X]$ 
  as a polynomial ring in one variable $X$ over $\R[Y]$.
  We can reduce  $\frac{\re F}{\im F} = \frac{S_1}{S_0}$ such that 
  $S_0,S_1 \in \R[X,Y]$ satisfy $\gcd(S_0,S_1) = 1$ in $\R(Y)[X]$.
  Pseudo-euclidean division in $\R[Y][X]$, as explained in \sref{sub:PseudoEuclideanDivision},
  produces a chain $(S_0,\dots,S_n)$ with $S_{k+1} = Q_k S_k - c_k^2 S_{k-1}$ 
  for some $Q_k \in \R[Y][X]$ and $c_k \in \R[Y]^*$ such that $\deg_X S_{k+1} < \deg_X S_{k}$.
  After $n$ iterations we end up with $S_{n+1} = 0$ and $S_n \in \R[Y]^*$.
  (If $\deg_X S_n > 0$, then $\gcd(S_0,S_1)$ in $\R(Y)[X]$ would be of positive degree.)

  \textit{Regular case.}
  Assume first that $S_n \in \R[Y]^*$ does not vanish at any point $y \in [y_0,y_1]$.
  Proposition \ref{prop:AlgebraicSturmCriterion} ensures that 
  for each $y \in [y_0,y_1]$ specializing $(S_0,\dots,S_n)$ 
  in $Y \mapsto y$ yields a Sturm chain in $\R[X]$.
  Likewise, for each $x \in [x_0,x_1]$, specializing $(S_0,\dots,S_n)$ in 
  $X \mapsto x$ yields a Sturm chain in $\R[Y]$ with respect to the interval $[y_0,y_1]$.
  In the sum over all four edges of $\Gamma$, 
  all contributions cancel each other in pairs:
  \begin{align*}
    2\cind{\partial\Gamma}{F} = & \textstyle
    + \Ind_{x_0}^{x_1}\bigl( \frac{\re F}{\im F} \bigm| Y \mapsto y_0 \bigr) 
    + \Ind_{y_0}^{y_1}\bigl( \frac{\re F}{\im F} \bigm| X \mapsto x_1 \bigr) 
    \\ & \textstyle
    + \Ind_{x_1}^{x_0}\bigl( \frac{\re F}{\im F} \bigm| Y \mapsto y_1 \bigr)
    + \Ind_{y_1}^{y_0}\bigl( \frac{\re F}{\im F} \bigm| X \mapsto x_0 \bigr) 
    \\ = & 
    + V_{x_0}^{x_1}\bigl( S_0,\dots,S_n \bigm| Y \mapsto y_0 \bigr) 
    + V_{y_0}^{y_1}\bigl( S_0,\dots,S_n \bigm| X \mapsto x_1 \bigr) 
    \\ & 
    + V_{x_1}^{x_0}\bigl( S_0,\dots,S_n \bigm| Y \mapsto y_1 \bigr) 
    + V_{y_1}^{y_0}\bigl( S_0,\dots,S_n \bigm| X \mapsto x_0 \bigr) 
    = 0 .
  \end{align*}

  \textit{Singular case.}  
  In general we have to cope with a finite set $\Y \subset [y_0,y_1]$ of zeros of $S_n$.  
  We can change the roles of $X$ and $Y$ and apply pseudo-euclidean division
  in $\R[X][Y]$; this leads to a finite set of zeros $\X \subset [x_0,x_1]$.
  We obtain a finite set $\Z = \X \times \Y$ of singular points in $\Gamma$,
  where both chains fail.  (These points are potential zeros of $F$.)

  \begin{figure}[ht]
    \centering
    \includegraphics[height=20ex]{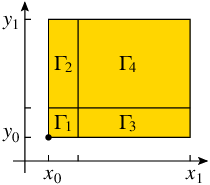}
    \caption{Isolating a singular point $(x_0,y_0)$ 
      within $\Gamma = [x_0,x_1] \times [y_0,y_1]$}
  \end{figure}

  By subdivision and symmetry we can assume that $(x_0,y_0)$ is the only 
  singular point in our rectangle $\Gamma = [x_0,x_1] \times [y_0,y_1]$.
  By hypothesis, $F$ does not vanish in $(x_0,y_0)$, so we can apply 
  Lemma \ref{lem:LocalIndexZero} to $\Gamma_1 = [x_0,x_0+\delta] \times [y_0,y_0+\delta]$ 
  with $\delta>0$ sufficiently small such that $\cind{\partial\Gamma_1}{F} = 0$.
  The remaining three rectangles $\Gamma_2 = [x_0,x_0+\delta] \times [y_0+\delta,y_1]$,
  $\Gamma_3 = [x_0+\delta,x_1] \times [y_0,y_0+\delta]$, and $\Gamma_4 = [x_0+\delta,x_1] \times [y_0+\delta,y_1]$
  do not contain any singular points, so that $\cind{\partial\Gamma_j}{F} = 0$ by appealing to the regular case.  

  Summing over all sub-rectangles we conclude that $\cind{\partial\Gamma}{F} = 0$.
\end{proof}

\begin{Annotation}
  \textbf{(Resultant)}
  The construction of the chain $(S_0,S_1,\dots,S_n)$ in $\R[Y][X]$ carried out in the proof 
  decreases the degree in $X$ but usually increases the degree in $Y$.  The final term $S_n$ 
  is a crude form of the resultant of $S_0$ and $S_1$.  We are rather careless about degrees here, 
  and the usual approach via (sub)resultants would give much better control (\sref{sub:SubresultantAlgorithm}).
  The crucial point in the proof, however, is that 
  we can specialize $(S_0,S_1,\dots,S_n)$ in either $X$ or $Y$
  and obtain a Sturm chain in the remaining variable, 
  by appealing to Proposition \ref{prop:AlgebraicSturmCriterion}.
  For subresultants a similar double specialization argument 
  is less obvious and deserves further study.
\end{Annotation}

\begin{Annotation}
  \textbf{(Counting zeros and poles of rational functions)}
  We have focused on polynomials $F \in \C[Z]$, but 
  Definition \ref{def:AlgebraicWindingNumber} of the algebraic winding number 
  and the product formula of Corollary \ref{cor:ComplexProductFormula}
  immediately extend to rational functions $F \in \C(Z)$.
  It is then an easy matter to establish the following generalization:
  
  \begin{Theorem}
    Consider a rational function $F \in \C(Z)$ and a domain $\Gamma \subset \C$ 
    with piecewise rational boundary $\partial\Gamma$.  If $F$ has neither zeros nor poles
    at the vertices of $\partial\Gamma$, then $\cind{\partial\Gamma}{F}$ counts, 
    with multiplicity, the number of zeros minus the number of poles of $F$ in $\Gamma$.
    Boundary points count for one half.
    \qed
  \end{Theorem}
  
\end{Annotation}

\subsection{Homotopy invariance} \label{sub:HomotopyInvariance}

We consider piecewise  polynomial loops $\gamma_0,\gamma_1 \colon [0,1] \to \C^*$.
A \emph{homotopy} between $\gamma_0$ and $\gamma_1$ is a map $F \colon [0,1] \times [0,1] \to \C^*$ 
with $F(0,t) = \gamma_0(t)$ and $F(1,t) = \gamma_1(t)$ as well as $F(s,0) = F(s,1)$ for all $s,t \in [0,1]$.
We also require that $F$ be \emph{piecewise polynomial}, which means that for 
some subdivision $0 = s_0 < s_1 < \dots < s_m = 1$ and $0 = t_0 < t_1 < \dots < t_n = 1$, 
the map $F$ is polynomial on each $\Gamma_{jk} = [s_{j-1},s_j] \times [t_{k-1},t_k]$.
We can now prove the homotopy invariance \eqref{property:HomotopyInvariance}
stated in Theorem \ref{thm:AlgebraicWindingNumber}.

\begin{theorem} \label{thm:HomotopyInvariance}
  We have $\Wind( \gamma_0 ) = \Wind( \gamma_1 )$ whenever 
  the loops $\gamma_0,\gamma_1$ are homotopic in $\C^*$.
\end{theorem}

\begin{proof}
  On $\Gamma = [0,1] \times [0,1]$ we have $\cind{\partial\Gamma}{F} = \Wind(\gamma_0 ) - \Wind(\gamma_1)$.
  This follows from our hypothesis that $F(s,0) = F(s,1)$ for all $s\in [0,1]$,
  so these two opposite edges cancel each other.
  Subdivision as above yields $\cind{\partial\Gamma}{F} = \sum_{jk} \cind{\partial\Gamma_{jk}}{F}$
  according to Lemma \ref{lem:Subdivision}.
  Since $F$ has no zero, Lemma \ref{lem:GlobalIndexZero} 
  ensures that $\cind{\partial\Gamma_{jk}}{F} = 0$ for all $j,k$.
\end{proof}

As a consequence, the winding number $\cind{\partial\Gamma}{F_t}$
does not change if we deform $F_0$ to $F_1$ avoiding zeros on $\partial\Gamma$.
To make this precise, we consider $F \in \C[Z,T]$; for each $t \in [0,1]$ 
we denote by $F_t$ the polynomial in $\C[Z]$ obtained by specializing $T \mapsto t$.

\begin{corollary} \label{cor:HomotopyInvariance}
  Suppose that $F \in \C[Z,T]$ is such that for each $t \in [0,1]$ 
  the polynomial $F_t \in \C[Z]$ has no zeros on $\partial\Gamma$.  
  Then $\cind{\partial\Gamma}{F_0} = \cind{\partial\Gamma}{F_1}$.
  \qed
\end{corollary}

%
%

\begin{remark} \label{rem:HomotopyEquivalence}
  We have deduced homotopy invariance from the crucial Lemma \ref{lem:GlobalIndexZero} 
  saying that $\cind{\partial\Gamma}{F} = 0$ whenever $F$ has no zeros in $\Gamma$.
  Both statements are in fact equivalent. After translation we can assume $(0,0) \in \Gamma$.
  The homotopy $F_t(X,Y) = F(tX,tY)$ deforms $F_1 = F$ to the constant $F_0 = F(0,0)$.
  If $F$ has no zeros in $\Gamma$, then $F_t$ has no zeros on the boundary $\partial\Gamma$,
  and homotopy invariance implies $\cind{\partial\Gamma}{F_1} = \cind{\partial\Gamma}{F_0} = 0$.
\end{remark}

Homotopy invariance implies that small perturbations do not change the winding
number and hence not the number of zeros.  Rouch\'e's theorem makes this explicit.

\begin{corollary}[Rouch\'e's theorem] \label{cor:Rouche}
  Let $F,G \in \C[Z]$ be two complex polynomials such that 
  $\abs{F(z)} > \abs{G(z)}$ for all $z \in \partial\Gamma$.
  Then $F$ and $F+G$ have the same number of zeros in $\Gamma$.
\end{corollary}

\begin{proof}
  For $F_t = F + t G$ we find $\abs{F_t} \ge \abs{F} - t\abs{G} > 0$ on $\partial\Gamma$ for all $t \in [0,1]$.
  By homotopy invariance (Corollary \ref{cor:HomotopyInvariance}) $F_0 = F$ and $F_1 = F+G$ have the same winding number 
  along $\partial\Gamma$, whence the same number of zeros in $\Gamma$ (Theorem \ref{thm:CountingComplexRoots}).
\end{proof}

\subsection{The global winding number} \label{sub:GlobalWindingNumber}

We can now prove Theorem \ref{thm:GlobalWindingNumber}, stating that 
$\cind{\partial\Gamma}{F} = \deg F$ for every polynomial $F \in \C[Z]^*$ 
and every sufficiently big rectangle $\Gamma$.  

\begin{proposition} \label{prop:CauchyBound}
  Given $F = Z^n + c_1 Z^{n-1} + \dots + c_n$ in $\C[Z]$ we define 
  its \emph{Cauchy radius} to be $\rho_F := 1 + \max\{\abs{c_1},\dots,\abs{c_n}\}$.
  This implies that $\abs{F(z)} \ge 1$ for every $z \in \C$ with $\abs{z} \ge \rho_F$.
  Hence all zeros of $F$ in $\C$ lie in the \emph{Cauchy disk} 
  $B(\rho_F) = \{\, z \in \C \mid \abs{z} < \rho_F \,\}$.
\end{proposition}

\begin{proof}
  The assertion is true for $\rho_F = 1$, since then $F = Z^n$.
  We can thus assume $\rho_F > 1$. 
  For all $z \in \C$ satisfying $\abs{z} \ge \rho_F$ we find 
  \begin{multline*}
    \abs{ F(z) - z^n } = \abs{ c_1 z^{n-1} + \dots + c_{n-1} z + c_n }
    \le \abs{c_1} \abs{z^{n-1}}  + \dots + \abs{c_{n-1}} \abs{z} + \abs{c_n} 
    \\ \le \max\left\{\abs{c_1},\dots,\abs{c_{n-1}},\abs{c_n}\right\} \left( \abs{z}^{n-1} + \dots + \abs{z} + 1 \right)
    = (\rho_F - 1) \frac{ \abs{z}^n - 1 }{ \abs{z} - 1 } 
    \le \abs{z}^n - 1 .
  \end{multline*}
  We conclude that $\abs{F(z)} \ge \abs{z^n} - \abs{F(z) - z^n} \ge 1$.
\end{proof}

This proposition holds over any ordered field $\R$ and its complex extension $\C = \R[i]$
because it uses only the general properties $\abs{ a + b } \le \abs{a} + \abs{b}$ 
and $\abs{ a \cdot b } \le \abs{a} \cdot \abs{b}$.
It is not an existence result, but only an \emph{a priori} bound:
if $F$ has zeros in $\C$, then they necessarily lie in $B(\rho_F)$.  
Over a real closed field $\R$, the algebraic winding number 
counts the number of zeros, and we arrive at the following conclusion.

\begin{theorem} \label{thm:BigRectangle}
  For every polynomial $F \in \C[Z]^*$ and every rectangle $\Gamma \subset \C$ 
  containing the Cauchy disk $B(\rho_F)$, we have $\cind{\partial\Gamma}{F} = \deg F$.
\end{theorem}

\begin{proof}
  The assertion is clear for $F \in \C^*$ of degree $0$.
  Consider $F = Z^n + c_1 Z^{n-1} + \dots + c_n$ with $n \ge 1$
  and set $M = \max\{\abs{c_1},\dots,\abs{c_n}\}$.
  The homotopy $F_t = Z^n + t(c_1 Z^{n-1} + \dots + c_n)$ deforms $F_1 = F$ to $F_0 = Z^n$.
  The Cauchy radius of $F_t$ is $\rho_t = 1 + t M$, which shrinks from $\rho_1 = \rho_F$ to $\rho_0 = 1$.
  By the previous proposition, the polynomial $F_t \in \C[Z]$ has no zeros on $\partial\Gamma$.  
  We conclude that $\cind{\partial\Gamma}{F_1} = \cind{\partial\Gamma}{F_0} = n$,
  using Corollaries  \ref{cor:HomotopyInvariance} and \ref{cor:CountingComplexRoots}.
\end{proof}

This completes the proof of the Fundamental Theorem of Algebra.
On the one hand Theorem \ref{thm:BigRectangle} says that 
$\cind{\partial\Gamma}{F} = \deg F$ provided that $\Gamma \supset B(\rho_F)$, 
and on the other hand Theorem \ref{thm:CountingComplexRoots} says that 
$\cind{\partial\Gamma}{F}$ equals the number of zeros of $F$ in $\Gamma \subset \C$.

\begin{remark} \label{rem:CauchyPolynomial}
  The Cauchy radius of Proposition \ref{prop:CauchyBound} 
  is the simplest of an extensive family of root bounds, 
  see Henrici \cite[\textsection 6.4]{Henrici:1974} 
  and Rahman--Schmeisser \cite[chap.\,8]{RahmanSchmeisser:2002}.
  We mention a nice and useful improvement:
  to each polynomial $F = c_0 Z^n + c_1 Z^{n-1} + \dots + c_n$ in $\C[Z]$ we associate 
  its \emph{Cauchy polynomial} $F^\circ = \abs{c_0} X^n - \abs{c_1} X^{n-1} - \dots - \abs{c_n}$ in $\R[X]$. 
  This implies $\abs{F(z)} \ge F^\circ(\abs{z})$ for all $z \in \C$.
  We assume $c_0 \ne 0$ and $c_n \ne 0$, such that $F^\circ(0) < 0$ and $F^\circ(x) > 0$ for large $x \in \R$.
  According to Descartes' rule of signs (Theorem \ref{thm:DescartesBudanFourier}),
  the polynomial $F^\circ$ has a unique positive root $\rho$, 
  whence $F^\circ(x) > 0$ for all $x > \rho$, and $F^\circ(x) < 0$ for all $0 \le x < \rho$.
  Given some $r > 0$ with $F^\circ(r) > 0$, we have $\abs{F(z)} > 0$ for all $|z| \ge r$,
  whence all zeros of $F$ in $\C$ lie in the disk $B(r)$.
  (Again this holds over any ordered field $\R$.)
\end{remark}

\begin{Annotation}
  \textbf{(Degree bounds)}
  The Fundamental Theorem of Algebra, in the form that we have just proven, 
  states that if the field $\R$ is real closed, i.e., every polynomial 
  $P \in \R[X]$ satisfies the intermediate value property over $\R$,
  then the field $\C = \R[i]$ is algebraically closed, i.e.,
  every polynomial $F \in \C[Z]$ splits into linear factors over $\C$.
  Since we are working exclusively with polynomials, 
  it is natural to study degree bounds.

  We call an ordered field $\R$ \emph{real $d$-closed} if every polynomial 
  $P \in \R[X]$ of degree $\le d$ satisfies the intermediate value property over $\R$. 
  Likewise, we call a field $\C$ \emph{algebraically $d$-closed} if every 
  polynomial $F \in \C[Z]$ of degree $\le d$ splits into linear factors over $\C$.
  As in Theorem \ref{thm:GeometricCharacterization} it is easy to establish 
  the following implication: if $\R$ is an ordered field such that $\R[i]$ 
  is algebraically $d$-closed, then $\R$ is real $d$-closed.
  The converse seems to be open:

  \begin{Question}
    If $\R$ is real $d$-closed, does this imply that $\R[i]$ is algebraically $d$-closed?
  \end{Question}
  
  This is trivally true for $d=1$.  The answer is also affirmative for $d=2,3,4$ 
  because quadratic, cubic, and quartic equations can be solved by radicals 
  of degree $n \le d$, i.e., roots of $Z^n - a$ with $a \in \C$, and
  these roots can be constructed in $\R[i]$ if $\R$ is real $n$-closed.
  Notice that quartic equations can be reduced to auxialiary equations of degree $\le 3$,
  so if $\R$ is real $3$-closed, then $\R[i]$ is algebraically $4$-closed
  and $\R$ is in fact real $4$-closed!

  What happens in degree $5$ and higher?  An affirmative answer 
  would be surprising\dots\  but a Galois-type obstruction seems unlikely, too.
  All of the arguments presented in this article immediately extend to 
  refined versions with the desired degree bounds -- the only exception is 
  our algebraic proof of Lemma \ref{lem:GlobalIndexZero}, where we construct 
  a Sturm chain in $\R[X,Y]$ with little control on the degrees.
  It seems to be an interesting research project to investigate
  this phenomenon in full depth and to prove optimal degree bounds.
\end{Annotation}

\subsection{Geometric characterization of the winding number}

We have constructed the algebraic winding number via Cauchy indices
\eqref{property:AlgebraicComputation} and then derived its 
geometric properties: normalization \eqref{property:Normalization}, 
multiplicativity \eqref{property:Multiplicativity}, 
and homotopy invariance \eqref{property:HomotopyInvariance}.
We now complete the circle by showing that 
\eqref{property:Normalization}, \eqref{property:Multiplicativity}, 
\eqref{property:HomotopyInvariance} characterize the winding number 
and imply \eqref{property:AlgebraicComputation}.
We begin with two fundamental examples.

\begin{example}[stars] \label{exm:StarlikeHomotopy}
  Every loop $\gamma$ in $U = \C \minus \R_{\le0}$ is homotopic in $U$ 
  to the constant loop $\gamma_0 = 1$ via $\gamma_s = 1 + s(\gamma-1)$, 
  whence \eqref{property:Normalization} and \eqref{property:HomotopyInvariance} 
  imply $\Wind(\gamma) = 0$.  The same holds in $\C \minus c \R_{\le0}$ for any $c \in \C^*$.
  Using \eqref{property:Multiplicativity} we obtain 
  $\Wind(c \gamma) = \Wind(\gamma)$ for all loops $\gamma$ and $c \in \C^*$.
\end{example}

\begin{figure}[ht]
  \hfill
  \includegraphics[height=18ex]{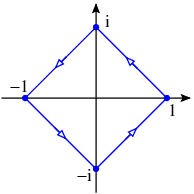} 
  \hfill
  \includegraphics[height=18ex]{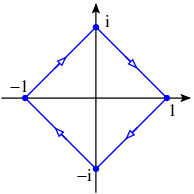} 
  \hfill
  \includegraphics[height=18ex]{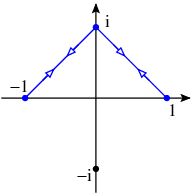} 
  \hfill
  \includegraphics[height=18ex]{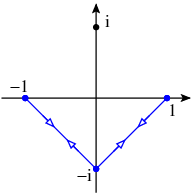} 
  \hfill{}
  \caption{The winding number $\Wind(\delta)$ of a diamond-shaped loop $\delta$.}
  \label{fig:DiamondLoop}
\end{figure}

\begin{example}[diamonds] \label{exm:DiamondLoop}
  For $0 < t_0 - \varepsilon < t_0 < t_0 + \varepsilon < 1$ 
  let $\delta \colon [0,1] \to \C$ be the loop that 
  linearly interpolates between 
  $\delta(0) = \delta(t_0-\varepsilon) = 1$, 
  $\delta(t_0-\nicefrac{\varepsilon}{2}) = \pm i$, 
  $\delta(t_0) = -1$, 
  $\delta(t_0+\nicefrac{\varepsilon}{2}) = \pm i$, 
  and $\delta(t_0+\varepsilon) = \delta(1) = 1$.
  Then $\Wind(\delta) = \frac{1}{2i}\bigl[ \delta(t_0-\nicefrac{\varepsilon}{2}) - \delta(t_0+\nicefrac{\varepsilon}{2}) \bigr]$ 
  can be deduced from \eqref{property:Normalization}, \eqref{property:Multiplicativity}, 
  \eqref{property:HomotopyInvariance} alone.  The proof is left as an exercise.
\end{example}

\begin{theorem} \label{thm:GeometricCharacterization}
  Consider an ordered field $\R$ and its complex extension $\C = \R[i]$ where $i^2=-1$.
  Let $\Omega$ be the set of piecewise polynomial loops $\gamma \colon [0,1] \to \C^*$.
  \begin{enumerate}
  \item \label{char:AlgebraicallyClosed}
    If some map $\Wind \colon \Omega \to \ZZ$ satisfies
    \eqref{property:Normalization}, \eqref{property:Multiplicativity}, 
    \eqref{property:HomotopyInvariance}, then $\C$ is algebraically closed.
  \item \label{char:RealClosed}
    If the field $\C = \R[i]$ is algebraically closed, then the ordered field $\R$ is real closed.
  \item \label{char:Unicity}
    If two maps $\Wind,\tilde\Wind \colon \Omega \to \ZZ$ satisfy
    \eqref{property:Normalization}, \eqref{property:Multiplicativity}, 
    \eqref{property:HomotopyInvariance}, then $\Wind = \tilde\Wind$.
  \end{enumerate}
\end{theorem}

\begin{proof}
  The result \eqref{char:AlgebraicallyClosed} has been deduced in \sref{sub:FTAbyWindingNumber}.
  Regarding \eqref{char:RealClosed}, every $P \in \R[X]$ factors as $P = c_0(X-z_1)\dots(X-z_n)$ 
  with $z_1,\dots,z_n \in \C$.  Since $P(\overline{z_k}) = \overline{ P(z_k) } = 0$, 
  each $z_k \in \C \minus \R$ comes with its conjugate.  Pairing these
  we have $P = c_0(X-x_1)\cdots(X-x_r) Q_1 \cdots Q_s$  where $x_1,\dots,x_r \in \R$ 
  and $Q_j = (X-w_j)(X-\overline{w_j})$ with $w_1,\dots,w_s \in \C \minus \R$.
  The minimum of $Q_j  = X^2 - 2 \re(w_j) X + \abs{w_j}^2$ is $Q_j(\re w_j) = \abs{w_j}^2 - \re(w_j)^2 > 0$, 
  whence $Q_j(x) > 0$ for all $x \in \R$.
  If $P(a) P(b) < 0$ for some $a<b$ in $\R$, then $a < x_k < b$ for some zero $x_k$ of $P$.

  It remains to prove unicity \eqref{char:Unicity} of the winding number.
  Let $\gamma \colon [0,1] \to \C^*$ be a piecewise polynomial loop.
  If $\gamma$ lies in $\C \minus \R_{\le0}$, then we know 
  $\Wind(\gamma) = 0$ from Example \ref{exm:StarlikeHomotopy}.
  In general, $\gamma$ will cross the negative real axis $\R_{<0}$.  
  Since $\im\gamma \colon [0,1] \to \R$ is piecewise polynomial and $\R$ is 
  real closed by \eqref{char:AlgebraicallyClosed} and \eqref{char:RealClosed}, 
  we can use the intermediate value property.  We can assume that $\gamma$ 
  intersects $\R$ only a finite number of times $t_1,\dots,t_k$, where 
  $0 < t_1 < \dots < t_n < 1$; if not, then $c\gamma$ will do for some $c \in \C^*$. 
  We separate $t_1,\dots,t_k$ in disjoint intervals $I_k = [t_k-\varepsilon,t_k+\varepsilon]$ 
  for some sufficiently small $\varepsilon > 0$.  If $\gamma(t_k) > 0$, we set $\delta_k = 1$.
  If $\gamma(t_k) < 0$, then we define $\delta_k$ to be the loop of Example \ref{exm:DiamondLoop} 
  with support $I_k$: since $\im\gamma|I_k$ changes sign at most at $t_k$, the signs 
  $\delta_k(t_k\pm\nicefrac{\varepsilon}{2}) \in \{\pm i\}$ can be so chosen that $\im\gamma \cdot \im\delta_k \le 0$.
  Multiplication by $\delta_k$ changes $\gamma$ only on $I_k$ and ensures that 
  $\gamma \delta_k|I_k$ intersects $\R$ only in $\R_{>0}$.
  We thus obtain $\gamma \delta_1 \cdots \delta_n$ in $\C \minus \R_{\le0}$.
  From Example \ref{exm:StarlikeHomotopy}, we know $\Wind(\gamma \delta_1 \cdots \delta_n) = 0$, whence
  $-\Wind(\gamma) = \Wind(\delta_1) + \cdots + \Wind(\delta_n)$ by \eqref{property:Multiplicativity},
  and the right hand side is determined by \eqref{property:Normalization}, 
  \eqref{property:Multiplicativity}, \eqref{property:HomotopyInvariance} 
  as in Example \ref{exm:DiamondLoop}.
\end{proof}

\begin{Annotation}
  \textbf{(Detailed calculations)}
  For $t = t_k$ the assertion $(\gamma\delta_k)(t) \in \R_{>0}$ is clear by construction.
  For $t \le t_k - \varepsilon$ and for $t \ge t_k + \varepsilon$ there is nothing to prove.
  For $t \in \ee{t_k - \varepsilon}{t_k + \varepsilon} \minus \{t_k\}$ we have $\im\gamma \cdot \im\delta_k < 0$.
  For $\gamma\delta_k$ we find $\re(\gamma\delta_k) = \re\gamma \cdot \re\delta_k - \im\gamma \cdot \im\delta_k$
  and $\im(\gamma\delta_k) = \im\gamma \cdot \re\delta_k + \re\gamma \cdot \im\delta_k$, whence
  \[ 
  \im(\gamma\delta_k) \cdot \re\gamma \cdot \im\delta_k = (\im\gamma \cdot \im\delta_k) (\re\gamma \cdot \re\delta_k) + (\re\gamma \cdot \im\delta_k)^2 .
  \]
  If $\im(\gamma\delta_k)$ vanishes in some $t \in I_k$, then $(\re\gamma \cdot \re\delta_k)(t) \ge 0$, whence $\re(\gamma\delta_k)(t) > 0$.
\end{Annotation}  


\section{Algorithmic aspects} \label{sec:AlgorithmicAspects}

The preceding sections \sref{sec:Cauchy} and \sref{sec:FTA} show how to
construct the algebraic winding number over a real closed field $\R$.
We have used it for proving existence and locating the roots of polynomials 
over $\C = \R[i]$.  This section discusses algorithmic questions.  
To this end we have to narrow the scope: in order to work with convergence 
of sequences in $\R$, we additionally assume the ordered field $\R$ to be archimedean, 
which amounts to $\R \subset \RR$. 

The algorithm described here is often attributed to Wilf \cite{Wilf:1978} in 1978, 
but it was already explicitly described by Sturm \cite{Sturm:1836} 
and Cauchy \cite{Cauchy:1837} in the 1830s. 
It can also be found in Runge's \emph{Encyklop\"adie} article 
\cite[Kap.\,IB3, \textsection a6]{EncyklMathWiss} in 1898. 
Numerical variants are known as \emph{Weyl's quadtree method} (1924)
or \emph{Lehmer's method} (1961); see \sref{sub:ConstructiveAspects}.
I propose to call it the \emph{Sturm--Cauchy method}, or \emph{Cauchy's algebraic method}
if emphasis is needed to differentiate it from Cauchy's analytic method using integration.
For a thorough study of complex polynomials see Marden \cite{Marden:1966}, 
Henrici \cite{Henrici:1974}, and Rahman--Schmeisser \cite{RahmanSchmeisser:2002};
the latter contains extensive historical notes and a guide to the literature.

\subsection{Turing computability} \label{sub:TuringComputability}

The theory of ordered or orderable fields, nowadays called \emph{real algebra},
was initiated by Artin and Schreier \cite{ArtinSchreier:1926,ArtinSchreier:1927} in the 1920s; 
a spectacular early success was Artin's solution \cite{Artin:1927} of Hilbert's 17th problem.
Since the 1970s real-algebraic geometry is flourishing anew \cite{BochnakCosteRoy:1998} 
and, with the advent of computers, algorithmic 
aspects have gained importance \cite{BasuPollakRoy:2006}.
We shall focus here on basic questions of computability.


\begin{definition}
  We say that an ordered field $(\R,+,\cdot,<)$ can be implemented
  on a Turing machine if each element $a \in \R$ can be coded as input/output 
  for such a machine and each of the field operations $(a,b) \mapsto a+b$, 
  $a \mapsto -a$, $(a,b) \mapsto a \cdot b$, $a \mapsto a^{-1}$ as well as 
  the comparisons $a = b$, $a < b$ can be carried out by a uniform algorithm.
\end{definition}

\begin{example}
  The field $(\RR,+,\cdot,<)$ of real numbers cannot be implemented
  on a Turing machine because the set $\RR$ is uncountable:
  it is impossible to code each real number by a finite string 
  over a finite alphabet, as required for input/output. 
  This argument is independent of the chosen representation.
  If we insist on representing each and every real number, then 
  this fundamental obstacle can only be circumvented by postulating
  a hypothetical \emph{real number machine} \cite{SmaleEtAl:1998},
  which transcends the traditional setting of Turing machines.
\end{example}

\begin{example}
  The subset $\RR_{\mathrm{comp}} \subset \RR$ of computable real numbers, 
  as defined by Turing \cite{Turing:1936} in his famous 1936 article, 
  forms a countable, real closed subfield of $\RR$.  
  Each computable number $a$ can be represented as input/output for a universal Turing machine 
  by an algorithm that approximates $a$ to any desired precision.
  This overcomes the obstacle of the previous example by restriction to $\RR_{\mathrm{comp}}$.
  Unfortunately, not all operations of $(\RR_{\mathrm{comp}},+,\cdot,<)$ 
  can be implemented.  There exists no algorithm that for each 
  computable real number $a$, given in form of an algorithm,
  determines whether $a=0$, or more generally determines the sign of $a$.
  (This is an instance of the notorious Entscheidungsproblem.)
\end{example}


\begin{example}
  The algebraic closure $\QQ^c$ of $\QQ$ in $\RR$ is a real closed field.
  Unlike the field of computable real numbers, the much smaller subfield $(\QQ^c,+,\cdot,<)$ 
  can be implemented on a Turing machine \cite{RoySzpirglas:1990,Roy:1996}.
  More specifically, consider a polynomial $F = c_0 Z^n +  c_1 Z^{n-1} + \dots + c_n$ 
  whose coefficients $c_k \in \CC$ are algebraic over $\QQ$.  Then $\re(c_k)$ and $\im(c_k)$ 
  are also algebraic, and the field $\R = \QQ(\re(c_0),\im(c_0),\dots,\re(c_n),\im(c_n)) \subset \QQ^c$ 
  is a finite extension over $\QQ$.  It can be generated by one element, which means $\R = \QQ(\alpha)$ 
  for some $\alpha \in \R$, and such a presentation makes it convenient for implementation.
\end{example}

\subsection{The Sturm--Cauchy root-finding algorithm} \label{sub:GlobalRootFinding}

We consider a complex polynomial 
\[ F = c_0 Z^n +  c_1 Z^{n-1} + \dots + c_{n-1} Z + c_n \quad\text{in}\quad \CC[Z] \]
that we assume to be \emph{Turing implementable}, that is, we require the ordered field 
\[ \QQ(\re(c_0),\im(c_0),\dots,\re(c_n),\im(c_n)) \subset \RR \]
to be implementable in the preceding sense. 
We begin with the following preparations.
\begin{itemize}
\item We divide $F$ by $\gcd(F,F')$ to ensure that all roots of $F$ are simple.
\item As in Remark \ref{rem:CauchyPolynomial} we determine $r \in \NN$ such that all roots of $F$ lie in $B(r)$.
\end{itemize}

The following terminology will be convenient:
a \emph{$0$-cell} is a singleton $\{a\}$ with $a \in \CC$;
a \emph{$1$-cell} is an open line segment, either vertical $\{x_0\} \times \ee{y_0}{y_1}$
or horizontal $\ee{x_0}{x_1} \times \{y_0\}$ with $x_0 < x_1$ and $y_0 < y_1$ in $\RR$;
a \emph{$2$-cell} is an open rectangle $\ee{x_0}{x_1} \times \ee{y_0}{y_1}$ in $\CC$.

It is immediate to check whether a $0$-cell contains a root of $F$.
Sturm's theorem (Corollary \ref{cor:CountingRealRoots}) allows us to count 
the roots of $F$ in a $1$-cell $\ee{a}{b}$: for $G = F( a + X (b-a) )$ in $\CC[X]$ 
calculate $P = \gcd(\re G, \im G)$ in $\RR[X]$ and count roots of $P$ in $\ee{0}{1}$.
Cauchy's theorem (Theorem \ref{thm:CountingComplexRoots}) allows us to count the roots in a $2$-cell.
In both cases the crucial subalgorithm is the computation of Sturm chains 
which we will discuss in \sref{sub:SubresultantAlgorithm} below.

Building on these methods, the root-finding algorithm successively for $t=0,1,2,3\dots$ 
constructs a list $L_t = \{\Gamma_1,\dots,\Gamma_{n_t}\}$ of disjoint cells, which behaves as follows:
\begin{itemize}
\item Each root of $F$ is contained in exactly one cell $\Gamma \in L_t$.
\item Each cell $\Gamma \in L_t$ contains at least one root of $F$.
\item Each cell $\Gamma \in L_t$ has diameter $\le 3r \cdot 2^{-t}$.
\end{itemize}

The algorithm proceeds as follows:
To begin, we initialize $L_0 = \{\Gamma\}$ with the square $\Gamma = \ee{-r}{+r} \times \ee{-r}{+r}$.
Given $L_t$ we construct $L_{t+1}$ by treating each cell in $L_t$ as follows.
\begin{enumerate}
  \setcounter{enumi}{-1}
\item
  Retain all $0$-cells unchanged.
\item
  Bisect each $1$-cell into two $1$-cells of equal length as in Figure \ref{fig:Quadrisection},
  which also creates one interior $0$-cell.  Retain each new cell that contains a root of $F$.
\item
  Quadrisect each $2$-cell into four $2$-cells of equal size as in Figure \ref{fig:Quadrisection},
  which creates four interior $1$-cells and one $0$-cell.  
  Retain each new cell that contains a root of $F$.
\end{enumerate}

\begin{figure}[ht]
  \hfill
  \includegraphics[height=15ex]{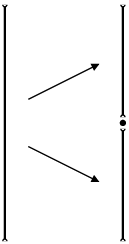} 
  \hfill
  \includegraphics[height=15ex]{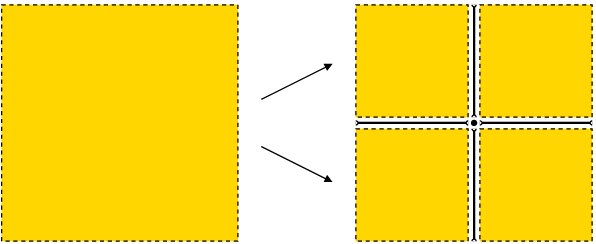} 
  \hfill{}
  \caption{Bisecting a $1$-cell and quadrisecting a $2$-cell}
  \label{fig:Quadrisection}
\end{figure}

Collecting all retained cells we obtain the new list $L_{t+1}$.
After some initial iterations (\sref{sub:SubresultantAlgorithm}) 
all roots will lie in disjoint cells $\Gamma_1,\dots,\Gamma_n$, each containing precisely one root.  
Taking the midpoint $u_k \in \Gamma_k$, this can be seen as $n$ approximate roots $u_1,\dots,u_n$, 
each with an error bound $\delta_k \le \sqrt{2} \, r / 2^t$ such that 
each $u_k$ is $\delta_k$-close to a root of $F$.

\subsection{Crossover to Newton's local method} \label{sub:NewtonMethod}

For $F \in \CC[Z]$, Newton's method consists in iterating the map 
$\Phi(z) = z - F(z)/F'(z)$ defined on $\{\, z \in \CC \mid F'(z) \ne 0 \,\}$.
This simple technique is very powerful because of its local behaviour around zeros.

\begin{theorem} \label{thm:LocalNewton}
  The fixed points of Newton's map $\Phi(z) = z - F(z)/F'(z)$ are the simple zeros of $F$,
  that is, the points $z_0 \in \CC$ such that $F(z_0) = 0$ and $F'(z_0) \ne 0$.
  For each fixed point $z_0$ there exists $\delta > 0$ 
  such that every initial value $u_0 \in B(z_0,\delta)$ satisfies 
  \begin{equation} \label{eq:NewtonConvergence}
    \abs{\Phi^t(u_0)-z_0} \le 2^{1-2^t} \cdot \abs{u_0-z_0} 
    \qquad \text{for all $t \in \NN$.} 
    \qedhere
  \end{equation}
\end{theorem}

The convergence to $z_0$ is thus very fast but requires 
a good initial approximation $u_0 \approx z_0$; 
otherwise Newton's iteration may be slow at first or not converge at all.
On a practical level this raises two problems:
first, how to find approximate zeros, and second, how to determine 
whether a given approximation is sufficiently good to guarantee 
fast convergence as in \eqref{eq:NewtonConvergence}?
The global root-finding algorithm of \sref{sub:GlobalRootFinding} approximates 
all roots simultaneously, and the following criterion 
exploits this information for launching Newton's method:

\begin{theorem} \label{thm:GlobalLocalSwitch}
  Let $F \in \CC[Z]$ be a separable polynomial of degree $n \ge 2$.
  Suppose we have separated the roots $z_1,\dots,z_n$ of $F$ 
  in closed disks $\bar{B}(u_1,\delta_1),\dots,\bar{B}(u_n,\delta_n)$ such that
  \begin{equation} \label{eq:NewtonSeparation}
    \textstyle
    3n \delta_k \le \abs{ u_k - u_j } 
    \qquad\text{for all $j \ne k$.}
  \end{equation}
  Then Newton's iteration satisfies 
  $\abs{\Phi^t(u_k)-z_k} \le 2^{1-2^t} \cdot \delta_k$ for all $t \in \NN$.
\end{theorem}

\begin{proof}
  For $F = (Z-z_1)\cdots(Z-z_n)$ we have $F'/F = \sum_{j=1}^n (Z-z_j)^{-1}$.
  This implies that $\Phi(z) = z - 1/\sum_{j=1}^n (z-z_j)^{-1}$, 
  provided that $F(z) \ne 0$ and $F'(z) \ne 0$, whence
  \[
  \frac{\Phi(z)-z_k}{z-z_k} = 1 - \frac{1}{\sum_{j=1}^n \frac{z-z_k}{z-z_j}}
  = \frac{ \sum_{j \ne k} \frac{z-z_k}{z-z_j} }{ 1 + \sum_{j \ne k} \frac{z-z_k}{z-z_j} } .
  \]
  
  By hypothesis, we have approximate roots 
  $u_1,\dots,u_n$ such that $\abs{u_k-z_k} \le \delta_k$.  
  Consider an arbitrary point $z \in \bar{B}(z_k,\delta_k)$,
  which entails $\abs{ z - u_k } \le 2 \delta_k$.
  For all $j \ne k$ we find
  \[
  \abs{ z - z_j } \ge \abs{ u_k - u_j } - \abs{ z - u_k } - \abs{ z_j - u_j } 
  \ge \abs{ u_k - u_j } - 2\delta_k - \delta_j \ge (3n-3) \delta_k ,
  \]
  where the last inequality $(3n-1) \delta_k + \delta_j \le \abs{ u_k - u_j }$
  is a convex linear combination of \eqref{eq:NewtonSeparation}. 
  This ensures that $\bigl|\sum_{j \ne k} \frac{z-z_k}{z-z_j}\bigr| 
  \le \sum_{j \ne k} \bigl|\frac{z-z_k}{z-z_j}\bigr| 
  \le \frac{\abs{z-z_k}}{3\delta_k} \le \frac{1}{3}$. 
  For $z \ne z_k$ this implies 
  \[
  \left| \frac{\Phi(z)-z_k}{z-z_k} \right|
  \le \frac{ \left| \sum_{j \ne k} \frac{z-z_k}{z-z_j} \right| }
  { 1 - \left| \sum_{j \ne k} \frac{z-z_k}{z-z_j} \right| }
  \le \frac{ \frac{1}{3\delta_k} \abs{z-z_k} }{ 1 - \frac{1}{3} } = \frac{\abs{z-z_k}}{2\delta_k} .
  \]
  
  For all $z \in \bar{B}(z_k,\delta_k)$ we conclude that 
  $\abs{ \Phi(z)-z_k } \le \frac{1}{2\delta_k} \abs{z-z_k}^2$,
  whence $\abs{\Phi^t(z)-z_k} \le 2^{1-2^t} \cdot \abs{z-z_k}$ by induction on $t \in \NN$.
  In particular this holds for $z = u_k$. 
\end{proof}

\begin{Annotation}
  \textbf{(Isolating a single root)}
  The proof of Theorem \ref{thm:GlobalLocalSwitch} also applies when we isolate only one root.
  The only requirement is that all other roots must be sufficiently far away:

  \begin{Theorem}
    Let $F \in \CC[Z]$ be of degree $n \ge 2$.
    Suppose that $\bar{B}(u_1,\delta)$ contains exactly one root $z_1$ of $F$,
    and $\bar{B}(u_1,(3n-1)\delta)$ contains no further root.  Then Newton's iteration 
    satisfies $\abs{\Phi^t(u_1)-z_1} \le 2^{1-2^t} \cdot \delta$ for all $t \in \NN$.
  \end{Theorem}

  \begin{proof}
    For $z \in \bar{B}(z_1,\delta)$ we have $\abs{ z - u_1 } \le 2 \delta$, 
    whence $\abs{ z - z_j } \ge \abs{ u_1 - z_j } - \abs{ z - u_1 }  \ge 3(n-1) \delta$.
    This ensures that $\bigl|\sum_{j > 1} \frac{z-z_1}{z-z_j}\bigr| 
    \le \sum_{j > 1} \bigl|\frac{z-z_1}{z-z_j}\bigr| 
    \le \frac{\abs{z-z_1}}{3\delta} \le \frac{1}{3}$.
    As above, we conclude that $\abs{ \Phi(z)-z_0 } \le \frac{1}{2\delta} \abs{z-z_0}^2$.
  \end{proof}
\end{Annotation}

\begin{Annotation}
  \textbf{(Better convergence)}
  We want to approximate the root $z_k$ starting from $u_k$. 
  But we also have approximations to the other roots, so we should try to exploit this extra information.
  Instead of $F$ we may apply Newton's method to $G_k(z) = F(z) / \prod_{j \ne k} (z-u_j)$.
  On $\bar{B}(z_k,\delta_k)$ both $F$ and $G_k$ have the same root $z_k$, but $G_k$ yields better convergence properties.
  This is obvious in the special case where $u_j = z_j$ for all $j \ne k$, which leads to the deflated polynomial $G_k(z) = z-z_k$.
  In general we only have $u_j \approx z_j$, so that $G_k$ is more complicated than $F$.
  Nevertheless it is advantageous for Newton's iteration because the basin of attraction around $z_k$ is bigger.

  \begin{Theorem} 
    Let $F \in \CC[Z]$ be a separable polynomial of degree $n \ge 4$.
    Suppose we have separated the roots $z_1,\dots,z_n$ of $F$ 
    in closed disks $\bar{B}(u_1,\delta_1),\dots,\bar{B}(u_n,\delta_n)$ such that
    \begin{equation} 
      \textstyle
      4 \sqrt{n-1} \cdot \delta_k \le \abs{ u_k - u_j } 
      \qquad\text{for all $j \ne k$.}
    \end{equation}
    Then Newton's iteration applied to $G_k$ satisfies 
    $\abs{\Phi^t(u_k)-z_k} \le 2^{1-2^t} \cdot \delta_k$ for all $t \in \NN$.
  \end{Theorem}

  \begin{proof}
    We follow the lines of the previous proof.   First we see that 
    \[ 
    \frac{G_k'(z)}{G_k(z)} 
    = \frac{1}{z-z_k} + \sum_{j \ne k} \frac{1}{z-z_j} - \frac{1}{z-u_j} 
    = \frac{1}{z-z_k} + \sum_{j \ne k} \frac{z_j-u_j}{ (z-z_j) (z-u_j) } .
    \]
    For $z \in \bar{B}(z_k,\delta_k)$ we iterate the map $\Phi(z) = z - G_k(z) / G_k'(z)$.  
    For $z \ne z_k$ the relative error is 
    \[ 
    \frac{ \Phi(z) - z_k }{ z - z_k } = \frac{ H(z) }{ 1 + H(z) }
    \quad\text{where}\quad
    H(z) = \sum_{j \ne k} \frac{ (z-z_k) (z_j - u_j) }{ (z-z_j) (z-u_j) } .
    \]
    For the factors in the denominator we find 
    $\abs{z-z_j} \ge \abs{u_k-u_j} - \abs{z-u_k} - \abs{z_j-u_j} \ge \abs{u_k-u_j} - 2 \delta_k - \delta_j$
    and $\abs{z-u_j} \ge \abs{u_k-u_j} - \abs{z-u_k} \ge \abs{u_k-u_j} - 2\delta_k$, whence
    \[
    \abs{z-z_j} \cdot \abs{z-u_j} 
    \ge \abs{u_k-u_j}^2 - (4\delta_k+\delta_j) \abs{u_k-u_j}
    \ge \abs{u_k-u_j}^2 ( 1 - 5/(4\sqrt{n-1}) ) .
    \]
    Here we have used the hypothesis $\delta_k \le \abs{ u_k - u_j } / (4 \sqrt{n-1})$.
    A small calculation leads to the estimate
    \[
    \frac{ \delta_k \delta_j }{ \abs{z-z_j} \cdot \abs{ z-u_j } }
    \le \frac{ 1/[16(n-1)] }{ 1 - 5/[4\sqrt{n-1}] }
    \le \frac{1}{3 (n-1)} .  
    \]
    We thus obtain 
    \[
    \abs{H(z)} 
    \le \frac{\abs{z-z_k}}{\delta_k} \sum_{j \ne k} \frac{ \delta_k \delta_j }{ \abs{z-z_j} \cdot \abs{ z-u_j } }
    \le \frac{\abs{z-z_k}}{\delta_k} \cdot \frac{1}{3} \le \frac{1}{3} .  
    \]
    For the relative error this implies 
    \[
    \left| \frac{ \Phi(z) - z }{ z - z_k } \right|
    \le \frac{ \abs{ H(z) } }{ 1 - \abs{ H(z) } }
    \le \frac{ \frac{1}{3\delta_k} \abs{z-z_k} }{ 1 - \frac{1}{3} } 
    \le \frac{ \abs{z-z_k} }{ 2 \delta_k } .
    \]
    For all $z \in \bar{B}(z_k,\delta_k)$ we conclude that 
    $\abs{ \Phi(z)-z_k } \le \frac{1}{2\delta_k} \abs{z-z_k}^2$,
    whence $\abs{\Phi^t(z)-z_k} \le 2^{1-2^t} \cdot \abs{z-z_k}$ by induction on $t \in \NN$.
    In particular this holds for $z = u_k$. 
  \end{proof}
\end{Annotation}  

As an alternative to the tailor-made criterion of Theorem \ref{thm:GlobalLocalSwitch},
the following theorem of Smale \cite[chap.\,8]{SmaleEtAl:1998} 
provides a far more general convergence criterion in terms of local data. 
It applies in particular to polynomials, where it is most easily implemented.

\begin{theorem}[Smale 1986]
  Let $f \colon U \to \CC$ be an analytic function on some open set $U \subset \CC$.
  Consider $u_0 \in U$ satisfying $f(u_0) \ne 0$ and $f'(u_0) \ne 0$, so that
  $\eta = \abs{ f(u_0) / f'(u_0) } > 0$ is the initial displacement in Newton's iteration.  
  Suppose further that $B(u_0,2\eta) \subset U$ and the expansion 
  $f(z) = \sum_{k=0}^\infty a_k (z-u_0)^k$ satisfies
  $\abs{a_k} \le (8\eta)^{1-k} \abs{a_1}$ for all $k \ge 2$.
  Then $f$ has a unique zero $z_0$ in $B(u_0,2\eta)$
  and Newton's iteration converges as in \eqref{eq:NewtonConvergence}.
\end{theorem}

\subsection{Fast Cauchy index computation} \label{sub:SubresultantAlgorithm}

To complete the picture we briefly consider the bit-complexity of 
the Sturm--Cauchy algorithm described in \sref{sub:GlobalRootFinding}.
In order to simplify we will work over the rational numbers.
The fundamental problem is, for given $R/S \in \QQ(X)$, to compute $\Ind_0^1(\frac{R}{S})$.
To this end we wish to construct some chain $S_0,S_1,\dots,S_n \in \QQ[X]^*$ 
starting with $S_1/S_0 = R/S$ and ending with $S_n \in \QQ^*$ such that 
\begin{equation} \label{eq:RationalChain}
  A_k S_{k+1} + B_k S_k + C_k S_{k-1} = 0
  \quad\text{with}\quad A_k \in \QQ^*, \; B_k \in \QQ[X], \; C_k \in \QQ 
\end{equation}
for all $k=1,\dots,n-1$.  The signs can then easily 
be arranged such that $A_k > 0$ and $C_k \ge 0$, 
which ensures that we have a Sturm chain according to 
Proposition \ref{prop:AlgebraicSturmCriterion}.

The euclidean algorithm for polynomials 
of degree $\le n$ takes $O(n^3)$ operations in $\QQ$.
A suitable divide-and-conquer algorithm \cite[chap.\,11]{GathenGerhard:2003}
reduces this to $\tilde{O}(n^2)$ operations in $\QQ$;
here the asymptotic complexity $\tilde{O}(n^\alpha)$ neglects 
logarithmic factors $\log(n)^\beta$. 
A closer look reveals that we only need the data $A_k,B_k,C_k$ for $k=1,\dots,n-1$, 
and these can be calculated with only $\tilde{O}(n)$ operations in $\QQ$.
Given $S_0,S_1$ and $A_k,B_k,C_k$ for all $k$, we can 
evaluate $S_0(x),S_1(x),\dots,S_n(x)$ at any given $x \in \QQ$ 
using the recursion \eqref{eq:RationalChain} with $O(n)$ operations in $\QQ$.
Finally, we have to control the size of the coefficients that appear during the computation.
According to Lickteig--Roy \cite{LickteigRoy:2001}, the result is the following.



\begin{theorem} \label{thm:FastCauchyIndexComputation}
  Given polynomials $R,S \in \ZZ[X]$ of degree $\le n$ and coefficients bounded by $2^a$,
  the Cauchy index $\Ind_0^1(\frac{R}{S})$ can be computed using $\tilde{O}(n^2 a)$ bit-operations.
  \qed
\end{theorem}




This can be applied to locating complex roots.
Let $F = c_0 Z^n + c_1 Z^{n-1} + \dots + c_n$ be a polynomial 
with Gaussian integer coefficients $c_0,c_1,\dots,c_n \in \ZZ[i]$
bounded by $\abs{\re c_k} < 2^a$ and $\abs{\im c_k} < 2^a$ for all $k=0,\dots,n$.
For simplicity we further assume that $n < 2^a$ and $a \le nb$, 
where $b$ is the desired bit-precision for approximating the roots.

\begin{corollary}
  Suppose that all roots of $F$ lie in the disk $B(r)$.
  The Sturm--Cauchy algorithm determines all roots of $F$ 
  to a precision $\sqrt{2} \, r / 2^b$ using $\tilde{O}( n^4 b^2 )$ bit-operations.
\end{corollary}

\begin{proof}
  According to Theorem \ref{thm:FastCauchyIndexComputation},
  we can compute $\Ind_0^1(\frac{\re F}{\im F})$ using $\tilde{O}(n^2 a)$ bit-operations.
  We can reparametrize $F$ to calulcate the index along any line segment, and thus along the boundary of any rectangle.
  In the Sturm--Cauchy algorithm (\sref{sub:GlobalRootFinding}), this has to be iterated $b$ times in order to 
  achieve the desired precision, and the coefficients are bounded by $2^{a+nb}$. 
  Since we assume all roots of $F$ to be distinct, they ultimately become separated
  so that the algorithm has to follow $n$ approximations in parallel.
  This multiplies the previous bound by a factor $n b$,
  so we arrive at $\tilde{O}( n^3 b (a+nb) )$ bit-operations.
\end{proof}

\begin{Annotation}
  \textbf{(Data structures)}
  The \emph{construction} of the Sturm chain 
  is the most expensive step in the above root-finding algorithm.
  In the real case we have to construct this chain
  only once because we can reuse it in all subsequent iterations.
  In the complex case, each segment requires a separate computation:
  it is thus advantageous to store each segment with 
  its corresponding Sturm chain, and each square 
  with the four Sturm chains along the boundary,
  so as to reuse precious data as much as possible.
\end{Annotation}  

\begin{Annotation}
  \textbf{(Algorithms)}
  The algebraic algorithm is straightforward to implement 
  except for two standard subalgorithms, namely fast integer arithmetic 
  and fast subresultant computation for integer polynomials.  
  These subalgorithms are theoretically well-understood,
  and their complexity is known and nearly optimal.
  Their implementation is laborious, but is available 
  in general-purpose libraries for integer and polynomial arithmetic.

  The algebraic algorithm uses exact arithmetic.  This means that during its execution 
  we do not have to worry about error propagation, which simplifies (formal) correctness proofs.
\end{Annotation}  

\begin{Annotation}
  \textbf{(Parallelization)}
  We can adapt the algorithm to find only \emph{one} root of $F$,
  and according to the preceding proof its complexity is $\tilde{O}( n^3 b^2 )$,
  again neglecting terms of order $\log(n)$.  This approach is parallelizable: 
  whenever bisection separates the roots into nonempty clusters,
  these can then be processed by independent computers working in parallel.
  The parallel complexity thus drops to $\tilde{O}( n^3 b^2 )$.
\end{Annotation}

To which bit-precision $b$ should we apply this algorithm?  Here is an a priori estimate.


\begin{corollary}
  We can switch to Newton's method after at most $b = 3 n a$ iterations 
  in the Sturm--Cauchy algorithm .  This amounts to $\tilde{O}( n^6 a^2 )$ bit-operations.
\end{corollary}

\begin{proof}
  Given $F = c_0 Z^n + c_1 Z^{n-1} + \dots + c_n = c_0 (Z-z_1) \cdots (Z-z_n)$ 
  as above with $c_0 \ne 0$, its discriminant $\disc(F) = c_0^{2n-2} \prod_{j<k} (z_j-z_k)^2$ 
  is an integer polynomial in the coefficients $c_0,c_1,\dots,c_n$.
  Here $c_0,c_1,\dots,c_n \in \ZZ[i]$, so $\disc(F) \in \ZZ[i]$.
  Since we assume $z_1,\dots,z_n$ to be pairwise distinct, 
  we have $\disc(F) \ne 0$, whence $\abs{\disc(F)} \ge 1$.
  According to Mahler \cite{Mahler:1964}, the minimal root distance 
  $\Delta(F) := \min_{j \ne k} \abs{ z_j - z_k }$ is bounded below by 
  \[ \Delta(F) > \sqrt{ 3 \, \abs{ \disc(F) } \, / n^{n+2} } \; \abs{F}^{1-n} , \]
  where $\abs{F} = \abs{c_0} + \abs{c_1} + \dots + \abs{c_n}$.  
  Our hypothesis $\abs{\re c_k} \le 2^a-1$ and $\abs{\im c_k} \le 2^a-1$
  implies $\abs{c_k} \le \sqrt{2} (2^a-1)$ for all $k=0,1,\dots,n$. 
  By Proposition \ref{prop:CauchyBound}, the zeros $z_1,\dots,z_n$ lie in the disk $B(r)$ of radius $r = \frac{3}{2} \cdot 2^a$.
  After $b$ quadrisections of the square $[-r,+r]^2$, 
  we have approximate roots $u_k \in \bar{B}(z_k,\delta_k)$ with $\delta_k \le \sqrt{2} \, r / 2^b$. 
  Assuming $b = 3 n a$ and $2^a > n$ we find, after some calculation, that $3 n \delta_k < \Delta(F)$, 
  so we can apply Theorem \ref{thm:GlobalLocalSwitch}.
\end{proof}

\begin{Annotation}
  \textbf{(Detailed calculations)}
  We have $\abs{F} \le (n+1) \sqrt{2} \cdot (2^a-1)$ and $(n+1) \le 2^a$.
  \begin{align*}
    3 n \delta_k & \le 3 n \sqrt{2} \, r / 2^b = \frac{9}{2}\sqrt{2} \, n \, 2^{a-3na}
    = \frac{9}{2}\sqrt{2} \, n \, 2^{-na/2} \cdot (2^a)^{-3n/2} \cdot (2^a)^{1-n}
    \\ & \le \frac{9}{2}\sqrt{2} \, n \, 2^{-na/2} \cdot (n+1)^{-3n/2} \cdot (n+1)^{n-1} \abs{F}^{1-n}
    \\ & \le \frac{9}{2}\sqrt{2} \, n \, 2^{-na/2} \Bigl( \frac{n}{n+1} \Bigr)^{n/2+1} \cdot n^{-n/2-1} \cdot \abs{F}^{1-n}
    \\ & \le \sqrt{3} \cdot n^{-n/2-1} \cdot \abs{F}^{1-n} < \Delta(F) .
  \end{align*}
  We should point out that the pessimistic bound $\abs{\disc(F)} \ge 1$ ensures root separation even in the worst case;
  for most polynomials the algorithm will separate the roots somewhat faster.
\end{Annotation}

\subsection{What remains to be improved?} \label{sub:AlgorithmicImprovements}

Root-finding algorithms of bit-complexity $\tilde{O}( n^3 b )$
are state-of-the-art since the ground-breaking work of Sch\"onhage 
\cite[Thm.\,19.2]{Schoenhage:1986} in the 1980s. 
The Sturm--Cauchy algorithm is of complexity $\tilde{O}( n^4 b^2 )$ 
and thus comes close, but in its current form remains two orders
of magnitude more expensive.  Sch\"onhage remarks:
\begin{quote}
  ``It is not clear whether methods based on Sturm sequences can possibly 
  become superior.  Lehmer \cite{Lehmer:1961,Lehmer:1969} and Wilf \cite{Wilf:1978} 
  both do not solve the extra problems which arise, if there is a zero 
  on the test contour (circle or rectangle) or very close to it.''
  \cite[p.\,5]{Schoenhage:1986} 
\end{quote}

Our algebraic development neatly solves the problem of roots on the boundary.
Regarding complexity, we have applied the \emph{divide-and-conquer} paradigm 
in the arithmetic subalgorithms (\sref{sub:SubresultantAlgorithm})
but not yet in the root-finding method itself.
In Sch\"onhage's method this is achieved by approximately factoring $F$ 
of degree $n$ into two polynomials $F_1, F_2$ of degrees close to $\frac{n}{2}$.
Perhaps an analogous strategy can be put into practice in the algebraic setting;
some clever idea and a more detailed investigation are needed here.

\begin{Annotation}
  \textbf{(Root isolation)}
  Sch\"onhage's algorithm is extremely well-crafted: it achieves root isolation
  using only $\tilde{O}(n^3 a)$ bit operations for any (square-free) polynomial
  of degree $n$ with integer coefficients bounded by $2^a$ \cite[Thm.\,20.1]{Schoenhage:1986}.
  This satisfactorily solves the problem of root finding: as explained above in Theorem \ref{thm:GlobalLocalSwitch}, 
  one can then switch to Newton's method and increase the precision to any desired accuracy with little extra cost.
\end{Annotation}  

Besides complexity there is still another problem:
approximating the roots of a polynomial $F \in \CC[Z]$ can only be as good as 
the initial data, and we therefore assume that $F$ is known exactly.
This is important because root-finding can be ill-conditioned 
\cite{Wilkinson:1959}. 
Even if exact arithmetic can avoid this problem during the computation, 
it comes back into focus when the initial data is itself only an approximation.  
In this situation the real-algebraic approach requires a detailed error analysis,
ideally in the setting of interval arithmetic. 



\subsection{Formal proofs} \label{sub:FormalProof}

In recent years the theory and practice of \emph{formal proofs} and 
\emph{computer-verified theorems} has become a full fledged enterprise.
Prominent examples include the Jordan Curve Theorem \cite{Hales:2007}
and the Four Colour Theorem \cite{Gonthier:2008}.
(For an overview of some ``top 100'' theorems see \cite{Wiedijk:Top100}.)
Driven by these achievements, the computer-verified proof community envisages 
much more ambitious goals, such as the classification of finite simple groups.
%
Such gigantic projects make results like the Fundamental Theorem of Algebra 
look like toy examples, but their formalization is by no means a trivial task.
%
The real-algebraic approach 
offers certain inherent advantages, mainly its simplicity and algorithmic nature.
The latter is an important virtue: Theorem \ref{thm:RootLocation} 
is not only an existence statement but provides an algorithm.  
A formal proof of the theorem can thus serve as a formal 
correctness proof of an implementation. 

\begin{Annotation}
  \textbf{(Ongoing debate)}
  Computer-assisted proofs have been intensely debated and 
  their scope and mathematical reliability have been questioned.
  The approach is still in its infancy compared to traditional viewpoints
  and its long-ranging impact on mathematics remains to be seen,
  but its achievements are already promising.
  
  We should emphasize that the formalization of mathematical 
  theorems and proofs and their computer verification may be motivated 
  by several factors.  Some theorems, of varying difficulty,
  have been formalized in order to show that this 
  is possible in principle and to gain practical experience.
  While pedagogically important for proof formalization itself,
  the traditional mathematician will find no added value in such examples.
  
  More complicated theorems, such as the examples above, 
  raise the intrinsic motivation for formalization and 
  computer-verified proofs, because there is an enormous number 
  of cases to be solved and verified.  Whenever human fallibility 
  becomes a serious practical problem, as in these cases, 
  a well established and trustworthy verification tool clearly has its merit.
  This is particularly true if the mathematical model is implemented 
  on a computer for practical applications, and a high level of security is required.  
  It is in this realm that computer-assisted correctness proofs are most widely appreciated.
\end{Annotation}


\section{Historical remarks} \label{sec:HistoricalRemarks}

The Fundamental Theorem of Algebra is a crowning achievement in the history of mathematics.
In order to place the real-algebraic approach into perspective, 
this section sketches its historical context.
For the history of the Fundamental Theorem of Algebra
we refer to Remmert \cite{Remmert:1991},
Dieudonn\'e \cite[chap.\,II, \textsection III]{Dieudonne:1978},
and van der Waerden \cite[chap.\,5]{vanDerWaerden:1985}. 
The history of Sturm's theorem has been examined 
in great depth by Sinaceur \cite{Sinaceur:1991}.

\subsection{Polynomial equations} \label{sub:Prehistory}

The method to solve quadratic equations was already known to the Babylonians.
Not much progress was made until the 16th century, when
del Ferro (around 1520) and Tartaglia (1535) 
discovered a solution for cubic equations by radicals.
Cardano's student Ferrari extended this to a solution of quartic equations by radicals.
Both formulae were published in Cardano's \textit{Ars Magna} in 1545.  
Despite considerable efforts during the following centuries,
no such formulae could be found for degree $5$ and higher.  They were finally
shown not to exist by Ruffini (1805), Abel (1825), and Galois (1831).
This solved one of the outstanding problems of algebra, alas in the negative.

The lack of general formulae provoked the question whether solutions exist at all.
The existence of $n$ roots for each real polynomial of degree $n$
was mentioned by Roth (1608) and explicitly conjectured by Girard (1629) 
and Descartes (1637).  They postulated these roots in some extension of $\RR$ 
but did not claim that all roots are contained in the field $\CC = \RR[i]$ of complex numbers.
Leibniz (1702) even speculated that this is in general not possible.
The first attempts to prove the Fundamental Theorem of Algebra were made 
by d'Alembert (1746), Euler (1749), Lagrange (1772), and Laplace (1795).

\subsection{Gauss' geometric proof} \label{sub:Gauss}

In his doctoral thesis (1799) Gauss criticized the shortcomings 
of all previous tentatives and presented a geometric argument, 
which is commonly considered the first satisfactory proof of the Fundamental Theorem of Algebra.

In summary, Gauss considers a polynomial $F = Z^n + c_1 Z^{n-1} + \dots + c_{n-1} Z + c_n$
and upon substitution of $Z = X + i Y$ obtains $F = R + i S$ with $R,S \in \RR[X,Y]$.
The zeros of $F$ are precisely the intersections 
of the two curves $R=0$ and $S=0$ in the plane.
Consider a disk $\Gamma$ centered in $0$ with sufficiently large radius.
Near the circle $\partial\Gamma$ these curves resemble 
the zero sets of the real and imaginary parts of $Z^n$.  
The latter are $2n$ straight lines passing through the origin.
Thus $\partial\Gamma$ intersects the curves $R=0$ and $S=0$ 
in two sets of $2n$ points placed in an alternating fashion around the circle.
(See Figure \ref{fig:GaussCurves}.)

\begin{figure}[ht]
  \captionsetup[subfloat]{width=26ex}
  \subfloat[width=25ex][The curves $R=0$ and $S=0$ outside of a sufficiently large disk $\Gamma$.]{\includegraphics[height=25ex]{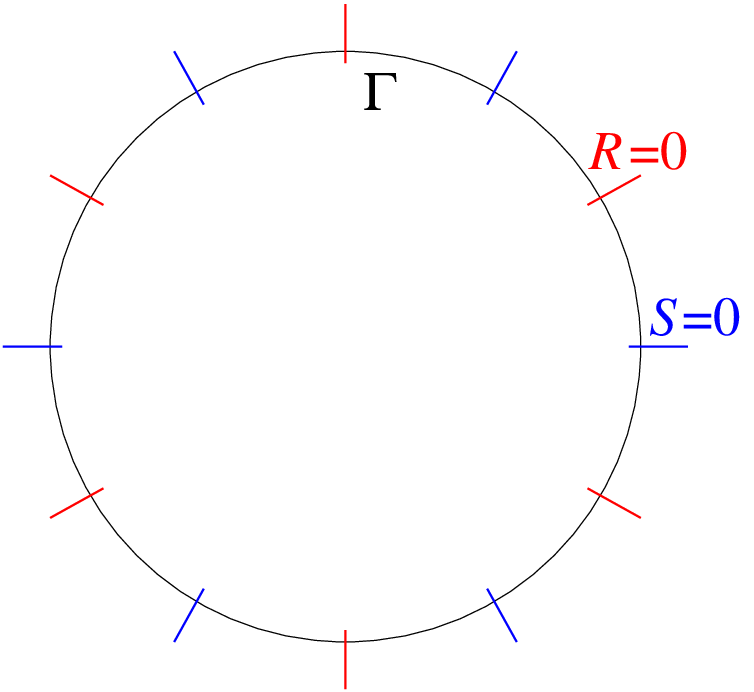}}
  \hfill
  \captionsetup[subfloat]{width=22ex}
  \subfloat[Joining the ends inside of $\Gamma$ forces the curves to intersect.]{\includegraphics[height=25ex]{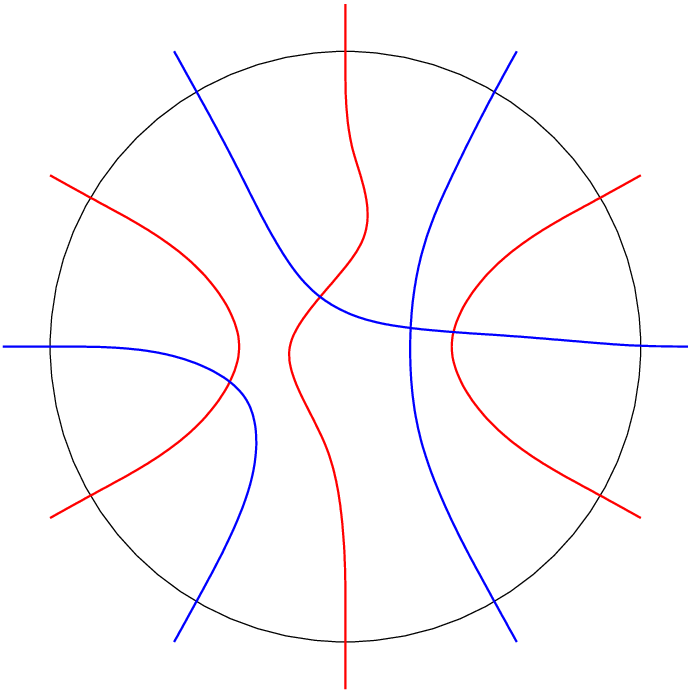}}
  \hfill
  \subfloat[Such pathological cases have to be ruled out, of course.]{\includegraphics[height=25ex]{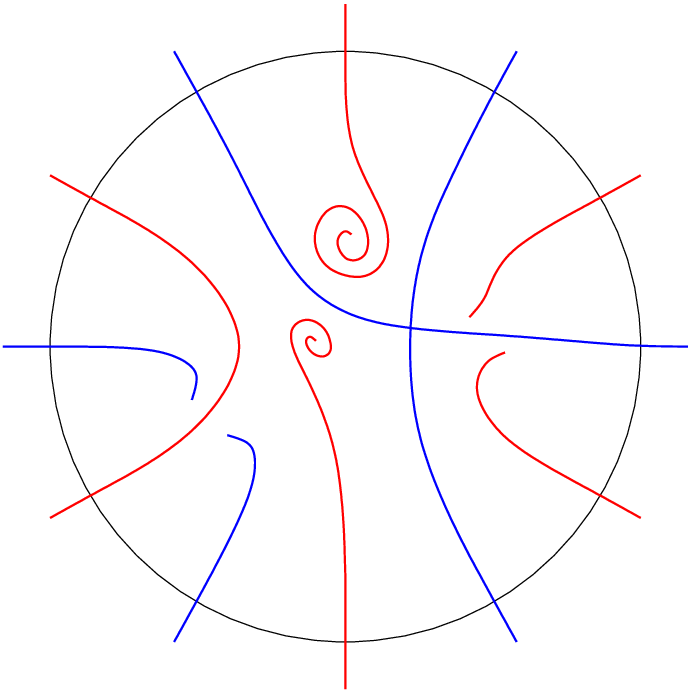}}
  \caption{Gauss' geometric argument for the existence of zeros}
  \label{fig:GaussCurves}
\end{figure}

Prolonging these curves into the interior of $\Gamma$, Gauss concludes 
that the curves $R=0$ and $S=0$ must intersect somewhere inside the disk $\Gamma$.  
The conclusion relies on certain (intuitively plausible)
assumptions, which Gauss clearly states but does not prove:
\begin{quote} 
  ``It seems to have been proved with sufficient certainty that 
  an algebraic curve can neither suddenly break off anywhere 
  (as it happens, for example, with the transcendental curve whose equation
  is $y = 1/\log x$) nor lose itself, so to say, in some point 
  after infinitely many coils (like the logarithmic spiral).  
  As far as I know, nobody has raised any doubts about this. 
  Should someone demand it, however, then I will undertake 
  to give a proof that is not subject to any doubt, on some other occasion.''%
  \footnote{ 
    ``Satis bene certe demonstratum esse videtur, 
    curvam algebraicam neque alicubi subito abrumpi posse 
    (uti e.g.\ evenit in curva transscendente, cuius aequatio $y = 1/\log x$), 
    neque post spiras infinitas in aliquo puncto se quasi perdere 
    (ut spiralis logarithmica), quantumque scio nemo dubium contra hanc rem movit. 
    Attamen si quis postulat, demonstrationem nullis dubiis obnoxiam 
    alia occasione tradere suscipiam.'' \cite[Bd.\,3, p.\,27]{Gauss:Werke}
    \selectlanguage{english} 
    My translation is adapted from Prof.\ Ernest Fandreyer's 
    (Fitchburg State College Library, Manuscript Collections),
    cf.\ van der Waerden \cite[p.\,96]{vanDerWaerden:1985}. }
\end{quote}



This amounts to a version of the Jordan Curve Theorem \cite{GerstenStallings:1988}.
By modern standards Gauss' geometric argument is thus incomplete.
The unproven assertions are indeed correct, and were rigorously worked out 
by Ostrowski \cite{Ostrowski:1920,Ostrowski:1933} more than a century later.
Gauss' ingenious insight was to apply geometric arguments to an algebraic problem.
In terms of winding numbers he shows $\cind{\partial\Gamma}{F} = n$ by an implicit homotopy $F \sim Z^n$.
Our development shows how to complete the proof using real-algebraic techniques. 

\begin{Annotation}
  \textbf{(How serious is this gap?)}
  Gersten and Stallings [\textit{On Gauss's first proof of the fundamental theorem of algebra}, 
    Proc.\ Amer.\ Math.\ Soc.\ 103 (1988), 331--332]:
  ``It is remarkable that Gauss himself was not able to fill in the details, 
  except to give a plausibility argument at the point in the proof where the free group enters; 
  the deep point in this argument involves the geometry of the 2-cell, 
  especially a version of the Jordan Curve Theorem which is used in the geometric theory of free groups.''
  Smale [\textit{The fundamental theorem of algebra and complexity theory}, Bull.\ Amer.\ Math.\ Soc.\ 4 (1981), 1--36]:
  ``I wish to point out what an immense gap Gauss' proof contained. 
  It is a subtle point even today that a real algebraic plane curve cannot enter a disk without leaving. 
  In fact even though Gauss redid this proof 50 years later, the gap remained. 
  It was not until 1920 that Gauss' proof was completed.''
  Ostrowksi's arguments essentially use compactness to prove existence and are not constructive.
  For the real-algebraic proof the intermediate value property of polynomials suffices;
  moreover, it makes the proof constructive and provides computational tools.
\end{Annotation}

Gauss gave two further proofs in 1816;
the second proof is algebraic (\sref{ssub:Algebra}), whereas
the third proof uses integration (\sref{ssub:Topology}) and 
foreshadows Cauchy's integral formula for the winding number.
Gauss' fourth proof in 1849 is essentially an improved 
version of his first proof \cite[chap.\,5]{vanDerWaerden:1985}.
When Gauss published it for his doctorate jubilee, 
the works of Sturm (1835) and Cauchy (1837) had been known for several years.
In particular Sturm's theorem had immediately risen 
to international acclaim, and was certainly familiar to Gauss.  
Gauss could have taken up his first proof and 
completed it by arguments similar to the ones presented here.
Completing Gauss' geometric argument, Ostrowski \cite{Ostrowski:1933} mentions 
the relationship with the Cauchy index but builds his proof on topological arguments.

\begin{Annotation}
  \textbf{(Historical sources)}
  In 1831 Gauss discussed Fourier's work \textit{Analyse des \'equations d\'etermin\'ees} 
  on counting and locating real roots (Werke, Band 3, pp.\,119--121).
  Kronecker attributes a certain example of Sturm chains to Gauss in 1849
  in his course \textit{Theorie der algebraischen Gleichungen} (1872),
  notes written by Kurt Hensel, archived at the University of Strasbourg,
  available at \link{num-scd-ulp.u-strasbg.fr/429}, page 165.
  Unfortunately I could not find this example or similar ones in Gauss' collected works.
  A more detailed research would be necessary to clarify
  the reception of the ideas of Sturm and Cauchy in Gauss' writings.
\end{Annotation}

\subsection{Cauchy, Sturm, Liouville} \label{sub:CauchySturmLiouville}

Argand in 1814 and Cauchy in 1820 proved the Fundamental Theorem of Algebra
by assuming the existence of a global minimum $z_0$ of $\abs{F}$ and 
a local argument to show that $F(z_0)=0$; see \sref{ssub:Analysis}.
While the local analysis is rigorous, the existence of a minimum
requires some compactness argument, which was yet to be developed;
see Remmert \cite[\textsection 1.8]{Remmert:1991}.

Sturm's theorem for counting real roots 
was announced in 1829 \cite{Sturm:1829} 
and published in 1835 \cite{Sturm:1835}.  
It was immediately assimilated by Cauchy in his residue 
calculus \cite{Cauchy:1831}, based on contour integration, 
which was published in 1831 during his exile in Turin.
In 1837 he published a more detailed exposition \cite{Cauchy:1837} 
with analytic-geometric proofs, and explicitly recognizes the relation 
to Sturm's theorem and algebraic computations. 

In the intervening years, Sturm and Liouville \cite{SturmLiouville:1836,Sturm:1836}
had elaborated their algebraic version of Cauchy's theorem, which they published in 1836.
(Loria \cite{Loria:1938} and Sinaceur \cite[I.VI]{Sinaceur:1991}
examine the interaction between Sturm, Liouville, and Cauchy in detail.)
As opposed to Cauchy, their arguments are based on 
what they call the ``first principles of algebra''.
In the terminology of their time this means the theory of complex numbers,
including trigonometric coordinates $z = r(\cos\theta + i \sin\theta)$ and de Moivre's formula, but excluding integration.
They use the intermediate value property of real polynomials as well as tacit compactness arguments.

\subsection{Sturm's algebraic vision} \label{sub:SturmsAlgebraicVision}

Sturm, in his article \cite{Sturm:1836} continuing his work with Liouville \cite{SturmLiouville:1836}, 
presents arguments which closely parallel our real-algebraic proof: 
the argument principle (Prop.\,1, p.\,294), multiplicativity (Prop.\,2, p.\,295), 
counting roots of a split polynomial within a given region (Prop.\,3, p.\,297), 
the winding number in the absence of zeros (Prop.\,4, p.\,297), and finally Cauchy's theorem (p.\,299).
One crucial step is to show that $\cind{\partial\Gamma}{F} = 0$ when $F$ does not vanish in $\Gamma$. 
This is solved by subdivision and a tacit compactness argument (pp.\,298--299); 
our compactness proof of Lemma \ref{lem:GlobalIndexZero} makes this explicit and completes his argument.
Sturm then deduces the Fundamental Theorem of Algebra (pp.\,300--302)
and expounds on the practical computation of the Cauchy index $\cind{\partial\Gamma}{F}$
using Sturm chains as in the real case (pp.\,303--308). 

Sturm's exposition strives for algebraic simplicity, but 
his proofs are still based on geometric and analytic arguments.
It is only on the final pages that Sturm employs his 
algebraic method for computing the Cauchy index.
This mixed state of affairs has been passed on ever since, even though 
it is far less satisfactory than Sturm's purely algebraic treatment of the real case \cite{Sturm:1835}.
Our proof shows that Sturm's algebraic vision of the complex case can be salvaged 
and his arguments can be put on firm real-algebraic ground.

We note that Sturm and Liouville \cite{SturmLiouville:1836} explicitly exclude zeros on the boundary:
\begin{quote} 
  ``We formally exclude, however, the case where for some point 
  of the curve we have simultaneously $P=0$ and $Q=0$:
  this special case does not enjoy any regular property 
  and cannot give rise to any theorem.''%
  \footnote{ \selectlanguage{french} 
    ``Toutefois nous excluons formellement le cas particulier o\`u, 
    pour quelque point de la courbe $ABC$, on aurait \`a la fois $P=0$, $Q=0$: 
    ce cas particulier ne jouit d'aucune propri\'et\'e r\'eguli\`ere 
    et ne peut donner lieu \`a aucun th\'eor\`eme.''
    \cite[p.\,288]{SturmLiouville:1836} }
\end{quote}
This seems overly pessimistic in view of our Theorem \ref{thm:RootCounting} above.
In his continuation \cite{Sturm:1836}, Sturm formulates the same problem more cautiously:
\begin{quote}
  ``It is under this hypothesis that we have proven the theorem of Mr.\ Cauchy; 
  the necessary modifications in the case where roots were on the contour 
  would require a long and meticulous discussion, which we wanted 
  to avoid by neglecting this special case.''%
  \footnote{ \selectlanguage{french}
    ``C'est en admettant cette hypoth\`ese que nous avons d\'emontr\'e le 
    th\'eor\`eme de M.\ Cauchy; les modifications qu'il faudrait y apporter
    dans le cas o\`u il aurait des racines sur le contour m\^eme $ABC$,
    exigeraient une discussion longue et minutieuse que nous avons voulu
    \'eviter en faisant abstraction de ce cas particulier.''
    \cite[p.\,306]{Sturm:1836} }
\end{quote}
It seems safe to say that our detailed discussion is just as 
``long and meticulous'' as the usual development of Sturm's theorem.
Modulo these details, the cited works of Gauss, Cauchy, and Sturm 
contain the essential ideas for the real-algebraic approach.  

\begin{Annotation}
  To this end our presentation refines the techniques in several ways:
  \begin{itemize}
  \item
    We purge all arguments of transcendental functions 
    and compactness assumptions.  This simplifies 
    the proof and generalizes it to real closed fields.
  \item
    The product formula (\sref{sub:ProductFormula})
    and homotopy invariance (\sref{sub:HomotopyInvariance})
    streamline the proof and avoid tedious calculations.
  \item
    The uniform treatment of boundary points extends 
    Sturm's theorem to piecewise polynomial functions
    and leads to straightforward algorithms.
  \end{itemize}
\end{Annotation}

\subsection{Further development in the 19th century} \label{sub:FurtherDevelopment}


Sturm's theorem was a decisive step in the development 
of algebra as an autonomous field; see Sinaceur \cite{Sinaceur:1991}.
Algebraic generalizations to higher dimensions were conjectured 
by Sylvester in 1840 and developped by Hermite from 1852 onwards.
In 1869 Kronecker \cite{Kronecker:1869} turned from algebra to integration
in order to construct his higher-dimensional index (also called Kronecker characteristic).
Subsequent work was likewise built on analytic or topological methods over $\RR$:
one gains in generality by extending the index to smooth or continuous functions, 
but one loses algebraic computability and the elementary setting of real closed fields.



\subsubsection{Applications}

Generalizing Example \ref{example:Signature}, the problem of \emph{stability of motion} 
led Routh \cite{Routh:1878} in 1878 and Hurwitz \cite{Hurwitz:1895} in 1895 to count, 
for a given polynomial, the number of complex roots having negative real part. 
With the celebrated Routh--Hurwitz theorem, the algebraic Cauchy index 
has transited from algebra to application, where it survives to the present day.


\subsubsection{Encyclopaedic surveys}

In the 1898 \emph{Encyklop\"adie der mathematischen Wissenschaften} \cite{EncyklMathWiss}, 
Netto's survey on the Fundamental Theorem of Algebra (Kap.\,IB1, \textsection a7) 
mentions Cauchy's algebraic approach only briefly (p.\,236), whereas
Runge's article on approximation of complex roots (Kap.\,IB3, \textsection a6)
discusses the Sturm--Cauchy method in detail (pp.\,418--422). 
In the 1907 \emph{Encyclop\'edie des Sciences Math\'ematiques} \cite{EncyclSciMath},
Netto and le Vavasseur give an overview of nearly 100 published proofs
(tome I, vol.\,2, chap.\,I-9, \textsection 80--88), 
including Cauchy's argument principle (\textsection 87).
The work of Sturm--Liouville \cite{SturmLiouville:1836,Sturm:1836} 
is listed but the algebraic approach via Sturm chains is not mentioned.


\subsubsection{Nineteenth-century textbooks}

While Sturm's theorem made its way to modern algebra textbooks, the 
algebraic approach to the complex case seems to have been lost on the way.
I will illustrate this by two prominent and perhaps representative textbooks.

In his 1866 textbook \textit{Cours d'alg\`ebre sup\'erieure}, starting with the 3rd edition,
Serret \cite[pp.\,117--131]{Serret:1866} presents the proof of 
the Fundamental Theorem of Algebra following Cauchy and Sturm--Liouville, 
with only minor modifications.  

In his 1898 textbook \emph{Lehrbuch der Algebra}, 
Weber \cite{Weber:1898} devotes over 100 pages to real-algebraic equations,
where he presents Sturm's theorem in great detail (\textsection 91--106).  
Calling upon Kronecker's index theory (\textsection 100--102),
he sketches how to count complex roots (\textsection 103--104).
Quite surprisingly, he uses only $\Ind\bigl(\frac{P'}{P})$ 
and Corollary \ref{cor:CountingRealRoots} where the general case 
$\Ind\bigl(\frac{R}{S})$ and Theorem \ref{thm:Sturm} would have been optimal.
Here Cauchy's algebraic method \cite{Cauchy:1837}, apparently  unknown to Weber,
had gone much further concerning explicit formulae and concrete computations.

\subsection{Survey of proof strategies} \label{sub:ProofStrategies}

Since the time of Gauss numerous proofs of the
Fundamental Theorem of Algebra have been developed.
We refer to Remmert \cite{Remmert:1991} for a concise overview and 
to Fine--Rosenberger \cite{FineRosenberger:1997} for a textbook presentation.
As mentioned in \sref{sub:HistoricalOrigins},
the proof strategies can be grouped into three families:

\subsubsection{Analysis} \label{ssub:Analysis}

Early proofs in this family are based on the existence of a global minimum 
$z_0$ of $\abs{F}$ and some local argument from complex analysis showing that $F(z_0)=0$
(d'Alembert 1746, Argand 1814, Cauchy 1820). 
See Remmert \cite[\textsection 2]{Remmert:1991} for a presentation in its historical context,
or Rudin \cite[chap.\,8]{Rudin:1976} in the context of a modern analysis course.
The most succinct formulation follows from Liouville's theorem for entire functions.

These existence proofs are in general not constructive and do not indicate the location of zeros.
For a discussion of constructive refinements see \cite[\textsection 2.5]{Remmert:1991}.

\subsubsection{Algebra} \label{ssub:Algebra}

Proofs in this family use the fundamental theorem of symmetric polynomials 
in order to reduce the problem from real polynomials 
of degree $2^k m$ with $m$ odd to degree $2^{k-1} m'$ with $m'$ odd
(Euler 1749, Lagrange 1772, Laplace 1795, Gauss 1816; see \cite[appendix]{Remmert:1991}).
The argument can also be reformulated using Galois theory;
see Cohn \cite[Thm.\,8.8.7]{Cohn:2003},
Jacobson \cite[Thm.\,5.2]{Jacobson:1985:1989},
or Lang \cite[\textsection VI.2, Ex.\,5]{Lang:2002}.
The induction is based, for $k=0$, on real polynomials of odd degree,
where the existence of at least one real root 
is guaranteed by the intermediate value theorem.

This algebraic proof works over every real closed field,
as elaborated by Artin and Schreier \cite{ArtinSchreier:1926} in 1926.
It is constructive but ill-suited to actual computations.

\subsubsection{Topology} \label{ssub:Topology}

Proofs in this family use some form of the winding number 
$\Wind(\gamma)$ of closed paths $\gamma \colon [0,1] \to \CC^*$ 
(Gauss 1799/1816, Cauchy 1831/37, Sturm--Liouville 1836).
The winding number appears in various guises;
see Remark \ref{rem:WindingNumberConstruction}.
In each case the difficulty is a rigorous construction 
and to establish its characteristic properties: 
normalization, multiplicativity and homotopy invariance, 
as stated in Theorem \ref{thm:AlgebraicWindingNumber}.

Our proof belongs to this last family.  Unlike previous proofs, however,
we do not base the winding number on analytic or topological arguments,
but on real algebra.

\subsection{Constructive and algorithmic aspects} \label{sub:ConstructiveAspects}

Sturm's method is eminently practical, by the standards of 19th 
century mathematics as for modern-day implementations.
As early as 1840 Sylvester \cite{Sylvester:1840} wrote
``Through the well-known ingenuity and proferred help of a distinguished friend,
I trust to be able to get a machine made for working Sturm's theorem (\dots)''.
It seems, however, that such a machine was never built.
Calculating machines had been devised by Pascal, Leibniz, and Babbage;
the latter was Lucasian Professor of Mathematics when Sylvester studied at Cambridge in the 1830s.
The idea of computing machinery seems to have been popular among mid-19th century mathematicians.
For example, in a small note of 1846, Ullherr \cite{Ullherr:1846} remarks that the argument principle 
``provides a method to find the roots of higher-degree equations by means of a mechanical apparatus.''%
\footnote{ \selectlanguage{german}
  ``Die bei dem ersten Beweise gebrauchte Betrachtungsart 
  giebt ein Mittel an die Hand, die Wurzeln der h\"oheren Gleichungen 
  mittels eines Apparates mechanisch zu finden.'' 
  \cite[p.\,234]{Ullherr:1846} } 

For separating and approximating roots, the state of the art at the end of 
the 19th century was surveyed in Runge's \emph{Encyklop\"adie} article 
\cite[Kap.\,IB3, \textsection a]{EncyklMathWiss}, and in particular
the Sturm--Cauchy method is discussed in detail (pp.\,416--422). 


In 1924 Weyl \cite{Weyl:1924} reemphasized that the analytic winding number 
can be used to find and approximate the roots of $F$. 
In this vein Weyl formulated his constructive proof of the Fundamental
Theorem of Algebra, which indeed translates to an algorithm:
a careful numerical approximation can be used to calculate 
the integer $\cind{\partial\Gamma}{F}$; see Henrici \cite[\textsection 6.11]{Henrici:1974}. 
While Weyl's motivation may have been philosophical,
it is the practical aspect that has proven most successful.
Variants of Weyl's algorithm are used in modern computer 
implementations for finding approximate roots, and 
are among the asymptotically fastest known algorithms.
The question of algorithmic complexity was pursued 
by Sch\"onhage \cite{Schoenhage:1986} and others since the 1980s.
See Pan \cite{Pan:1998} for an overview.

The fact that Sturm's and Cauchy's theorems can be combined 
to count complex roots seems not to be as widely known as it should be.
It is surprising that the original publications in the 1830s 
did not have a lasting effect (\sref{sub:FurtherDevelopment}) and 
likewise Runge's presentation 
in the 1898 \emph{Encyklop\"adie} fell into oblivion.
In the 1969 Proceedings \cite{ConstructiveAspects:1969} on constructive aspects 
of the Fundamental Theorem of Algebra, the Sturm--Cauchy method is not mentioned.
It reappears in 1978 in a small note by Wilf \cite{Wilf:1978},
and is briefly mentioned in Sch\"onhage's report \cite[p.\,5]{Schoenhage:1982}.
Most often the computer algebra literature credits Weyl for the analytic-numeric method, 
and Lehmer or Wilf for the algebraic-numeric method, but not Cauchy or Sturm.  
Their real-algebraic method for complex root location seems largely ignored.




\begin{Appendix}

\appendix

\section{The Routh--Hurwitz stability theorem} \label{sec:Stability}

For a polynomial with only real roots, as in Example \ref{example:Signature},
Descartes' rule of signs 
quickly computes the number of negative resp.\ positive roots.
More generally, in certain applications it is important to determine, 
for a given complex polynomial $F \in \CC[Z]$, how many roots lie 
in the left half-plane $\{\, z \in \CC \mid \re(z) < 0 \,\}$.
This question originated from the theory of dynamical systems and 
the problem of \emph{stability of motion}: 

\begin{example}
  Let $A \in \RR^{n \times n}$ be a square matrix with real coefficients.
  The differential equation $y' = A y$ with initial value $y(0) = y_0$
  has a unique solution, given by $\exp(t A) y_0$.  
  In terms of dynamical systems, the origin $0$ is a fixed point; 
  it is \emph{stable} if all eigenvalues $\lambda_1,\dots,\lambda_n \in \CC$ 
  of $A$ satisfy $\re \lambda_k < 0$: 
  in this case $\exp(t A)$ has eigenvalues $\exp(t \lambda_k)$ of absolute value $<1$,
  whence $\exp(t A) \to 0$ for $t \to +\infty$.
\end{example}

\begin{example}
  The foregoing argument holds locally around fixed points 
  of any dynamical system given by a differential equation $y' = \Phi(y)$ 
  where $\Phi \colon \RR^n \to \RR^n$ is continuously differentiable.  
  Suppose that $a$ is a fixed point, i.e., $\Phi(a) = 0$.
  It is \emph{stable} if all eigenvalues of the matrix $A = \Phi'(a) \in \RR^{n \times n}$ 
  have negative real part: in this case there exists a neighbourhood $V$ of $a$ 
  that is attracted to $a$: every trajectory $f \colon \RR_{\ge0} \to \RR^n$, 
  starting at $f(0) \in V$ and satisfying $f'(t) = \Phi(f(t))$ for all $t \ge 0$, 
  satisfies $f(t) \to a$ for $t \to +\infty$.
\end{example}

In this sense, stability means that trajectories are robust under small perturbations.

Given $F \in \CC[Z]$ we can determine the number of roots with positive 
real part simply by calculating $\cind{\partial\Gamma}{F}$ with respect to
a rectangle $\Gamma = [0,r] \times [-r,r]$ for $r$ sufficiently large.
(One could use the Cauchy radius $\rho_F$ defined in \sref{sub:GlobalWindingNumber}.)
Routh's theorem, however, offers a simpler solution by 
calculating the Cauchy index along the imaginary axis.
This is usually proven using contour integration, but
here we will give a real-algebraic proof.  As before we consider
a real closed field $\R$ and its extension $\C = \R[i]$ with $i^2 = -1$.

\begin{definition} \label{def:RouthIndex}
  For every polynomial $F \in \C[Z]^*$ we define its \emph{Routh index} as 
  \begin{equation} \label{eq:RouthIndex}
    \textstyle
    \Routh(F) := 
    \Ind_{+r}^{-r}\bigl( \frac{\re F(iY)}{\im F(iY)} \bigr)
    + \Ind_{-1/r}^{+1/r}\bigl( \frac{\re F(i/Y)}{\im F(i/Y)} \bigr) 
  \end{equation}
  for some arbitrary parameter $r \in \R_{>0}$;
  the result is independent of $r$ 
  by Proposition \ref{prop:RealIndexProperties}(b).
\end{definition}

\begin{remark} \label{rem:RouthIndexSimplified}
  We can decompose $F(iY) = R + i S$ with $R,S \in \R[Y]$ and
  compare the degrees $m = \deg S$ and $n = \deg R$.  If $m \ge n$, then 
  the fraction $\frac{R(1/Y)}{S(1/Y)} = \frac{Y^m R(1/Y)}{Y^m S(1/Y)}$
  has no pole at $0$, so the second index vanishes for $r$ sufficiently large,
  and Equation \eqref{eq:RouthIndex} simplifies to 
  \begin{equation} \label{eq:RouthIndexSimplified}
    \textstyle
    \Routh(F) = -\Ind_{-\infty}^{+\infty}\bigl( \frac{\re F(iY)}{\im F(iY)} \bigr) .
  \end{equation}
\end{remark}

\begin{example}
  In general the second index in Equation \eqref{eq:RouthIndex} cannot be neglected,
  as illustrated by $F = (Z-1)(Z-2)$: here $F(iY) = -Y^2 - 3i Y + 2$, whence 
  \[
  \textstyle
  \frac{\re F(iY)}{\im F(iY)} = \frac{Y^2 - 2}{3Y}
  \quad\text{and}\quad 
  \frac{\re F(i/Y)}{\im F(i/Y)} = \frac{1 - 2Y^2}{3Y} .
  \]
  Both indices in Equation \eqref{eq:RouthIndex}
  contribute $+1$ such that $\Routh(F) = +2$.
\end{example}

\begin{lemma}
  We have $\Routh(Z-z_0) = \sign( \re z_0 )$ for all $z_0 \in \C$.
\end{lemma}

\begin{proof}
  For $F = Z-z_0$ we find $F(iY) = R + i S$ with $R = - \re z_0$ and $S = Y - \im z_0$.  
  Thus $\Routh(F) = - \Ind_{-\infty}^{+\infty}\bigl( \frac{R}{S} ) =
  \Ind_{-\infty}^{+\infty}\bigl( \frac{\re z_0}{Y - \im z_0} ) = \sign( \re z_0 )$.
\end{proof}

\begin{lemma}
  We have $\Routh( F G ) = \Routh( F ) + \Routh( G )$ for all $F,G \in \C[Z]^*$.
\end{lemma}

\begin{proof}
  This follows from the product formula \eqref{eq:RealProductFormula}
  as in Corollary \ref{cor:Multiplicativity}.
\end{proof}

\begin{remark}
  For every $c \in \C^*$ we have $\Routh(c) = 0$, whence $\Routh( c F ) = \Routh( F )$.
  We can thus ensure the favourable situation of Remark \ref{rem:RouthIndexSimplified}:
  if $\deg S < \deg R$, then it is advantageous to pass 
  from $F$ to $i F$, that is, to replace $(R,S)$ by $(-S,R)$.
\end{remark}

We can now deduce the following formulation of the famous Routh--Hurwitz theorem:

\begin{theorem} \label{thm:Routh}
  The Routh index of every polynomial $F \in \C[Z]^*$ satisfies $Routh(F) = p - q$
  where $p$ resp.\ $q$ is the number of roots of $F$ in $\C$ having positive resp.\ negative real part.
\end{theorem}

\begin{proof}
  The Fundamental Theorem of Algebra ensures that 
  $F = c_0(Z-z_1) \cdots (Z-z_n)$ for some $c \in \C^*$ and $z_1,\dots,z_n \in \C$,
  so the Routh index follows from the preceding lemmas.
\end{proof}

\begin{remark}
  By a linear transformation $z \mapsto a z + b$, with $a \in \C^*$ and $b \in \C$,
  we can map the imaginary line onto any other straight line, 
  so we can apply the theorem to count roots in any half-space in $\C$.
  The transformation $z \mapsto \frac{z - 1}{z + 1}$
  maps $\R i \cup \{\infty\}$ onto the unit circle,
  and the right half-plane to the unit disk.
  Again by linear transformation we can thus apply 
  the theorem to count roots in any given disk in $\C$.
\end{remark}

Routh's criterion is often applied to real polynomials
$P \in \R[X]$, as in the motivating examples above,
which warrants the following more detailed formulation:

\begin{corollary} \label{cor:Routh}
  Consider $P = c_0 X^n + c_1 X^{n-1} + \dots + c_{n-1} X + c_n$ in $\R[X]$
  and denote by $p$ resp.\ $q$ the number of roots of $P$ in $\C$ having 
  positive resp.\ negative real part. Then 
  \begin{equation} \label{eq:RouthReal}
    p - q = \Routh(P) = \begin{cases}
      -\Ind_{-\infty}^{+\infty}\bigl( \frac{\re P(iY)}{\im P(iY)} \bigr)
      & \text{if $n$ is odd,} \\
      +\Ind_{-\infty}^{+\infty}\bigl( \frac{\im P(iY)}{\re P(iY)} \bigr)
      & \text{if $n$ is even.} 
    \end{cases}
  \end{equation} 
  Both cases can be subsumed into the unique formula
  \begin{equation} \label{eq:RouthRealUnified}
    q - p = \Ind_{-\infty}^{+\infty}\left( \frac{ c_1 X^{n-1} - c_3 X^{n-3} + \dots }{c_0 X^n - c_2 X^{n-2} + \dots } \right) .
  \end{equation}
  
  This implies Routh's criterion: All roots of $P$ have negative real part 
  if and only if $q=n$ and $p=0$, which is equivalent to saying that 
  the Cauchy index in \eqref{eq:RouthRealUnified} evaluates to $n$.
\end{corollary}

Routh's formulation via Cauchy indices is unrivaled in its simplicity, and 
can immediately be calculated using Sturm's theorem (\sref{sub:SturmTheorem}).
Hurwitz' formulation uses determinants, which has the advantage 
to produce explicit polynomial formulae in the given coefficients.
See 
Henrici \cite[\textsection 6.7]{Henrici:1974}, Marden \cite[chap.\,IX]{Marden:1966},
or Rahman--Schmeisser \cite[chap.\,11]{RahmanSchmeisser:2002}.


\section{Brouwer's fixed point theorem over real closed fields} \label{sec:Brouwer}

Brouwer's theorem states that every continuous map
$f \colon [0,1]^n \to [0,1]^n$ of a cube in $\RR^n$ to itself has a fixed point.
While in dimension $n=1$ this follows directly from the intermediate value theorem,
the statement in dimension $n\ge2$ is more difficult to prove: one employs either 
sophisticated machinery (differential topology, Stokes' theorem, co/homology)
or subtle combinatorial techniques (Sperner's lemma, Nash's game of Hex).
These proofs use Brouwer's mapping degree, in a more or less explicit way,
and the compactness of $[0,1]^n$.
Such proofs are often non-constructive and do not 
address the question of locating fixed points.
Using the algebraic winding number we can prove Brouwer's theorem, 
in dimension $n = 2$, in a constructive way over every real closed field.  
To this end, we have to restrict the statement from continuous to polynomial functions:


\begin{theorem} \label{thm:Brouwer}
  Let $\R$ be a real closed field and let $\Gamma = [-1,+1]^2$ in $\R^2$.
  Then for every polynomial map $f \colon \Gamma \to \Gamma$ 
  there exists $z \in \Gamma$ such that $f(z) = z$.
\end{theorem}

\begin{proof}
  We consider the homotopy $g_t = \id - t f$ from $g_0 = \id$ to $g_1 = \id - f$.
  For $z \in \partial\Gamma$ we have $g_t(z) = 0$ if and only if 
  $t = 1$ and $f(z) = z$; in this case the assertion holds.
  Otherwise, we have $g_t(z) \ne 0$ for all $z \in \partial\Gamma$ and $t \in [0,1]$.
  We can then apply homotopy invariance (Theorem \ref{thm:HomotopyInvariance})
  to conclude that $\cind{\partial\Gamma}{g_1} = \cind{\partial\Gamma}{g_0} = 1$.  
  Lemma \ref{lem:GlobalIndexZero} implies that there exists 
  $z \in \Int\Gamma$ such that $g_1(z) = 0$, whence $f(z) = z$.
\end{proof}

\begin{remark}
  As for the Fundamental Theorem of Algebra, the algebraic proof 
  of Theorem \ref{thm:Brouwer} also provides an algorithm 
  to approximate a fixed point to any desired precision
  (assuming $\R$ to be archimedean).
  Quadrisecting successively, we can construct a sequence of subsquares
  $\Gamma = \Gamma_0 \supset \Gamma_1 \supset \dots \supset \Gamma_k$
  such that $f$ has a fixed point on $\partial\Gamma_k$, 
  or $\cind{\partial\Gamma_k}{\id - f} \ne 0$.
  In the first case, a fixed point on the boundary $\partial\Gamma_k$ 
  is signalled during the computation of $\cind{\partial\Gamma_k}{\id - f}$
  and leads to a one-dimensional search problem. 
  In the second case, we continue the two-dimensional approximation.
\end{remark}

\begin{remark}
  Tarski's theorem says that all real closed fields share the same elementary theory (\sref{sub:ElementaryTheory}).
  This implies that the statement of Brouwer's fixed point theorem, for polynomial maps,
  extends from the real numbers $\RR$ to every real closed field $\R$: 
  as formulated above it is a first-order assertion in each degree.  
  It is remarkable that there exists a first-order proof over $\R$ 
  that is as direct as the usual second-order proof over $\RR$.
\end{remark}

\begin{remark}
  Over the field $\RR$ of real numbers the algebraic version implies the continuous version:
  Since $\Gamma \subset \RR^2$ is compact, the Stone-Weierstrass theorem 
  ensures that every continuous function $f \colon \Gamma \to \Gamma$ can be approximated 
  by polynomials $g_n \colon \Gamma \to \RR^2$, where $n=1,2,3,\dots$, such that $\abs{g_n-f} \le \frac{1}{n}$.
  The polynomials $f_n = \frac{n}{n+1} g_n$ satisfy $f_n(\Gamma) \subset \Gamma$ and $\abs{f_n-f} \le \frac{2}{n}$.
  For each $n$ there exists $z_n \in \Gamma$ such that $f_n(z_n) = z_n$ according to Theorem \ref{thm:Brouwer}.
  Again by compactness of $\Gamma$ we can extract a convergent subsequence.  Assuming $z_n \to z$, we find 
  \[ \abs{f(z)-z} \le \abs{f(z)-f(z_n)} + \abs{f(z_n)-f_n(z_n)} + \abs{z_n-z}\to 0 ,\] which proves $f(z) = z$.
\end{remark}

\end{Appendix}
  

\section*{Acknowledgments}

Many colleagues had the kindness to comment on successive versions of this article 
and to share their expertise on different facets of this venerable topic.
It is my heartfelt pleasure to thank Roland Bacher, Theo de Jong, Christoph Lamm, 
Bernard Parisse, Cody Roux, Marie-Fran\c{c}oise Roy, Siegfried Rump, Francis Sergeraert, and Duco van Straten. 
Numerous suggestions of referees and editors greatly helped to improve the exposition.


\bibliographystyle{monthly}
\bibliography{roots}

\providecommand{\bysame}{\leavevmode\hbox to3em{\hrulefill}\thinspace}
\providecommand{\MR}{\relax\ifhmode\unskip\space\fi MR }
\providecommand{\MRhref}[2]{%
  \href{http://www.ams.org/mathscinet-getitem?mr=#1}{#2}
}
\providecommand{\href}[2]{#2}
\begin{thebibliography}{10}

\bibitem{Artin:1927}
E.~Artin, {\"U}ber die {Z}erlegung definiter {F}unktionen in {Q}uadrate,
  \emph{Abh. Math. Sem. Univ. Hamburg} \textbf{5} (1926) 100--115, Collected
  Papers \cite{Artin:Papers}, pp. 273--288.

\bibitem{Artin:Papers}
\bysame, \emph{Collected {P}apers}, Edited by S. Lang and J. T. Tate,
  Springer-Verlag, New York, 1982, Reprint of the 1965 original.

\bibitem{ArtinSchreier:1926}
E.~Artin and O.~Schreier, Algebraische {K}onstruktion reeller {K}{\"o}rper,
  \emph{Abh. Math. Sem. Univ. Hamburg} \textbf{5} (1926) 85--99, Collected
  Papers \cite{Artin:Papers}, pp. 258--272.

\bibitem{ArtinSchreier:1927}
\bysame, Eine {K}ennzeichnung der reell abgeschlossenen {K}{\"o}rper,
  \emph{Abh. Math. Sem. Univ. Hamburg} \textbf{5} (1927) 225--231, Collected
  Papers \cite{Artin:Papers}, pp. 289--295.

\bibitem{BasuPollakRoy:2006}
S.~Basu, R.~Pollack, and M.-F. Roy, \emph{Algorithms in real algebraic
  geometry}, second ed., Springer-Verlag, Berlin, 2006, Available at
  \link{perso.univ-rennes1.fr/marie-francoise.roy}.

\bibitem{SmaleEtAl:1998}
L.~Blum, F.~Cucker, M.~Shub, and S.~Smale, \emph{Complexity and real
  computation}, Springer-Verlag, New York, 1998.

\bibitem{BochnakCosteRoy:1998}
J.~Bochnak, M.~Coste, and M.-F. Roy, \emph{Real algebraic geometry},
  Springer-Verlag, Berlin, 1998.

\bibitem{Cauchy:1831}
A.-L. Cauchy, Sur les rapports qui existent entre le calcul des r{\'e}sidus et
  le calcul des limites, \emph{Bulletin des Sciences de F{\'e}russac}
  \textbf{16} (1831) 116--128, {\OE}uvres \cite{Cauchy:Oeuvres}, S{\'e}rie 2,
  tome 2, pp. 169--183.

\bibitem{Cauchy:1837}
\bysame, Calcul des indices des fonctions, \emph{Journal de l'{\'E}cole
  Polytechnique} \textbf{15} (1837) 176--229, {\OE}uvres \cite{Cauchy:Oeuvres},
  S{\'e}rie 2, tome 1, pp. 416--466.

\bibitem{Cauchy:Oeuvres}
\bysame, \emph{{\OE}uvres compl{\`e}tes}, Gauthier-Villars, Paris, 1882--1974,
  Available at \link{mathdoc.emath.fr/OEUVRES/}.

\bibitem{Cohn:2003}
P.~M. Cohn, \emph{Basic algebra}, Springer-Verlag, London, 2003.

\bibitem{ConstructiveAspects:1969}
B.~Dejon and P.~Henrici, eds., \emph{Constructive aspects of the fundamental
  theorem of algebra}, John Wiley \& Sons Inc., London, 1969.

\bibitem{Dieudonne:1978}
J.~Dieudonn\'e, \emph{Abr\'eg\'e d'histoire des math\'ematiques. 1700--1900.},
  Hermann, Paris, 1978.

\bibitem{EbbinghausEtAl:1991}
H.-D. Ebbinghaus, H.~Hermes, F.~Hirzebruch, M.~Koecher, K.~Mainzer,
  J.~Neukirch, A.~Prestel, and R.~Remmert, \emph{Numbers}, Graduate Texts in
  Mathematics, vol. 123, Springer-Verlag, New York, 1991.

\bibitem{FineRosenberger:1997}
B.~Fine and G.~Rosenberger, \emph{The fundamental theorem of algebra},
  Undergraduate Texts in Mathematics, Springer-Verlag, New York, 1997.

\bibitem{Fuller:1975}
A.~T. Fuller, ed., \emph{Stability of motion}, Taylor \& Francis, Ltd., London,
  1975, A collection of early scientific publications by E. J. Routh, W. K.
  Clifford, C. Sturm and M. B{\^o}cher.

\bibitem{GathenGerhard:2003}
J.~v.~z. Gathen and J.~Gerhard, \emph{Modern computer algebra}, second ed.,
  Cambridge University Press, Cambridge, 2003.

\bibitem{Gauss:Werke}
C.~F. Gau{\ss}, \emph{Werke. {B}and {I}--{XII}}, Georg Olms Verlag, Hildesheim,
  1973, Reprint of the 1863--1929 original, available at
  \link{resolver.sub.uni-goettingen.de/purl?PPN235993352}.

\bibitem{GerstenStallings:1988}
S.~M. Gersten and J.~R. Stallings, On {G}auss's first proof of the fundamental
  theorem of algebra, \emph{Proc. Amer. Math. Soc.} \textbf{103} (1988)
  331--332.

\bibitem{Gonthier:2008}
G.~Gonthier, The four colour theorem, \emph{Notices Amer. Math. Soc.}
  \textbf{55} (2008) 1382--1393, Available at
  \link{www.ams.org/notices/200811}.

\bibitem{Hales:2007}
T.~C. Hales, The {J}ordan curve theorem, formally and informally, \emph{Amer.
  Math. Monthly} \textbf{114} (2007) 882--894.

\bibitem{Henrici:1974}
P.~Henrici, \emph{Applied and computational complex analysis}, vol.~1, John
  Wiley \& Sons Inc., New York, 1974.

\bibitem{Hurwitz:1895}
A.~Hurwitz, Ueber die {B}edingungen, unter welchen eine {G}leichung nur
  {W}urzeln mit negativen reellen {T}heilen besitzt, \emph{Math. Ann.}
  \textbf{46} (1895) 273--284, Math. Werke \cite{Hurwitz:Werke}, Band 2, pp.
  533--545. Reprinted in \cite{Fuller:1975}.

\bibitem{Hurwitz:Werke}
\bysame, \emph{Mathematische {W}erke}, Birkh{\"a}user, Basel, 1962--1963.

\bibitem{Jacobson:1985:1989}
N.~Jacobson, \emph{Basic algebra {I}-{II}}, second ed., W. H. Freeman and
  Company, New York, 1985, 1989.

\bibitem{Kasner:1905}
E.~Kasner, The present problems of geometry, \emph{Bull. Amer. Math. Soc.}
  \textbf{11} (1905) 283--314.

\bibitem{Kronecker:1869}
L.~Kronecker, Ueber {S}ysteme von {F}unctionen mehrer {V}ariabeln,
  \emph{Monatsberichte Akademie Berlin} (1869) 159--193, 688--698, Werke
  \cite{Kronecker:Werke}, Band {I}, pp. 175--226.

\bibitem{Kronecker:Werke}
\bysame, \emph{{W}erke}, Chelsea Publishing Co., New York, 1968, Reprint of the
  1895--1930 original.

\bibitem{Lang:2002}
S.~Lang, \emph{Algebra}, third ed., Graduate Texts in Mathematics, vol. 211,
  Springer-Verlag, New York, 2002.

\bibitem{Lehmer:1961}
D.~H. Lehmer, A machine method for solving polynomial equations, \emph{J.
  Assoc. Comput. Mach.} \textbf{8} (1961) 151--162.

\bibitem{Lehmer:1969}
\bysame, Search procedures for polynomial equation solving, in
  \emph{Constructive aspects of the fundamental theorem of algebra
  \cite{ConstructiveAspects:1969}}, John Wiley \& Sons Inc., 1969, 193--208.

\bibitem{LickteigRoy:2001}
T.~Lickteig and M.-F. Roy, Sylvester-{H}abicht sequences and fast {C}auchy
  index computation, \emph{J. Symbolic Comput.} \textbf{31} (2001) 315--341.

\bibitem{Loria:1938}
G.~Loria, {C}harles {S}turm et son {\oe}uvre math{\'e}matique, \emph{Enseign.
  Math.} \textbf{37} (1938) 249--274.

\bibitem{Mahler:1964}
K.~Mahler, An inequality for the discriminant of a polynomial, \emph{Michigan
  Math. J.} \textbf{11} (1964) 257--262.

\bibitem{Marden:1966}
M.~Marden, \emph{Geometry of polynomials}, Second edition. Mathematical
  Surveys, No. 3, Amer. Math. Soc., Providence, R.I., 1966.

\bibitem{EncyklMathWiss}
W.~F. Meyer, ed., \emph{Encyklop{\"a}die der mathematischen {W}issenschaften},
  B. G. Teubner, Leipzig, 1898.

\bibitem{EncyclSciMath}
J.~Molk, ed., \emph{Encyclop{\'e}die des {S}ciences {M}ath{\'e}matiques},
  Gauthier-Villars, Paris, 1907.

\bibitem{Ostrowski:1920}
A.~Ostrowski, \emph{{\"U}ber den ersten und vierten {G}ausschen {B}eweis des
  {F}undamentalsatzes der {A}lgebra}, vol. X.2, ch.~3 in \cite{Gauss:Werke},
  Georg Olms Verlag, 1920, Collected Papers \cite{Ostrowski:Works}, vol. 1, pp.
  538--553.

\bibitem{Ostrowski:1933}
\bysame, {\"U}ber {N}ullstellen stetiger {F}unktionen zweier {V}ariablen,
  \emph{J. Reine Angew. Math.} \textbf{170} (1933) 83--94, Collected Papers
  \cite{Ostrowski:Works}, vol. 3, pp. 269--280.

\bibitem{Ostrowski:Works}
\bysame, \emph{Collected {M}athematical {P}apers}, Birkh{\"a}user, Basel, 1983.

\bibitem{Pan:1998}
V.~Y. Pan, Solving polynomials with computers, \emph{American Scientist}
  \textbf{86} (1998) 62--69.

\bibitem{RahmanSchmeisser:2002}
Q.~I. Rahman and G.~Schmeisser, \emph{Analytic theory of polynomials}, London
  Mathematical Society Monographs. New Series, vol.~26, Oxford University
  Press, Oxford, 2002.

\bibitem{Remmert:1991}
R.~Remmert, \emph{The Fundamental Theorem of Algebra}, ch.~4 in
  \cite{EbbinghausEtAl:1991}, Springer-Verlag, New York, 1991.

\bibitem{Routh:1878}
E.~J. Routh, \emph{A treatise on the stability of a given state of motion},
  Macmillan, London, 1878, Reprinted in \cite{Fuller:1975}, pp. 19--138.

\bibitem{Roy:1996}
M.-F. Roy, Basic algorithms in real algebraic geometry and their complexity:
  from {S}turm's theorem to the existential theory of reals, in \emph{Lectures
  in real geometry (Madrid, 1994)}, de Gruyter Exp. Math., vol.~23, de Gruyter,
  Berlin, 1996, 1--67.

\bibitem{RoySzpirglas:1990}
M.-F. Roy and A.~Szpirglas, Complexity of computation on real algebraic
  numbers, \emph{J. Symbolic Comput.} \textbf{10} (1990) 39--51.

\bibitem{Rudin:1976}
W.~Rudin, \emph{Principles of mathematical analysis}, third ed., McGraw-Hill
  Book Co., New York, 1976.

\bibitem{Rump:2003}
S.~M. Rump, Ten methods to bound multiple roots of polynomials, \emph{J.
  Comput. Appl. Math.} \textbf{156} (2003) 403--432.

\bibitem{Schoenhage:1982}
A.~Sch{\"o}nhage, The fundamental theorem of algebra in terms of computational
  complexity, Preliminary report, Math. Inst. Univ. T{\"u}bingen, T{\"u}bingen,
  Germany, 1982, 49 pages, available at
  \link{www.informatik.uni-bonn.de/~schoe/fdthmrep.ps.gz}.

\bibitem{Schoenhage:1986}
\bysame, Equation solving in terms of computational complexity, in \emph{Proc.
  Int. Congress of Math., Berkeley, 1986}, vol.~1, Amer. Math. Soc., 1987,
  131--153.

\bibitem{Serret:1866}
J.~A. Serret, \emph{Cours d'alg{\`e}bre sup{\'e}rieure}, 3rd ed.,
  Gauthier-Villars, Paris, 1866, Available at
  \link{www.archive.org/details/coursdalgbresup06serrgoog}. (4th ed.\ 1877, 5th
  ed.\ 1885.).

\bibitem{Sinaceur:1991}
H.~Sinaceur, \emph{Corps et {M}od{\`e}les}, Librairie Philosophique J. Vrin,
  Paris, 1991, Translated as \cite{Sinaceur:2008}.

\bibitem{Sinaceur:2008}
\bysame, \emph{Fields and {M}odels}, Birkh{\"a}user, Basel, 2008.

\bibitem{Sturm:1829}
C.-F. Sturm, M{\'e}moire sur la r{\'e}solution des {\'e}quations
  num{\'e}riques, \emph{Bulletin des Sciences de F{\'e}russac} \textbf{11}
  (1829) 419--422, Collected Works \cite{Sturm:Works}, pp. 323--326.

\bibitem{Sturm:1835}
\bysame, M{\'e}moire sur la r{\'e}solution des {\'e}quations num{\'e}riques,
  \emph{Acad{\'e}mie Royale des Sciences de l'Institut de France} \textbf{6}
  (1835) 271--318, Collected Works \cite{Sturm:Works}, pp. 345--390.

\bibitem{Sturm:1836}
\bysame, Autres d{\'e}monstrations du m{\^e}me th{\'e}or{\`e}me, \emph{J. Math.
  Pures Appl.} \textbf{1} (1836) 290--308, Collected Works \cite{Sturm:Works},
  pp. 486--504, English translation in \cite{Fuller:1975}, pp. 189--207.

\bibitem{Sturm:Works}
\bysame, \emph{Collected {W}orks}, Edited by J.-C. Pont, Birkh{\"a}user, Basel,
  2009, Some of the articles are also available at
  \link{www-mathdoc.ujf-grenoble.fr/pole-bnf/Sturm.html}.

\bibitem{SturmLiouville:1836}
C.-F. Sturm and J.~Liouville, D{\'e}monstration d'un th{\'e}or{\`e}me de {M}.
  {C}auchy, relatif aux racines imaginaires des {\'e}quations, \emph{J. Math.
  Pures Appl.} \textbf{1} (1836) 278--289, Collected Works \cite{Sturm:Works},
  pp. 474--485.

\bibitem{Sylvester:1840}
J.~J. Sylvester, A method of determining by mere inspection the derivatives
  from two equations of any degree, \emph{Philosophical Magazine} \textbf{16}
  (1840) 132--135, Collected Papers \cite{Sylvester:Works}, vol. I, pp. 54--57.

\bibitem{Sylvester:Works}
\bysame, \emph{Collected {M}athematical {P}apers}, Cambridge University Press,
  Cambridge, 1904--1912.

\bibitem{Turing:1936}
A.~M. Turing, On computable numbers, with an application to the
  {E}ntscheidungsproblem, \emph{Proc. Lond. Math. Soc., II. Ser.} \textbf{42}
  (1936) 230--265, Collected Works \cite{Turing:Works}, vol. IV, pp. 18--56.

\bibitem{Turing:Works}
\bysame, \emph{Collected {W}orks}, North-Holland Publishing Co., Amsterdam,
  1992.

\bibitem{Ullherr:1846}
J.~C. Ullherr, Zwei {B}eweise f{\"u}r die {E}xistenz der {W}urzeln der
  h{\"o}hern algebraischen {G}leichungen, \emph{J. Reine Angew. Math.}
  \textbf{31} (1846) 231--234.

\bibitem{vanDerWaerden:1985}
B.~L. van~der Waerden, \emph{A history of algebra}, Springer-Verlag, Berlin,
  1985.

\bibitem{Weber:1898}
H.~Weber, \emph{Lehrbuch der {A}lgebra}, second ed., vol.~1, F. Vieweg \& Sohn,
  Braunschweig, 1898, Reprint: Chelsea Pub Co, New York, 3rd edition, January
  2000.

\bibitem{Weyl:1921}
H.~Weyl, {\"U}ber die neue {G}rundlagenkrise der {M}athematik. ({V}ortr\"age
  gehalten im mathematischen {K}olloquium {Z}\"urich), \emph{Math. Z.}
  \textbf{10} (1921) 39--79, Ges. Abh. \cite{Weyl:Werke}, Band II, pp.
  143--180.

\bibitem{Weyl:1924}
\bysame, Randbemerkungen zu {H}auptproblemen der {M}athematik, {II}.
  {F}undamentalsatz der {A}lgebra und {G}rundlagen der {M}athematik,
  \emph{Math. Z.} \textbf{20} (1924) 131--150, Ges. Abh. \cite{Weyl:Werke},
  Band II, pp. 433--453.

\bibitem{Weyl:Werke}
\bysame, \emph{Gesammelte {A}bhandlungen}, Springer-Verlag, Berlin, 1968.

\bibitem{Wiedijk:Top100}
F.~Wiedijk, Formalizing 100 theorems, \link{www.cs.ru.nl/~freek/100}, accessed
  25/03/2012.

\bibitem{Wilf:1978}
H.~S. Wilf, A global bisection algorithm for computing the zeros of polynomials
  in the complex plane, \emph{J. Assoc. Comput. Mach.} \textbf{25} (1978)
  415--420.

\bibitem{Wilkinson:1959}
J.~H. Wilkinson, The evaluation of the zeros of ill-conditioned polynomials,
  \emph{Numer. Math.} \textbf{1} (1959) 150--180.

\end{thebibliography}

\end{document}